\documentclass{amsart}
\usepackage[margin=1.5in]{geometry}

\allowdisplaybreaks[1]

\usepackage{breakurl}

\usepackage{color}
\usepackage{tikz-cd}
\usetikzlibrary{arrows}
\usepackage{amsmath,amssymb}
\usepackage{graphicx}
\usepackage{mathrsfs}
\usepackage{mathabx}
\usepackage[all]{xy}
\usepackage{adjustbox}
\usepackage[normalem]{ulem}
\usepackage{spectralsequences}
\usepackage{tcolorbox}
\usepackage{xfrac}
\usepackage{soul}
\usepackage{varwidth}

\usepackage{hyperref}
\usepackage[capitalise, noabbrev]{cleveref}
\usepackage[T1]{fontenc}
\usepackage[utf8]{inputenc}
\usepackage{caption}
\usepackage{subcaption}
\usepackage{multicol}

\newtheorem{theorem}{Theorem}[section]
\newtheorem{cor}[theorem]{Corollary}
\newtheorem{lemma}[theorem]{Lemma}
\newtheorem{proposition}[theorem]{Proposition}

\newtheorem{remark}[theorem]{Remark}
\newtheorem{definition}[theorem]{Definition}
\newtheorem{example}[theorem]{Example}

\newtheorem*{thm*}{Theorem}
\newtheorem*{cor*}{Corollary}
\newtheorem*{lem*}{Lemma}
\newtheorem*{prop*}{Proposition}

\theoremstyle{plain}

\numberwithin{theorem}{section}
\numberwithin{equation}{section}
\numberwithin{figure}{section}
\DeclareMathOperator{\Ext}{Ext}

\DeclareMathOperator{\wt}{wt}

\newcommand{\mC}{{\mathbb C}}

\newcommand{\mF}{{\mathbb F}}

\newcommand{\mM}{{\mathbb M}}

\newcommand{\mR}{{\mathbb R}}
\newcommand{\mS}{{\mathbb S}}

\newcommand{\cA}{{\mathcal A}}

\newcommand{\cE}{{\mathcal E}}

\definecolor{green'}{HTML}{009E73}
\definecolor{pink'}{HTML}{CC79A7}
\definecolor{orange'}{HTML}{D55E00}
\definecolor{darkblue'}{HTML}{0072B2}
\definecolor{yellow'}{HTML}{F0E442}
\definecolor{lightblue'}{HTML}{56B4E9}
\definecolor{gold'}{HTML}{E69F00}

\usepackage[textwidth=25mm, textsize=tiny]{todonotes}

\usepackage[style=alphabetic]{biblatex}
\usepackage{float}
\title{Splittings of truncated motivic Brown--Peterson cooperations algebras}
\author{Jackson Morris, Sarah Petersen, Elizabeth Tatum}

\begin{document}

\begin{abstract}
    We construct spectrum-level splittings of $BPGL \langle 1 \rangle \wedge BPGL \langle 1 \rangle$ at all primes $p$, where $BPGL \langle 1 \rangle$ is the first truncated motivic Brown--Peterson spectrum. Classically, $BP\langle 1 \rangle \wedge BP\langle 1 \rangle$ was first described by Kane and Mahowald in terms of Brown-Gitler spectra. This splitting was subsequently reinterpreted by Lellman and Davis-Gitler-Mahowald in terms of Adams covers. In this paper, we give motivic lifts of these splittings in terms of Adams covers, over the base fields $\mC, \, \mR,$ and $\mF_q$, where $\textup{char}(\mathbb{F}_q) \neq p$. As an application, we compute the $E_1$-page of the $BPGL\langle 1 \rangle$-based Adams spectral sequence as a module over $BPGL\langle 1 \rangle$, both in homotopy and in terms of motivic spectra. We also record analogous splittings for $BPGL \langle 0 \rangle \wedge BPGL \langle 0 \rangle$.
\end{abstract}

\maketitle
\tableofcontents

\section{Introduction} \label{sec:Intro}
\subsection{Motivation} 
The computation of the stable homotopy groups of spheres $\pi_* \mS$ is a central problem within stable homotopy theory and has motivated many significant developments in the field \cite{Adams66, DevHopSmi88,IsaWanXu23}. The primary tool for computing $\pi_*\mathbb{S}$ is the Adams spectral sequence, which converts homological algebra related to a homology theory into homotopy groups. Roughly, given a homology theory $E$, one begins by taking a sort of injective resolution of the sphere spectrum $\mathbb{S}$ relative to the homology theory $E$. Then, by applying totalization and homotopy groups, one gets the $E$-based Adams spectral sequence \cite{Rav86}.

Perhaps the most well-known Adams spectral sequence comes from taking ordinary mod $p$-homology, which is represented by the Eilenberg--MacLane spectrum $H\mathbb{F}_p$. This results in the mod-$p$ Adams spectral sequence, which takes the form
\[E_2 = \text{Ext}^{s,f}_{\mathcal{A}_p^\vee}(\mathbb{F}_p, \mathbb{F}_p) \implies \pi_s\mathbb{S}_p.\]
Here $\mathcal{A}_p^\vee = \pi_*(H\mathbb{F}_p \wedge H\mathbb{F}_p)$ is the $p$-primary dual Steenrod algebra, also known as the homotopy ring of cooperations for the homology theory $H\mathbb{F}_p$, and $\Ext$ is taken in the category of comodules over $\mathcal{A}^\vee_p$. This spectral sequence is particularly well suited for studying homotopy groups from a large-scale computation perspective \cite{linWanXu25}. 

Another perspective on understanding $\pi_*\mathbb{S}$ was first observed by Adams \cite{Adams66}. Consider the mod-$p$ Moore spectrum $\mathbb{S}/p$, defined as the cofiber of $\mathbb{S} \xrightarrow{\times p}\mathbb{S}$. This spectrum has a self map of the form $\Sigma^d\mathbb{S}/p \xrightarrow{v_1} \mathbb{S}/p$ which can then be used to construct what is known as the $\alpha$-family, a periodic family of nontrivial elements $\{\alpha_n\} \subseteq \pi_*\mathbb{S}$:
\[\alpha_n: \Sigma^{dn}\mathbb{S} \to\Sigma^{dn}\mathbb{S}/p \xrightarrow{v_1}\Sigma^{d(n-1)}\mathbb{S}/p \to \cdots \to \mathbb{S}/p \to \Sigma\mathbb{S}.\]
Chromatic homotopy organizes all of the elements of $\pi_*\mathbb{S}$ into $v_n$-periodic families inspired by the one above \cite{MilRavWil77}. 

One may attempt to access the $v_n$-periodic families by means of an Adams spectral sequence. However, the mod-$p$ Adams spectral sequence is not well suited for this task, so one must resort to other homology theories. For example, $v_1$-periodicity, which is closely related to the $\alpha$-family, has been studied by using the $bo$-Adams spectral sequence \cite{Mahowald81,LelMah87,BBBCX}, where $bo$ is the connective cover of the real topological K-theory spectrum $KO$, and the $BP\langle 1 \rangle$-Adams spectral sequence \cite{Gonzalez00}, where $BP \langle n \rangle$ is the $n^{th}$ truncated Brown-Peterson spectrum. For the purposes of this introduction, we will focus on the $BP\langle n \rangle$-Adams spectral sequence.

The trade-off for more immediate access to $v_n$-periodicity is that these spectral sequences are less immediately computable. The spectra $BP\langle n \rangle$ are not flat in the sense of Adams \cite{Adams74}, and so to study the $BP\langle n \rangle$-Adams spectral sequences, one must begin by analyzing their $E_1$-pages. This takes the form
\[E_1 = \pi_{s+f}(BP\langle n \rangle \wedge \overline{BP \langle n \rangle}^{\wedge f}) \implies \pi_s\mathbb{S},\]
where $\overline{BP \langle n \rangle}$ is the cofiber of the unit map $\mathbb{S} \to BP \langle n \rangle$. Key to working with this spectral sequence is an understanding of the homotopy ring of cooperations  $\pi_{*}(BP\langle n \rangle \wedge BP \langle n \rangle)$. 

For $n=0$, a model for $BP\langle 0 \rangle$ is $H\mathbb{Z}_{(p)}$. The homotopy ring of cooperations \newline $\pi_*(BP \langle 0 \rangle \wedge BP \langle 0 \rangle)$
was first studied by Cartan \cite{Cartan1954-1955}, and it was shown by Kochmann \cite{Kochman82} that there is an isomorphism up to $p$-completion
\[\pi_*(BP \langle 0 \rangle \wedge BP \langle 0 \rangle ) \cong \mathbb{Z}_p \oplus W,\] 
where $W$ is a graded $\mathbb{F}_p$-vector space. It is well-known that this splitting can also be realized on the level of spectra. 

For $n=1$, a model for $BP \langle 1 \rangle$ is $\ell_{(p)}$, the $p$-local connective Adams summand. The cooperations algebra $\pi_*(BP \langle 1 \rangle \wedge BP \langle 1 \rangle )$ was first described in terms of Brown-Gitler spectra, with Kane and Mahowald \cite{kane1981operations, Mahowald81} showing that up to $p$-completion, 
\[BP \langle 1 \rangle \wedge BP \langle 1 \rangle \simeq \bigvee_{k \geq 0} \Sigma^{2(p-1)k}BP \langle 1 \rangle \wedge H\mathbb{Z}_k\]
where the spectra $H\mathbb{Z}_k$ are the integral Brown--Gitler spectra \cite{goerss1986some}. Lellman and Davis-Gitler-Mahowald subsequently reinterpreted this splitting in terms of Adams covers \cite{lellmann1984operations, davis1981stable}, showing that up to $p$-completion,
\begin{equation}\label{intro classical splitting}
BP \langle 1 \rangle \wedge BP \langle 1 \rangle \simeq \bigvee_{k \geq 0}\Sigma^{2(p-1)}BP \langle 1 \rangle^{\langle\nu_p(k!)\rangle} \vee V,\end{equation}
where $BP \langle 1 \rangle ^{\langle n \rangle}$ denotes the $n^{th}$ Adams cover for $BP \langle 1 \rangle$ in a minimal $H\mathbb{F}_p$-Adams resolution and $V$ is a sum of suspensions of mod-$p$-Eilenberg-Maclane spectra. This spectrum-level splitting of the cooperations algebra, combined with the homotopy ring of cooperations, makes computing with the $BP\langle 1\rangle$-Adams spectral sequence much more tractable.

In the last 15 years, much of the advancement of the computation of $\pi_*\mathbb{S}$ has come through the lens of motivic homotopy theory, which is a homotopy theory for smooth schemes over a base scheme $S$ \cite{VoeA1, Mor12}. Originally created to solve algebro-geometric problems using homotopical techniques, it has also been extremely valuable as a way to study the classical stable homotopy category (see for instance \cite{Pst23}). In particular, much has been gained out of the study of motivic homotopy theory over the affine schemes $S=\text{Spec}(F)$ for $F$ a field \cite{IsaWanXu23,BacBurXu25}. Many of these advancements have come from a large-scale computation perspective by computing with the motivic $H\mathbb{F}_p$-Adams spectral sequence, particularly at the prime $p=2$.

The classical $v_n$-periodic layers of the stable homotopy category lift via the algebraic cobordism spectrum $MGL$, a motivic analogue of the complex cobordism spectrum $MU$ \cite{LM07}. This enables one to study $v_n$-periodicity in the context of motivic homotopy theory. For $n=1$, there has been some work in this direction. In \cite{CQ21}, Culver--Quigley use an analogue of the $bo$-Adams spectral sequence to compute the $\mathbb{C}$-motivic $v_1$-periodic stable stems. In \cite{belmontisaksenkong-v1R} (resp. \cite{kongquigley}), Belmont--Isaksen--Kong (resp. Kong--Quigley) use different techniques to compute data related to the $\mathbb{R}$-motivic (resp. $\mathbb{F}_q$ and $\mathbb{Q}_p$-motivic) $v_1$-periodic stable stems. In \cite{Realkqcoop,finitekqcoop}, the first author computes the homotopy ring of cooperations for a motivic analogue of $bo$ over $\mathbb{R}$ and $\mathbb{F}_q$ where $\text{char}(\mathbb{F}_q) \neq 2$. All of these computations are performed at the prime 2, and none of these results produce spectrum-level splittings. A principal obstruction to producing spectrum-level splittings is that, at the present, there is no construction of motivic Brown--Gitler spectra.

Related to the algebraic cobordism spectrum are the truncated motivic Brown-Peterson spectra $BPGL \langle n\rangle$, which are motivic analogues of the classical spectra $BP \langle n \rangle$. In this paper, we initiate the study of the $BPGL\langle n \rangle$-motivic Adams spectral sequence. We begin by computing the homotopy ring of cooperations in the cases where $n=0$ and $n=1$ at all primes and over a variety of base fields. In fact, we are able to give spectrum-level splittings of the cooperations algebra.

\subsection{Main results}
The main result of this paper is a lift of the splitting of \cref{intro classical splitting} to motivic spectra.
\begin{theorem}[\Cref{thm:BPGL1SplittingAlpha}]
    Let $F \in \{\mathbb{C}, \mathbb{R}, \mathbb{F}_q\}$ and let $p$ be any prime, where $\textup{char}(\mathbb{F}_q) \neq p$. There is a splitting of $p$-complete motivic spectra
    \[
    BPGL \langle 1 \rangle \wedge BPGL \langle 1 \rangle \simeq \bigvee_{k \geq 0} \Sigma^{2k(p-1), k(p-1)} BPGL\langle 1 \rangle^{\langle \nu_p (k!) \rangle} \vee V
    \]
    where $BPGL \langle 1 \rangle ^{\langle \nu_{p}(k!) \rangle}$ is the $\nu_{p}(k!)^{th}$-Adams cover in a minimal $H\mF_{p}$-resolution of $BPGL\langle 1 \rangle$, and $V$ is a wedge of suspensions of $H\mF_{p}$.
\end{theorem}

A similar splitting has been shown in the $C_2$-equivariant category at the prime $p=2$ for $ku_\mathbb{R} \wedge ku_\mathbb{R}$ by the second and third authors in \cite[Corollary 6.15]{LiPetTat25}. They use their results to prove the $F=\mathbb{R}$ and $p=2$ case of our theorem \cite[Corollary 6.18]{LPTSplittingkuR}.

Our proof requires the development of several new techniques in motivic homtopy theory. In particular, we introduce the relative Adams spectral sequence in motivic homotopy theory, give a Whitehead theorem for motivic Margolis homology, and compute the $\mathcal{E}(1)^\vee_p$-comodule structure of motivic integral Brown--Gitler comodules.

We deduce these splittings in the absence of motivic Brown--Gitler spectra, for which no construction is currently known. Note that at the prime $2$, $BPGL \langle 1 \rangle$ is the effective (2-local) algebraic $K$-theory spectrum $kgl$, and at odd primes $BPGL \langle 1 \rangle$ is the connective ($p$-local) Adams summand $m\ell$ \cite{NauSpiOst15}. 

In addition to this spectrum-level splitting, we compute the homotopy ring of cooperations.

\begin{theorem}[\Cref{thm:BPGL1homotopyCooperations}]
    Let $F \in \{\mathbb{C}, \mathbb{R}, \mathbb{F}_q\}$ and let $p$ be any prime, where $\textup{char}(\mathbb{F}_q) \neq p$. Then there is an isomorphism:
    \begin{align*}
        \pi^F_{*,*}(BPGL \langle 1 \rangle & \wedge BPGL \langle 1 \rangle) \\ 
        & \cong \bigoplus_{k \geq 0} \Big( (\pi^F_{*,*}BPGL \langle 0 \rangle)\{x_0, \dots, x_{\nu_p(k!)-1}\}
    \oplus (\pi^F_{*,*}BPGL \langle 1 \rangle) \{x_{\nu_p(k!)}\} \oplus W_k \Big),
    \end{align*}
    with relations $v_1x_{i-1} = 2x_i$, where $|x_i| = (2i(p-1),  i(p-1))$
    and $W_k$ is a sum of suspensions of $\pi^F_{*,*}H\mathbb{F}_p$.    
\end{theorem}

Another immediate corollary of our spectrum-level splitting is a description of the operations algebra $BPGL\langle 1 \rangle^* BPGL\langle 1 \rangle$, given in \Cref{thm: operations}.

Just as at the spectrum level, a similar computation has been shown in the $C_2$-equivariant category at the prime $p=2$ for $\pi_{*,*}^{C_2}(ku_\mathbb{R} \wedge ku_\mathbb{R})$ by the second and third authors in \cite{LiPetTat25}. They use their results to prove the $F=\mathbb{R}$ and $p=2$ case of our theorem \cite[Corollary 6.17]{LiPetTat25}.

Using our results, we are able to completely identify the $E_1$-page of the $BPGL\langle 1\rangle$-based motivic Adams spectral sequence.

\begin{theorem}[\Cref{cor:n-line}]
    Let $F \in \{\mathbb{C}, \mathbb{R}, \mathbb{F}_q\}$ and let $p$ be any prime where $\textup{char}(\mathbb{F}_q) \neq p$. Then the $n$-line of the $BPGL \langle 1 \rangle$-motivic Adams spectral sequence can be described as
    \[
    E_1^{*,n,*} \cong \bigoplus_{I \in \mathcal{I}_n} \Big( (\pi_{*,*}^FBPGL \langle 0 )\rangle\{x_0, \dots, x_{m-1}\} \\
    \oplus (\pi_{*,*}^FBPGL \langle 1 \rangle)\{x_m\} \oplus W_I \Big),
    \]
    where $\mathcal{I}_n = \{I=(k_1, \dots, k_n): k_j \geq 1 \text{ for all } 1 \leq j \leq n\}$, $m = \nu_p(k_1!) + \cdots \nu_p(k_n!)$ for $ (k_1, \dots, k_n)=I \in \mathcal{I}_n$, and $W_I$ is a sum of suspensions of $\pi_{*,*}^FH\mathbb{F}_p$.
\end{theorem}

This describes the $n$-line of the spectral sequence as a module over the $0$-line. 

We also record an analogous spectrum level splitting in the $BPGL\langle0\rangle$ case. 

\begin{theorem}[\Cref{thm:bpgl0coopalgebra}]
    Let $F \in \{\mathbb{C}, \mathbb{R}, \mathbb{F}_q\}$ and let $p$ be any prime, where $\textup{char}(\mathbb{F}_q) \neq p$. There is a splitting of $p$-complete motivic spectra
    \[BPGL \langle 0 \rangle \wedge BPGL \langle 0 \rangle \simeq BPGL \langle 0 \rangle \vee V\]
    where $V$ is a wedge of suspensions of $H\mathbb{F}_p$.
\end{theorem}

As in the $BPGL\langle 1 \rangle$-case, we also compute the homotopy ring of cooperations (\cref{thm:BPGL0homotopyringcoops}). We remark that a model for $BPGL\langle 0 \rangle$ is the ($p$-local) integral motivic cohomology spectrum $H\mathbb{Z}$, and so our results may be interpreted as a motivic lift of Kochman's results \cite{Kochman82}.

Finally, we show that the mod-$p$ homology of all truncated Brown--Peterson spectra splits as an $\mathcal{E}(n)^\vee_p$-comodule into a wedge of sums of Brown--Gitler comodules

\begin{theorem}
    For all $n \geq 0$, There is an $\mathcal{E}(n)^\vee_p$-comodule isomorphism
    \[H_{*,*}BPGL \langle n \rangle \cong \bigoplus_{k \geq 0}\Sigma^{2k(p-1), k(p-1)}B_{n-1}(k).\]
\end{theorem}

Our results in motivic homology allow for one to access heigher height invariants. In particular, combined with the K\"unneth isomorphism for the mod-p homology of truncated motivic Brown--Peterson spectra \cref{prop:BPGLnKunneth} and a change-of-rings isomorphism, this allows one to decompose the $E_2$-page of the motivic Adams spectral sequence computing $\pi_{*,*}^F(BPGL\langle n \rangle \wedge BPGL \langle n \rangle)$ for any $n \geq 0$.

\subsection{New directions}

\subsubsection{$BPGL\langle 1 \rangle$-based motivic Adams spectral sequence}
This paper begins the study of the $BPGL \langle 1 \rangle$-based Adams spectral sequence. We expect this spectral sequence to be useful for a myriad of purposes, some of which we outline below.

First, our thorough description of the $E_1$-page of the $BPGL\langle 1 \rangle$-based motivic Adams spectral sequence gives us much control on its $v_1$-periodic and $v_1$-torsion summands. Using techniques similar to those used in \cite{BBBCX} will allow for large scale computations of the motivic stable stems. The $\mathbb{C}$ and $\mathbb{R}$-motivic stable stems have been well studied at the prime 2 (see \cite{IsaWanXu23,BelIsa22}), but not as much work has been done at odd primes over any field. We expect this spectral sequence will make new contributions to odd primary motivic stable stem computations and extend \O stv\ae r--Wilson's computations of $\pi_{*,*}^{\mathbb{F}_q}\mathbb{S}$ over finite fields \cite{WO-finite} at the prime 2.

Second, as is true in the classical case \cite{Gonzalez00}, we expect that the $BPGL\langle 1 \rangle$-based motivic Adams spectral sequence has a vanishing line which allows one to identify $v_1$-periodic elements in $\pi_{*,*}^F\mathbb{S}$. At the prime 2, one may compare these elements with those detected by the $kq$-resolution. At odd primes, less is known about $v_1$-periodicity in motivic homotopy theory. Similar to the work of Bachmann--Hopkins at the prime 2 \cite{BH21}, one can show that there are motivic Adams operations $\psi^n$ on $BPGL \langle 1 \rangle$ for all $n \in \mathbb{Z}_p^\times$. Let $q$ be a topological generator for $\mathbb{Z}_p^\times$. Motivated by \cite{BruRog22}, define the odd primary image of J spectrum $j$ by the fiber sequence
\[j \to BPGL\langle 1 \rangle \xrightarrow{\psi^q-1}\Sigma^{2(p-1), p-1}BPGL\langle 1 \rangle.\]
The motivic homotopy groups of $j$ are closely related to the $v_1$-periodicicity in $\pi_{*,*}^F\mathbb{S}$. We expect that $\pi_{*,*}^Fj$ appears in the 0- and 1-line of the $E_\infty$-page of the $BPGL\langle 1 \rangle$-based Adams spectral sequence.

Finally, the $BPGL\langle 1 \rangle$-based motivic Adams spectral sequence will be a useful tool in an analysis of motivic analogues of the height 1 telescope conjecture. At each height, the telescope conjecture compares two different localizations of the stable homotopy category. Though the classical telescope conjecture is false at all primes and all heights $n \geq 2$ \cite{BHLS23}, at the height 1 it was shown to be true by comparing $\pi_*(v_1^{-1}\mathbb{S}/p)$ and $\pi_{*}L_{KU/p}\mathbb{S}/p$, where $KU$ is the complex $K$-theory spectrum\cite{Mahowald81, Miller81}. In particular, Mahowald used the $ko$-based Adams spectral sequence to compute $\pi_*(v_1^{-1}\mathbb{S}/p)$ at the prime $p=2$. We expect that the $BPGL\langle 1\rangle$-Adams spectral sequence will play an analogous role at odd primes in the motivic setting, allowing us to compare $\pi_{*,*}^F(\mathbb{S}/p)$ and $\pi_{*,*}^FL_{KGL/p}\mathbb{S}/p$, as $BPGL\langle 1 \rangle$ is a summand of the $p$-local algebraic $K$-theory spectrum.

\subsubsection{Generalization to other base schemes}
The methods developed in this paper may be applied to compute the homotopy ring of cooperations $\pi_{*,*}^F(BPGL\langle 1 \rangle \wedge BPGL \langle 1 \rangle)$ and produce a splitting of the cooperations algebra $BPGL \langle 1 \rangle \wedge BPGL \langle 1 \rangle$ over any base field, so long as one has a computation of $\text{Ext}_{\mathcal{E}(1)^\vee_p}^{s,f,w}(\mathbb{M}_p^F, \mathbb{M}_p^F)$. A natural place to extend these computations, then, is over the $p$-adic rationals $\mathbb{Q}_p$ and the rationals $\mathbb{Q}$, as the aforementioned Ext computations were performed by Ormsby \cite{Ormsby11} and Ormsby--\O stv\ae r \cite{OrmOst13}. 
 
In an orthogonal direction, there is an interesting technique that allows one to compute motivic homotopy groups without any knowledge of homological algebra over the dual Steenrod algebra, i.e. circumventing the Adams spectral sequence. For any prime $\ell \neq p$ such that $\ell \not\equiv 1\, (p^2)$, and any motivic ring spectrum $E \in SH(\mathbb{Z}[1/p])$, work of Bachmann--\O stv\ae r \cite{BO22} produces a pullback of (ordinary, $p$-complete) spectra 
\[\begin{tikzcd}
	{E(\mathbb{Z}[1/p])} & {E(\mathbb{R})} \\
	{E(\mathbb{F}_\ell)} & {E(\mathbb{C})}
	\arrow[from=1-1, to=1-2]
	\arrow[from=1-1, to=2-1]
	\arrow["\lrcorner"{anchor=center, pos=0.125}, draw=none, from=1-1, to=2-2]
	\arrow[from=1-2, to=2-2]
	\arrow[from=2-1, to=2-2]
\end{tikzcd}\]
where $E(F) = \text{map}_{F}(\mathbb{S}, E)$ is the ordinary spectrum of maps from the $F$-motivic sphere to the pullback of $E$ along the canonical map $\mathbb{Z}[1/p] \to F$. Our results allow for one to use this square to  deduce the $E_1$-page of the $\mathbb{Z}[1/p]$-motivic $BPGL\langle 1 \rangle$-based Adams spectral sequence. This enables one to study $v_1$-periodicity over a more arithmetic class of base schemes. This is particularly enticing as there are currently no computations in homological algebra over the $\mathbb{Z}[1/p]$-motivic Steenrod algebra. It would also be interesting to use a variant of this square to generalize the spectrum-level splittings we produce.

We also anticipate that there is a technique to compute the homotopy ring of cooperations $\pi_{*,*}^F(BPGL\langle 1 \rangle \wedge BPGL \langle 1 \rangle )$ over a broad class of base fields of small cohomological dimension using slice theoretic techniques. These techniques would eschew the Adams spectral sequence from our analysis, and would instead use the effective slice spectral sequence in a manner similar to Ormsby-\O stv\ae r \cite{OrmOst14}.

\subsubsection{Higher heights}
One of the implicit takeaways of our results is that the $E_1$-page of the $BPGL\langle 1 \rangle$-based motivic Adams spectral sequence is only slightly more complicated than its classical analogue, with complications determined by the motivic cohomology of the base field with integral coefficients. We expect this thread of ideas to extend to higher chromatic heights. In particular, moving up to height 2, we expect that much of the prior work on the $BP \langle 2 \rangle$-based Adams spectral sequence (for example, see \cite{culver-bp2coop,CulverOddp20}) generalizes quite nicely to the $BPGL \langle 2 \rangle$-based motivic Adams spectral sequence. This spectral sequence would offer a glimpse into $v_2$-periodicity in the motivic stable stems. For $F=\mathbb{R}$ and $p=2$, we expect that these computations would also inform one about $v_2$-periodicity in the $C_2$-equivariant stable stems.

\subsubsection{Exotic Periodicity}
As another implication of our work, we see that generalized motivic Adams spectral sequences are highly computable tools which allow one to access $v_n$-periodicity. We describe now a related approach that captures richer structure in $\pi_{*,*}^\mathbb{C}\mathbb{S}$.

An interesting feature of motivic homotopy theory is the failure of Nishida's nilpotence theorem. Classically, Nishida's theorem informs us that any positive degree element of $\pi_*\mathbb{S}$ is nilpotent. Motivically, we see that this fails drastically. For instance, the Hopf map $\eta \in \pi_{1,1}\mathbb{S}$ is non-nilpotent \cite{MorelKMW}, leading to interesting $\eta=:w_0$-periodicity in the motivic stable stems. Now, let $F=\mathbb{C}$ and $p=2$. It was shown by Andrews \cite{Andrews18} that there is a self map of the cofiber of $\eta$ of the form
\[w_1^4:\Sigma^{20,12}\mathbb{S}/\eta \to \mathbb{S}/\eta.\]
Parallel to how the $v_1^4$-self map of $\mathbb{S}/2$ allows one to study $v_1$-periodicity, this $w_1$-self map allows one to study $w_1$-periodic phenomena in $\pi^\mathbb{C}_{*,*}\mathbb{S}$.

$\mathbb{C}$-motivic homotopy theory also benefits from an interpretation via $BP$-synthetic spectra \cite{Pst23}. There is an element $\tau \in \pi^{\mathbb{C}}_{0, -1}\mathbb{S}$ which acts as a deformation parameter between classical homotopy theory, cellular $\mathbb{C}$-motivic homotopy theory, and the category of derived even $BP_*BP$-comodules. This can be summarized by the following diagram. 
\[\begin{tikzcd}
	Sp & {SH(\mathbb{C})^{cell}} & {D^{even}(BP_*BP).}
	\arrow["{\cdot\tau^{-1}}"', from=1-2, to=1-1]
	\arrow["{-\wedge \mathbb{S}/\tau}", from=1-2, to=1-3]
\end{tikzcd}\]
This recasts the category of $\mathbb{S}/\tau$-modules as an algebraic category related to classical homotopy theory, as we may interpret any $BP_*BP$-comodule as a $\mathbb{C}$-motivic spectrum over $\mathbb{S}/\tau$. This is motivated by the close relationship with the classical Adams--Novikov spectral sequence shown by Isaksen \cite{Isaksen19}, in that there is an isomorphism
\[\pi_{s,w}^{\mathbb{C}}\mathbb{S}/\tau \cong \text{Ext}^{2w-s, 2w}_{BP_*BP}(BP_*, BP_*).\]

In \cite{Gheorghe18}, Gheorghe used this approach to construct $\mathbb{S}/\tau$-modules denoted $wBP\langle n \rangle$ whose homotopy groups are
\[\pi^{\mathbb{C}}_{*,*}wBP\langle n\rangle = \mathbb{F}_2[w_0, w_1, \dots, w_n].\]
We may treat these spectra in a similar way to how we use the spectra $BPGL \langle n \rangle$. There is a $\mathbb{S}/\tau$-linear $wBP\langle n \rangle$-based motivic Adams spectral sequence, completely internal to the category of $\mathbb{S}/\tau$-modules, which takes the form
\[E_1^{s,f,w} = \pi^{\mathbb{C}}_{s+f, w}(wBP\langle n \rangle \wedge_{\mathbb{S}/\tau} wBP\langle n \rangle^{\wedge_{\mathbb{S}/\tau}f}) \implies \pi^{\mathbb{C}}_{s,w}\mathbb{S}/\tau.\]
The analysis of this spectral sequence gives another way to compute the $E_2$-page of the classical Adams--Novikov spectral sequence. In particular, just as the $BPGL\langle n \rangle$-motivic Adams spectral sequence detects $v_1$-periodicity in the stable stems, we expect the 0- and 1-lines of the $wBP\langle n \rangle$-based $\mathbb{S}/\tau$-linear motivic Adams spectral sequence to detect $w_n$-periodicity in the cofiber of $\tau$, and hence in the Adams--Novikov. By pulling back along the quotient map $\mathbb{S} \to \mathbb{S}/\tau$, one can then lift these periodic classes to the $\mathbb{C}$-motivic stable stems. Moreover, we expect this spectral sequence to be simpler to compute with than the $BPGL \langle n \rangle$-based motivic Adams spectral sequence, as is evident from the simpler form that the homotopy groups of $wBP\langle n \rangle$ take.

In the case of $n=1$, the $w_1$-periodicity detected by the $\mathbb{S}/\tau$-linear $wBP\langle 1 \rangle$-based motivic Adams spectral sequence can be compared with the $w_1$-periodic families studied by Isaksen--Kong--Li--Ruan--Zhu \cite{IsaKonLiYuaZhu25}. For $n \geq 2$, this technique offers a first glimpse into higher exotic periodicity.

\subsection{Organization}
In \Cref{sec:motivic prelim}, we recall salient facts in motivic homotopy theory. We review the motivic spectra $BPGL \langle n \rangle$, discuss the dual Steenrod algebra, relative homology, and various Adams spectral sequences which feature in our arguments. In \Cref{sec:motivic homology}, we make several homological computations. First, we compute the $BPGL \langle 0 \rangle$- and $BPGL \langle 1 \rangle$-relative dual Steenrod algebras and deduce a K\"unneth isomorphism for the homology of truncated motivic Brown--Peterson spectra. Then, we interpret the $\mathcal{E}(n)^\vee_p$-comodule structure of $H_{*,*}BPGL \langle n \rangle$ in terms of Brown--Gitler comodules, prove a Whitehead theorem for Margolis homology over $\mathcal{E}(1)_p$, and introduce motivic lightning flash modules. In \Cref{sec:BPGL0}, we compute the homotopy ring of cooperations $\pi_{*,*}(BPGL \langle 0 \rangle \wedge BPGL \langle 0 \rangle)$ and construct a spectrum-level splitting of the cooperations algebra. In \Cref{sec:BPGL1}, we compute the homotopy ring of cooperations $\pi_{*,*}(BPGL \langle 1 \rangle \wedge BPGL \langle 1 \rangle)$ and construct a spectrum-level splitting of the cooperations algebra in terms of motivic Adams covers for $BPGL \langle 1 \rangle$. Then, we apply our results to compute the $E_1$-page of the $BPGL\langle 1 \rangle$-based motivic Adams spectral sequence.

\subsection{Notation}

\begin{enumerate}
    \item Let $H$ be the spectrum representing motivic cohomology with $\mF_p$ coefficients.
    \item Let $\mathbb{M}_p^F:= \pi_{*,*}^FH$. When clear from context, we will omit the superscript $F$.
    \item Let $\cA_p^\vee$ denote the motivic dual Steenrod algebra, and let $\cE(n)_p^\vee$ its associated subalgebras. Generally, the field $F$ will be clear from context.
    \item Let $A^\vee_{p}$ denote the topological dual Steenrod algebra, and let $E(n)^\vee_{p}$ denote its associated subalgebras.
    \item All $\Ext$ groups are graded via stem, Adams filtration, and motivic weight, denoted $(s,f,w)$.
    \item For $B$ any $\mathbb{M}_p$-algebra, the $\Ext$ group $\text{Ext}^{s,f,w}_B(M)$ is shorthand for \\   
    $\text{Ext}_B^{s,f,w}(\mathbb{M}_p, M)$.
    \item All charts are written in $(s,f)$ grading with motivic weight suppressed.
    \item Motivic lightning flash modules (\cref{def: lightning}) are denoted $L_p(k)$. Classical topological lightning flash modules are denoted $L^{cl}_p(k)$.
    \item Let $\doteq$ denote equivalence up to a unit in $\mathbb{M}_p$.
    \item Let $\textbf{mASS}_p^F(X)$ denote the $F$-motivic $H\mathbb{F}_p$-based motivic Adams spectral sequence for a spectrum $X$. When a statement is made agnostic of base field, we will omit the superscript $F$.
\end{enumerate}

\subsection*{Acknowledgments} The authors would like to thank Mark Behrens, Christian Carrick, Mike Hill, Guchuan Li, Kyle Ormsby, John Palmieri, J.D. Quigley, John Rognes, and Vesna Stojanoska for enlightening conversations. The first author would like to thank Lorelei for inspiration \cite{ThisIsLorelei}. The second and third authors would like to thank the Isaac Newton Institute for Mathematical Sciences, Cambridge, for support and hospitality during the programme Equivariant homotopy theory in context, where work on this paper was undertaken. This work was supported by EPSRC grant EP/Z000580/1. 

\section{Motivic preliminaries} 
\label{sec:motivic prelim}
In this section, we review stable motivic homotopy theory. We introduce the motivic spectra which feature in our work and discuss the complexities that arise in their homotopy groups as one varies the base field. This leads to different structure in the motivic homology of a point and the dual Steenrod algebra depending on choice of base field. Following this, we discuss the $BPGL \langle 1 \rangle$-based and $H\mathbb{F}_p$-based motivic Adams spectral sequences. To close this section, we introduce relative homology and the relative Adams spectral sequence in the motivic setting.

\subsection{Motivic spectra and the dual Steenrod algebra}
For $S$ a scheme, we let $SH(S)$ be Voevosdsky's motivic stable homotopy category \cite[Def. 5.7]{Voevosdsky98} \cite{Jardine00}. Recall that this is triangulated, and has a compatible closed symmetric monoidal structure given by the motivic sphere spectrum $\mS = \Sigma^\infty S_+$, the smash product pairing $- \wedge -$, the twist isomorphism $\gamma$ and the function spectrum $F(-,-)$. As we will be primarily be working over affine schemes of the form $S = \text{Spec}(F)$, we will instead use the notation $SH(F):= SH(\text{Spec}(F))$. We will refer to $F$ as the base field. 

Let $S^{p,q}=(S^1)^{\wedge p-q} \wedge (\mathbb{A}^1 - 0 )^{\wedge q}$, where $S^1$ denotes the constant simplicial circle and $(\mathbb{A}^1 - 0)$ is the image of the smooth $F$-scheme under the Yoneda embedding. For $x \in \pi_{p,q} (X)$, where $X$ is a motivic spectrum, we refer to $p$ as the topological or stem degree and $q$ as the weight of $x$. 

Let $H = H \mF_p$ be the motivic Eilenberg--MacLane spectrum representing motivic cohomology with coefficients in $\mF_p$. It is a commutative ring spectrum, with unit map $\eta: \mS \to H$ and product $\mu: H \wedge H \to H$. 


Let $\mathbb{M}_p: = \pi_{*,*} (H) = H^{-*,-*}$ denote the mod-$p$ motivic homology and cohomology groups of the base field $F$. When working over a particular ring $F$, we will denote this by $\mathbb{M}_p^F.$ Note that the cup product induced by $\mu$ gives $\mathbb{M}_p$ the structure of a bigraded commutative $\mF_p$-algebra. The calculation of $\mathbb{M}_p^F$ for $F$ a field is due to Voevosdky and is deeply related to Milnor-Witt $K$-theory \cite{Voemotiviccohomology}. We give a few examples below, citing references for our particular choice of notation.
\begin{example}
\rm For $F = \mathbb{C}$, we have $\mathbb{M}_p^\mathbb{C} = \mathbb{F}_p[\tau]$, where $|\tau|=(0, -1)$ for all primes \cite{Voereduced}.
\end{example}
\begin{example}
\rm For $F=\mathbb{R}$, we have $\mathbb{M}_2^\mathbb{R} = \mathbb{F}_2[\rho, \tau]$ for $|\rho| = (-1, -1)$ and $|\tau|=(0, -1)$ \cite{Voemotiviccohomology}, while for odd primes we have $\mathbb{M}_p^\mathbb{R} = \mathbb{F}_p[\theta]$, where $|\theta|=(0, -2)$ \cite{GreRog-Segal}.
\end{example}
\begin{example}
\rm For $F = \mathbb{F}_q$ a finite field, the structure of $\mathbb{M}_p^{\mathbb{F}_q}$ is more complicated. For $\textup{char}(\mathbb{F}_q) \neq 2$, let $\tau \in \mu_2(\mathbb{F}_q)$ be a second root of unity. Then, we have
\[\mathbb{M}_2^{\mathbb{F}_q} =  \left\{ \begin{array}{cl}
\mathbb{F}_2[u, \tau]/(u^2) &  q \equiv 1 \, (4)\\
\mathbb{F}_2[\rho, \tau]/(\rho^2) & q \equiv 3 \, (4),
\end{array} \right. \] 
where $|\tau| = (0, -1)$ and $|u|=|\rho| = (-1, -1)$. Here $\rho = [-1] \in  H^{1,1}(\mathbb{F}_q; \mathbb{Z}/2) \cong \mathbb{F}_q^\times/\mathbb{F}_q^{\times 2}$. Note that if $q \equiv 1 \, (4)$, then $\rho =0$.

For $\text{char}(\mathbb{F}_q) \neq p$, let $i$ be the smallest integer such that $p\,|\,q^i-1$, and let $\zeta \in \mu_p(\mathbb{F}_{q^i})$ be a primitve $p^{th}$ root of unity. Then, we have
\[\mathbb{M}_p^{\mathbb{F}_q} =  \left\{\begin{array}{cl}
\mathbb{F}_p[u,\zeta]/(u^2)  & q^i \equiv 1 \, (p^2) \\
\mathbb{F}_p[\gamma, \zeta]/(\gamma^2)  & q^i \not\equiv 1 \, (p^2),
\end{array} \right. \]
where $|\zeta| = (0, -i)$ and $|u| = |\gamma| = (-1, -i)$. Here $\gamma \in H^{1, i}(\mathbb{F}_q; \mathbb{Z}/p) \cong \mathbb{F}_{q^i}^\times/\mathbb{F}_{q^i}^{\times p}$ denotes the image of a primitive $p^{th}$ of unity. Note that if $q^i \equiv 1 \, (p^2)$, then $\gamma=0$ \cite{Wilson16}. Although there are abstract isomorphisms between $\mathbb{M}_p^{\mathbb{F}_q}$ for $q^i \equiv 1 \, (p^2)$ and $q^i \not \equiv 1 \, (p^2)$, we use different notation to distinguish their behavior over the dual Steenrod algebra, which we will investigate below.
\end{example}


Let $\cA_p = H^{*,*} H = \pi_{-*,-*} F(H,H)$ denote the motivic Steenrod algebra and let $\cA^\vee_p= \pi_{*,*}(H \wedge H)$ denote its dual. The dual Steenrod algebra $\mathcal{A}_p^\vee$ has the structure of a Hopf algebroid over $\mathbb{M}_p^F$ \cite{Voereduced} and takes the following form for any base field $F$: 
\[
\cA^\vee_p = \mathbb{M}_p^F [\overline{\tau}_0, \overline{\tau}_1, \dots, \overline{\xi}_1, \overline{\xi}_2, \dots] /I,
\]
where $|\overline{\tau}_i| = (2p^i-1, p^i-1)$ and $|\overline{\xi}_i| = (2p^i-2, p^i-1).$ The ideal $I$ of relations is dependent on the base field $F$. In the cases we are concerned with, 
\begin{equation} \label{eq:SteenrodRelIdeal}
    I = \left\{\begin{array}{llrr}
    (\overline{\tau}_i^2) & F=\mathbb{C},\mathbb{R},\mathbb{F}_q&\text{char}(\mathbb{F}_q)\neq p & p>2;\\
    (\overline{\tau}_i^2 = \tau\overline{\xi}_{i+1}) & F = \mathbb{C}, \mathbb{F}_q \, & q \equiv 1 \, (4) &p=2;\\
    (\overline{\tau}_i^2 = \tau\overline{\xi}_{i+1} + \rho \overline{\tau}_{i+1} + \rho \overline{\tau}_0\overline{\xi}_{i+1}) & F = \mathbb{R}, \mathbb{F}_q \, & q \equiv 3 \, (4)&p=2.
\end{array}\right.
\end{equation}
The coproduct $\psi:\mathcal{A}^\vee_p \to \mathcal{A}^\vee_p \otimes_{\mathbb{M}_p}\mathcal{A}^\vee_p$ is given by
\[\psi(\overline{\tau}_k) = 1 \otimes \overline{\tau}_k + \sum_{i+j=k}\overline{\tau}_i \otimes \overline{\xi}_j^{p^i}, \quad \psi(\overline{\xi}_k) = \sum_{i+j=k}\overline{\xi}_i \otimes \overline{\xi}_j^{p^i}.\]
Note that in the odd primary case, $\mathcal{A}^\vee_p$ is equivalent to the base change of the classical odd primary dual Steenrod algebra to $\mathbb{M}_p^F$. 
We remark that in the case that $F = \mathbb{R}$ or $F=\mathbb{F}_q$ for $q \equiv 3 \, (4)$ when $p=2$, and for $F=\mathbb{F}_q$ for $q \not\equiv 1 \, (p^2)$ when $p>2$, the ring $\mathbb{M}_2^F$ is not central in the dual Steenrod algebra $\mathcal{A}_p^\vee$. Indeed, there is a nontrivial action of the Bockstein on motivic cohomology in these cases, which translates to the following formulae in the dual Steenrod algebra:
\[\eta_R(\tau) = \tau + \rho\overline{\tau}_0, \quad \eta_R(\zeta) = \zeta + \gamma\overline{\tau}_0,\]
where $\eta_R:\mathbb{M}_p^F \to \mathcal{A}_p^\vee$ is the Hopf algebroid right unit map \cite{HoyKelOst17,Voemotiviccohomology}. This simple observation makes computation over the dual Steenrod algebra much more difficult in these cases. In all cases, the left unit $\eta_L:\mathbb{M}_p^F \to \mathcal{A}^\vee_p$ is the usual inclusion.

For $n \geq 0$ the quotients $\mathcal{E}(n)_p^\vee$ of the dual Steenrod algebra take the form
\[
\cE(n)^\vee_p = 
\mathbb{M}_p^F[\overline{\tau}_0, \cdots, \overline{\tau}_n]/J,
\]
where
\begin{equation}
\label{eq:EnRelations}
    J = \left\{ \begin{array}{llrr}
   (\overline{\tau}_i^2)  & F =\mathbb{C}, \mathbb{R}, \mathbb{F}_q&\text{char}(\mathbb{F}_q)\neq p&p>2;  \\
   (\overline{\tau}_i^2) & F = \mathbb{C}, \mathbb{F}_q &q \equiv 1 \, (4)&p=2;\\
   (\overline{\tau}_i^2=\rho\overline{\tau}_{i+1}, \overline{\tau}_n^2: 0 \leq 1 \leq n-1)  & F = \mathbb{R}, \mathbb{F}_q&q \equiv 3 \, (4)&p=2.
\end{array} \right.
\end{equation}
In particular, $\mathcal{E}(n)^\vee_p$ is the dual of the exterior subalgebra $\mathcal{E}(n)_p$ of the Steenrod algebra generated by the Milnor primitives $Q_0, \dots, Q_n$.

Let $BPGL$ denote the motivic Brown--Peterson spectrum \cite{HuKriz01,Hoyois15}. One may obtain this spectrum by taking the algebraic cobordism spectrum $MGL$, localizing at a prime $p$, and looking at the image of the motivic Quillen idempotent. In particular, this spectrum depends on a choice of prime. For $n \geq 0$, there exist canonical elements $v_n \in \pi_{*,*}BPGL$ of degree $(2(p^n-1), p^n-1)$, where we set $v_0:=p$. We define the motivic truncated Brown--Peterson spectrum $BPGL \langle n \rangle$ as the quotient by the regular sequence $(v_{n+1}, v_{n+2}, \dots)$:
\[BPGL \langle n \rangle := BPGL/(v_{n+1}, v_{n+2}, \dots).\]
By construction, we are supplied with cofiber sequences
\begin{equation} \label{eq:BPGLcofib}
    \Sigma^{2(p^n-1), p^n-1}BPGL \langle n \rangle \xrightarrow{v_n}BPGL \langle n \rangle \to BPGL \langle n-1 \rangle.
\end{equation}



\subsection{Motivic Adams spectral sequence} 
\label{subsec:adams spectral sequence}
Let $E \in SH(F)$ be a motivic ring spectrum. There is a canonical Adams tower associated to the unit map $\mathbb{S} \to E$ taking the form
\[\begin{tikzcd}
	{\mathbb{S}} & {\Sigma^{-1, 0}\overline{E}} & {\Sigma^{-2, 0}\overline{E} \wedge \overline{E}} & \cdots \\
	{E} & {\Sigma^{-1, 0}\overline{E} \wedge E} & {\Sigma^{-2, 0}\overline{E} \wedge \overline{E} \wedge E}
	\arrow[from=1-1, to=2-1]
	\arrow[from=1-2, to=1-1]
	\arrow[from=1-2, to=2-2]
	\arrow[from=1-3, to=1-2]
	\arrow[from=1-3, to=2-3]
	\arrow[from=1-4, to=1-3]
\end{tikzcd}\]
where $\overline{E}$ denotes the cofiber of the unit map. For any motivic spectrum $X$, we can apply the functor $\pi_{*,*}^F(X \wedge-)$ to this tower. This yields the $E$-\textit{based motivic Adams spectral sequence for $X$}. By construction, the $E_1$-page of this spectral sequence has signature
\[E_1 = \pi^F_{s+f, w}(E \wedge \overline{E}^{\wedge f} \wedge X) \implies \pi_{s, w}^FX_E, \quad d_r:E_r^{s,f,w} \to E_r^{s-1, f+r, w}\]
where $X_E$ denotes the $E$-nilpotent completion of $X$ \cite{Bousfield79}. In general, convergence is not immediate, but we will only be interested in particular cases that are well-understood. 

For $E=BPGL \langle n \rangle$, there is a $BPGL \langle n \rangle$-based motivic Adams sequence. When $X=\mathbb{S}$ this takes the form
\[E_1=\pi^F_{s+f, w}(BPGL \langle n \rangle \wedge \overline{BPGL \langle n \rangle}^{\wedge f}) \implies \pi^F_{s,w}\mathbb{S}_{(p)}.\]
This $BPGL \langle n \rangle$-based spectral sequence is the primary motivation for our work. In order to compute the $E_1$-page, we will investigate both the homotopy and spectrum-level behavior of the cooperations algebra $BPGL \langle n \rangle \wedge BPGL \langle n \rangle$ for $n=0, 1$.

For $E=H$, the resulting motivic Adams spectral sequence will be referred to as the $\textbf{mASS}_p^{F}(X)$. In this case, since $\mathcal{A}_p^\vee = \pi_{*,*}^{F}(H\wedge H)$ is flat as a module over $\mathbb{M}_p^{F}$, we are able to pass directly to the $E_2$-page. The $\textbf{mASS}_p^{F}(X)$ takes the form
\[E_2 = \text{Ext}^{s,f,w}_{\mathcal{A}_p^\vee}(\mathbb{M}_p^F,H_{*,*}(X)) \implies \pi_{s,w}^{F}X.\]
Convergence for this motivic Adams spectral sequence was first studied by \cite{HKO-convergencemASS,DImASS}. For $X=BPGL \langle n \rangle \wedge BPGL \langle n \rangle$, this spectral sequence takes the form
\[E_2 = \text{Ext}^{s,f,w}_{\mathcal{A}^\vee_p}(\mathbb{M}_p^F,H_{*,*}(BPGL \langle n \rangle \wedge BPGL \langle n \rangle )) \implies \pi_{s,w}^F(BPGL \langle n \rangle \wedge BPGL \langle n \rangle).\] 
It is through this spectral sequence that we will compute the homotopy ring of cooperations for $BPGL \langle0 \rangle$ and $BPGL \langle 1 \rangle$. We will often omit the first argument of an Ext group if it is given by $\mathbb{M}_p^F$.

\begin{remark}
    \rm One of the reasons we are interested in the $BPGL \langle 1 \rangle$-based motivic Adams spectral sequence is that it allows one to access $v_1$-periodicity. At the prime $p=2$, another way to access $v_1$-periodicity is by way of the $kq$-based motivic Adams spectral sequence, where $kq$ denotes the very effective Hermitian $K$-theory spectrum. This spectral sequence was studied over $\mathbb{C}$ by Culver--Quigley \cite{CQ21}, and the homotopy ring of cooperations were computed over $\mathbb{R}$ and $\mathbb{F}_q$ by the second author \cite{Realkqcoop, finitekqcoop}. It will be interesting to compare these two approaches to detecting motivic $v_1$-peiodicity.
\end{remark}

\subsection{Relative homology}
\label{subsec:relative}
We will also make use of motivic relative homology and relative Adams spectral sequences in our computations. Analogously to the classical non-motivic setting, if $E$ is an $R$-algebra and $M$ is an $R$-module in spectra, then $R$-relative $E$-homology is $E$-homology in the category of $R$-modules:
\[
E^R_{*,*} (M) : = \pi^F_{*,*} (E \underset{R}{\wedge} M).
\]
Note that
\[
E^R_{*,*} (R \wedge M) \cong \pi^F_{*,*} (E \underset{R}{ \wedge } R {\wedge} M ) \cong E_{*,*} M.
\]

In \cite[Prop 2.1]{BakerLazarev01}, Baker--Lazarev introduce a relative Adams spectral sequence in the category of $R$-modules. The motivic version exists by the same construction.

\begin{proposition}
    Let $R$ be a motivic ring spectrum, $E$ be an $R$-algebra, and $X,\,Y$ be $R$-modules such that $E_{*,*}^RE$ is flat over $E_{*,*}$. If $E^R_{*,*}X$ is projective as an $E_{*,*}$-module, then there exists an $E$-based motivic Adams spectral sequence in the category of $R$-modules
    \[
    E_2^{s,f,w} \cong \Ext_{E^R_{*,*}E}^{s,f,w} (E^R_{*,*} X , \, E^R_{*,*} Y ) \implies [X , Y]_{\hat{E}, (s,w)}^{R}  
    \]
    where $[X, Y]_{\hat{E}}^R$ denotes the $E$-nilpotent completion of $R$-module maps from $X$ to $Y$, $s$ denotes the stem, $f$ denotes the homological degree, and $w$ denotes the weight. 
\end{proposition} 

We will use the $BPGL\langle 0 \rangle$- and $BPGL \langle 1 \rangle$-relative Adams spectral sequences to produce our spectrum-level splittings of the respective cooperations algebras.

\section{Motivic homology results}
\label{sec:motivic homology}

In this section, we make multiple computations in motivic homology. First, we compute the $BPGL\langle 0 \rangle$- and $BPGL \langle 1 \rangle$-dual Steenrod algebras. Then, we recall the homology of $BPGL \langle n \rangle$ and prove a K\"unneth isomorphism. Following this, we establish a criterion for checking freeness for comodules over $\mathcal{E}(n)_p^\vee$ and a Whitehead Theorem for Margolis homology over $\mathcal{E}(1)^\vee_p$. We next introduce Brown--Gitler comodules and show that they assemble to give an $\mathcal{E}(n)^\vee_p$-comodule splitting of $H_{*,*}BPGL \langle n \rangle$. Finally, we introduce homological lightning flash modules. These modules simplify our computations in the $BPGL \langle 1 \rangle$-even further than the aforementioned Brown--Gitler comodule splitting. All results are independent of base field unless otherwise stated.
\subsection{Relative dual Steenrod algebra}
We will make use of the following relative homology computation.

\begin{proposition} \label{prop:relHomology}
    The $BPGL \langle 0 \rangle$-relative dual Steenrod algebra is
    \[
    H_{*,*}^{BPGL \langle 0 \rangle} H \cong \cE(0)_p^\vee
    \]
    and the $BPGL \langle 1 \rangle$-dual Steenrod algebra is
    \[
    H_{*,*}^{BPGL \langle 1 \rangle} H \cong \cE(1)_p^\vee
    \]
\end{proposition}


\begin{proof}
    We give an argument for the second statement and leave the first to the reader. Consider the cofiber sequence
    \[
    \Sigma^{2p-2, p-1}BPGL \langle 1 \rangle \xrightarrow{v_1} BPGL \langle 1 \rangle \xrightarrow{} BPGL \langle 0 \rangle.
    \]
    Smashing $H \underset{BPGL \langle 1 \rangle}{ \wedge } -$ with the cofiber sequence yields a splitting 
    \[
    H \underset{ BPGL \langle 1 \rangle }{\wedge} BPGL \langle 0 \rangle \simeq H\wedge (S^{0,0} \vee {S^{2p-1,p-1}})
    \]
    which suffices to determine the additive structure of $H_{*,*}^{BPGL \langle 1 \rangle} BPGL \langle 0 \rangle$. Now consider the cofiber sequence
   \[
    BPGL \langle 0 \rangle \xrightarrow{p} BPGL \langle 0 \rangle \xrightarrow{} H.
    \]
    Smashing $H\underset{BPGL \langle 1 \rangle}{ \wedge } -$ with this cofiber sequence yields a splitting \[
    H \underset{ BPGL \langle 0 \rangle }{\wedge} H \simeq H \wedge (S^{0,0} \vee S^{1,0} \vee {S^{2p-1,p-1}} \vee {S^{2p,p-1}}) .
    \]
    The multiplicative structure follows from the algebra map 
    \[ 
    H \wedge H \to H \underset{BPGL \langle 1 \rangle}{ \wedge } H .
    \]
\end{proof}

\subsection{The homology of $BPGL \langle n \rangle$}
We record the following results on the mod $p$-homology of truncated motivic Brown--Peterson spectra.
\begin{proposition} \cite[Theorem 6.19]{Hoyois15}
\label{prop:BPGLnHomology}
    The map $H_{*,*} BPGL \langle n \rangle \to H_{*,*} H$ is injective with image the $\mathcal{A}^\vee_p$-subalgebra
    \[
    H_{*,*} BPGL \langle n \rangle \cong \mathbb{M}_p^F [\bar{\xi}_1, \bar{\xi}_2, \bar{\xi}_3, \cdots, \bar{\tau}_{n+1}, \bar{\tau}_{n + 2}, \bar{\tau}_{n + 3}, \cdots]/I,
    \]
    where $I$ is the ideal defined in \Cref{eq:SteenrodRelIdeal}. In particular, for any prime $p$, we have an $\mathcal{A}_p^\vee$-comodule isomorphism 
    \[
    H_{*,*} BPGL \langle n \rangle \cong \cA_p^\vee \underset{\cE(n)^\vee_p}{\Box} \mathbb{M}_p
    .\]
\end{proposition}
The left $\cE(n)_p^\vee$-coaction on $\mathcal{A}_p^\vee \underset{\cE(n)_p^\vee}{\Box} \mM_p$ is given by
\begin{equation} \label{eq:coaction}
    \alpha: \mathcal{A}_p^\vee \underset{\cE(n)_p^\vee}{\Box} {\mM_p} \xrightarrow{\psi \otimes 1} \mathcal{A}_p^\vee \otimes \left(\mathcal{A}_p^\vee \underset{\cE(n)_p^\vee}{\Box} {\mM_p} \right) \xrightarrow{\pi \otimes 1} \cE(n)_p^\vee \otimes \left(\mathcal{A}_p^\vee \underset{\cE(n)_p^\vee}{\Box} \mM_p \right)
\end{equation}
which on generators $\bar{\xi}_k$ and $\bar{\tau}_k$ is 
\begin{align*}
    \alpha (\bar{\xi}_k) & = 1 \otimes \bar{\xi}_k \\
    \alpha (\bar{\tau}_{n + k}) & = 1 \otimes \bar{\tau}_{n + k} + \sum_{i = 0}^n \bar{\tau}_i \otimes \bar{\xi}_{n + k - i}^{2^i}.
\end{align*}
Let $M$ be a left $\cE(n)_p^\vee$-comodule. Then there is an induced right $\cE(n)$-module action 
\[
\lambda: M \otimes \cE(n)_p \to M
\]
defined by
\[
\lambda (x, \theta) = (\theta \otimes Id_M) \circ \alpha (x),
\]
where $\alpha(x) = \sum_i \theta_i \otimes x_i$ (see \cite[\textsection 6]{Boardman82}).

Since $\bar{\tau}_i$ is the dual of $Q_i$, the right $\cE(1)$-module action on $H_{*,*}BPGL \langle 1 \rangle$ is
\begin{align*}
    \bar{\xi}_j Q_i & = 0 \\
    \bar{\tau}_k Q_0 & = \bar{\xi}_k \\
    \bar{\tau}_k Q_1 & = \bar{\xi}_{k -1}^2.
\end{align*}

\begin{proposition}
\label{prop:BPGLnKunneth}
    For any prime $p$ and any $n, m \geq 0$, there is an isomorphism of $\mathcal{A}_p^\vee$-comodule algebras:
    \[H_{*,*}(BPGL\langle m \rangle \wedge BPGL \langle n \rangle) \cong H_{*,*}(BPGL \langle m \rangle ) \otimes H_{*,*}(BPGL \langle n \rangle ). \]
\end{proposition}

\begin{proof}
    By \cref{prop:BPGLnHomology}, $H_{*,*}BPGL \langle n \rangle$ is free over $\mathbb{M}_p$. Thus the K\"unneth spectral sequence
    \[E_2 = \text{Tor}^{\mathbb{M}_p}(H_{*,*}BPGL \langle m \rangle, H_{*,*}BPGL \langle n \rangle) \implies H_{*,*}(BPGL \langle m \rangle \wedge BPGL \langle n \rangle)\]
    collapses at the $E_2$-page, giving the result.
\end{proof}

Since $\cE(n)^\vee_p$ is free over $\mM_p$, we obtain the following proposition through standard algebra (see for example \cite[Theorem 4.7]{BrzezinskiWisbauer03}).

\begin{proposition} \label{prop:EquivofCats}
    There is an equivalence of categories between left $\cE(n)_p^\vee$-comodules and right $\cE(n)_p$-modules.
\end{proposition}
This equivalence of categories allows us to prove facts about $\mathcal{E}(n)_p^\vee$-comodules by instead working with $\mathcal{E}(n)_p$-modules.
\begin{proposition}
\label{prop:freeImpliesInjective}
    If $M$ is a free $\cE(n)^\vee_p$-comodule, then $M$ is an injective $\cE(n)^\vee_p$-comodule
\end{proposition}

\begin{proof}
    By the equivalence of categories of \Cref{prop:EquivofCats}, we may instead show that any free $\mathcal{E}(n)_p$-module is injective. Since any free $\mathcal{E}(n)_p$-module is a sum of shifts of $\mathcal{E}(n)_p$, it suffices to show that $\mathcal{E}(n)_p$ is self-injective. This follows from modifying May's proof of self-injectivity for $\mathbb{M}_2^{C_2}$ given in the $C_2$-equivariant setting \cite{Clover20}. In particular, the graded ideals of $\mathcal{E}(n)_p$ are just the graded ideals of $\mathbb{M}_p$ with the addition of various $Q_i$.
\end{proof}
\subsection{Margolis homology and free $\cE(n)$-modules}
Classically, Margolis homology is an invariant of modules over a subalgebra of the Steenrod algebra. One particularly useful consequence of this invariant is that a module over the Steenrod algebra is free if and only if it has trivial Margolis homology with respect to a particular family of elements \cite[Theorem 19.6]{Margolis83}. 

In the motivic setting, Margolis homology is slightly more complicated and subtle (see, for example, \cite[Example 4.6 and Proposition 4.7]{GheorgheIsaksenRicka18}). Nonetheless, with the imposition of additional freeness criteria, motivic analogues of Margolis homology have proven useful (see, for instance, \cite[Section 5]{HeardKrause18} and \cite[Section 2]{BhattacharyaGuillouLi22}). 

The goal of this subsection is to give a Whitehead Theorem for Margolis homology, that is, a criterion for when an $\cE(1)_p$-module map is a stable equivalence, in the setting of motivic homotopy. To this end, we establish some motivic freeness criteria extending those of \cite[Section 5]{HeardKrause18} and \cite[Section 2]{BhattacharyaGuillouLi22} in the $\mC$ and $\mR$-motivic settings, respectively. Let
\[
        x = \begin{cases}
            1 & F = \mC\\
            1 & F = \mR \text{ and } p \neq 2\\
            \rho & F = \mR \text{ and } p = 2 \\
            \rho & F = \mF_q \text{ with } \textup{char}({F_q}) \neq 2 \text{ and } q \equiv 3 \, (4) \\
            u & F = \mF_q \text{ with } \textup{char}({F_q}) \neq 2 \text{ and } q \equiv 1 \, (4) \\
            u & F = \mF_q \text{ with } \textup{char}(\mF_p) \neq p \text{ and } q^i \equiv 1 \, (p^2) \\
            \gamma & F = \mF_q \text{ with } \textup{char}(F_q) \neq p \text{ and } q^i \neq 1 \, (p^2)
            \end{cases}.
        \] 
Our first aim will be to prove the following proposition.
\begin{proposition} \label{prop:marg1}
    A finitely generated $\cE(n)_p$-module $M$ is free if and only if
    \begin{enumerate}
        \item $M$ is free as an $\mF_p[x]$-module, where $x \in \mM_p^F$ as defined above;
        \item $\mF_p \otimes_{\mM_p^F} M$ is free as an $E(n)_p$-module.
    \end{enumerate}
\end{proposition}
To prove \cref{prop:marg1}, we will make use of the following lemma. 
\begin{lemma} \label{lem:marg1}
    Let $x \in \mM_p^F$ as defined in \cref{prop:marg1}. A finitely generated $\cE(n)_p$-module $M$ is free if and only if
    \begin{enumerate}
        \item $M$ is a free $\mF_p [x]$-module and
        \item $M/(x)$ is a free $\cE(n)_p / (x)$-module. 
    \end{enumerate}
\end{lemma}
\begin{proof}
    When $x = 1$, the statement is trivial. When $F = \mR$ and $p = 2$, the statement is exactly that of \cite[Lemma 2.1]{BhattacharyaGuillouLi22}. The remaining cases where $F = \mF_q$ can be proved using the same argument as the proof of \cite[Lemma 2.1]{BhattacharyaGuillouLi22}, and in fact the induction there is made simpler since $x^2 = 0$.
\end{proof}

We also make a further reduction. 
 Let 
    \[
    y = \begin{cases}
        \tau & F = \mC \\
        \tau & F = \mR \text{ and } p = 2 \\
        \theta & F = \mR \text{ and } p \neq 2 \\
        \tau & F = \mF_q \text{ with } \textup{char}(\mF_q) = p \\
        \zeta &  F = \mF_q \text{ with } \textup{char}(\mF_q) \neq p 
    \end{cases}.
    \]
\begin{lemma} \label{lem:marg2}
    A finitely generated $\cE(n)_p / (x)$-module $M$ is free if and only if $M /(y)$ is a free $(\cE(n)_p/(x))/(y)$-module.
\end{lemma}

\begin{proof}
    The proof follows from the argument given in \cite[Lemma 5.2]{HeardKrause18} along with the observation that the category of connective graded comodules over a connective, graded, flat, finite-type Hopf algebroid has enough projectives \cite{Salch23}.
\end{proof}

\begin{proof}[Proof of \cref{prop:marg1}]
If a finitely generated $\cE(n)_p$-module $M$ is free, the two conditions are immediately satisfied. Thus we consider $M$ where the  two conditions are satisfied and show $M$ is a free $\cE(n)_p$-module. First, note that condition (2) of \cref{prop:marg1} is equivalent to saying that $(M / (x))/(y)$ is free over $(\cE(n)_p/(x))/(y)$. \cref{lem:marg2} then implies that $M / (x)$ is a free $\cE(n)/ (x)$-module. So by \cref{lem:marg1}, $M$ is indeed a free $\cE(n)_p$-module.
\end{proof}
Recall the following freeness criteria for modules over $E(n)$.
\begin{proposition}[{\cite[Theorem 18.8i]{Margolis83}}] \label{prop:MargClassical}
    Let $M$ be a bounded-below $E(n)$-module. Then $M$ is free if and only if $M_*(M, Q_i) = 0$ for all $0 \leq i \leq n$. 
\end{proposition}
Note that a map of $\cE(n)_p$-modules $f:M \to N$ is said to be a stable equivalence if there exist free $\cE(n)_p$-modules $P,Q$ such that the map $M \oplus P \to N \oplus Q$ induced by $f$ is an isomorphism \cite[p. 205]{Margolis83}.

Combining \cref{prop:MargClassical} with \cref{prop:marg1} yields a motivic analogue of the Whitehead Theorem for Margolis homology over the base field $F$. The proof proceeds in exactly the same way as its classical analogue (see \cite[Theorem 18.8ii]{Margolis83} or \cite[Corollary 4.10]{GheorgheIsaksenRicka18} for the details when $F = \mC$).

\begin{proposition}[Whitehead Theorem] \label{prop:MotivicWhitehead}
    Let $x \in \mathbb{M}_p^F$ as defined above, let $M$ and $N$ be finitely generated $\cE(1)_p$-modules that are $\mF_p [ x]$-free, and let $f: M \to N$ be an $\cE(1)_p$-module map. Then $f$ is a stable equivalence if and only if $(f/(x))/(y): (M/(x))/(y) \to (N/(x))/(y)$ induces an isomorphism in Margolis homologies with respect to $Q_0$ and $Q_1$. 
\end{proposition}

\subsection{Brown--Gitler subcomodules}


Define a weight filtration on the dual Steenrod algebra $\cA_p^\vee$ by setting
\[
\wt(\bar{\xi}_i) = \wt(\bar{\tau}_i) = p^i
\]
and extend multiplicatively so
\[
\wt(xy) = \wt(x) + \wt(y).
\]
\begin{definition}
The $k^{th}$ Brown--Gitler comodule, denoted $B_n(k)$, is the subspace of $\cA // \cE (n)^\vee $ spanned by monomials of weight less than or equal to $pk$.
\end{definition}

\begin{remark}
    \rm The motivic Brown--Gitler comodules defined in \cite{CQ21,Realkqcoop}, while denoted in the same fashion as above, are a different family of comodules, as they are obtained from the Mahowald weight filtration on $\mathcal{A}//\mathcal{A}(n)^\vee$. However, note that there is no ambiguity in $B_0(k)$ as there is an isomorphism $\mathcal{A}//\mathcal{A}(0)^\vee \cong \mathcal{A}//\mathcal{E}(0)^\vee$.
\end{remark}


\begin{proposition} \label{prop:BGfreeinjective}
    The Brown--Gitler comodule $B_{-1}(k)$ is free and injective over $\cE(0)_p^\vee$.
\end{proposition}

\begin{proof}
    We must only show that $B_{-1}(k)$ is free due to \Cref{prop:freeImpliesInjective}. Since $B_{-1}(k)$ is finitely-generated over $\mathcal{E}(0)^\vee_p$ by construction, after base change to $\mathbb{F}_p$, $B_{-1}(k) \otimes_{\mathbb{M}_p}\mathbb{F}_p$ is equivalent to the homology of the classical Brown--Gitler spectrum. This is free over the classical $E(0)_p^\vee$. By \Cref{prop:marg1}, $B_{-1}(k)$ is free over $\mathcal{E}(0)^\vee_p$, finishing the proof. 
\end{proof}

\begin{proposition}\label{prop:BGdecomposition} 
    For all $n \ge 0$, there exists an $\mathcal{E}(n)_p^\vee$-comodule isomorphism 
    \[
    H_{*,*}BPGL\langle n \rangle \cong \bigoplus\limits_{k=0}^{\infty} \Sigma^{2k(p-1),k(p-1)} B_{n-1}(k).
    \]
\end{proposition}
\begin{proof}
    This is the same as the proof in the classical setting. The classical statement is an immediate consequence of \cite[Cor~4.10]{CulverOddp20} combined with the first sentence after Proposition 4.6 of \cite{CulverOddp20}. Note that \cite{CulverOddp20} uses the indexing convention 
    \[
    N_{i}(k) = \{ x \in A//E(n)_{*}| wt(x) \le pn \},
    \]
    as opposed to our indexing convention 
    \[
    B_{i}(k) = \{ x \in A//E(n)_{*}| wt(x) \le n \}.
    \]
\end{proof}

\subsection{Homological lightning flash modules}

\begin{definition}\label{def: lightning}
    The motivic homological lightning flash module is given by 
    $$L_p(k) = \mathcal{E}(1)_p \{x_1, x_2, \cdots, x_k \, |\,  x_{i + 1} Q_1 =  x_i Q_0,\, 1 \leq i \leq k-1 \}$$
    where $| x_i| = (2p-2, p-1)i+(1,0).$ Further define $L_p(0) = \mM_p.$
\end{definition}

These lightning flash modules can be easily visualized as in \cref{fig:HomLightFlash}, which depicts $L_3(2)$ using the motivic grading convention that the total topological degree is plotted along the horizontal axis and the weight is plotted along the vertical axis.  Here, a point denotes a copy of $\mM_3$, straight arrows indicate the
(non-trivial) operation of $Q_0,$ and curved arrows indicate the (non-trivial) operation of $Q_1.$  

\begin{figure}[ht]\label{fig:lightning flash}
    \centering
    \resizebox{2in}{1in}{
    \begin{tikzpicture}5
		

        \draw[step=1.0,gray!50!white,thin] (-0.5,-0.5) grid (9.5,4.5);
    
    \draw[thick] (5,2) -- (4,2); 
    \node [scale=1.5] at (4.5,2) {\textup{$<$}}; 
    \draw[thick] (9,4) -- (8,4); 
    \node [scale=1.5] at (8.5,4) {\textup{$<$}}; 

    \draw[thick] (0,0)  arc (180:90:3.5); 
    \draw[thick] (3.5,3.5) arc (90:0:1.5); 
    \node [scale=1.5] at (3.5,3.5) {\textup{$<$}}; 
    \node [scale = 1.5] at (5.5,1.75) {\textup{$x_1$}}; 

    \draw[thick] (4,2) arc (180:270:1.5); 
    \draw[thick] (5.5,0.5) arc (270:360:3.5); 
    \node [scale=1.5] at (5.5,0.5) {\textup{$<$}}; 
    \node [scale=1.5] at (9.5,3.75) {$x_2$}; 
    

    

    \node at (0,0) {\textup{$\bullet$}}; 
    \node at (5,2) {\textup{$\bullet$}}; 
    \node at (4,2) {\textup{$\bullet$}}; 
    \node at (9,4) {\textup{$\bullet$}}; 
    \node at (8,4) {\textup{$\bullet$}}; 

    \end{tikzpicture}}
    \caption{Homological lightning flash module: $L_3(2)$}
    \label{fig:HomLightFlash}
\end{figure}
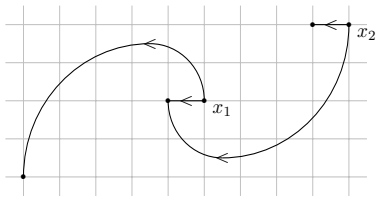

Notice that the inclusion of the top $L_p(k-1)$-submodule of $L_p(k)$ supplies us with a short exact sequence of $\mathcal{E}(1)^\vee_p$-comodules
\begin{equation}
\label{lightning flash ses}
0 \to \Sigma^{2(p-1),p-1} L_p(k-1) \to L_p (k) \to \left( \cE(1)_p // \cE(0)_p \right)^\vee \to 0.
\end{equation}

\begin{proposition}
\label{prop:LightningFlashBGiso} 
    There is an isomorphism of $\mathcal{E}(1)_p^\vee$-comodules:
    \[B_0(k) \cong L_p(\nu_p(k!)) \oplus W_k,\]
    where $W_k$ is a sum of suspensions of $\mathcal{E}(1)_p^\vee$.
\end{proposition}

\begin{proof}
    We will use \cref{prop:MotivicWhitehead} to prove this isomorphism. We start by computing the relevant Margolis homologies. First, by considering figures similar to \cref{fig:HomLightFlash} we see that
    \[
    M_* \big([(L_p(\nu_p (k!)))/(x)]/(y), Q_0\big) \cong \mathbb{M}_p \{1\}
    \]
    \[
    M_* \big([(L_p(\nu_p (k!)))/(x)]/(y), Q_0\big) \cong \mathbb{M}_p \{Q_0 x_{\nu_p(k!)} \}.
    \]

    Next we must compute the Margolis homology of the Brown--Gitler comodules $B_0(k)$. It is easier to first compute the Margolis homology for $H_{*,*} BPGL \langle 1 \rangle$, 
    \begin{align*}
        M_*\left(\left[\left(H_{*,*} BPGL \langle 1 \rangle\right)/\left(x\right)\right]/\left(y\right), \, Q_0 \right) & \cong \mathbb{M}_p \{1  \} \\
        M_* \left([(H_{*,*} BPGL \langle 1 \rangle)/(x)](y), \, Q_1 \right) & \cong \mathbb{M}_p \{ \bar{\xi}_1^{\epsilon_1} \bar{\xi}_2^{\epsilon_2} \cdots \bar{\xi}_n^{\epsilon_n} | \, 0 \leq \epsilon_i \leq 1 \}.
    \end{align*}
    (This computation proceeds similarly to Adams' argument in the classical setting when $p = 2$ \cite[ Lemma 16.9]{Adams74}.)
    By \cref{prop:BGdecomposition}, 
    \[
    H_{*,*}BPGL\langle n \rangle \cong \bigoplus\limits_{k=0}^{\infty} \Sigma^{2k(p-1),k(p-1)} B_{n-1}(k).
    \]
    Since $H_{*,*} BPGL \langle 1 \rangle$ is of finite-type, we can apply this decomposition to the corresponding Margolis homology and get
    \begin{align*}
        M_* \big([B_0(k)/(x)]/ (y) , Q_0 \big) & \cong \mF_p \{1\} \\
        M_* \left([B_0(k)/(x)]/ (y) , Q_1 \right) & \cong \mF_p \{\bar{\xi}_{i_1} \bar{\xi}_{i_2} \cdots \bar{\xi}_{i_n} \}, \\ 
    \end{align*}
    where $i_1 < i_2 < \cdots < i_n$ and $p^{i_1} + p^{i_2} + \cdots + p^{i_n}$ is the $p$-adic expansion of $k$. 
    Note that $|x_{\nu_p(k!)}|_{s,f} = | \bar{\xi}_{i_1} \bar{\xi}_{i_2} \cdots \bar{\xi}_{i_n} |_{s,f}$. So if we construct an $\cE(1)_p$-module map
    \[
    L_p(\nu_p(k!)) \to B_0(k)
    \]
    realizing the isomorphism, then we can apply \cref{prop:MotivicWhitehead} to conclude the proof. Let $S(k)$ denote the submodule of $B_0(k)$
    \[
    S(k) = \cE(1)_p \{\bar{\xi}_{i_1} \bar{\xi}_{i_2} \cdots \bar{\xi}_{i_n} \bar{\tau}_j | j \leq i_1 < i_2 < \cdots < i_n \}.
    \]
    Observe that $S(k)$ is naturally isomorphic to $L_p(\nu_p(k!))$ ans so the inclusion $S(k) \to B_0(k)$ is exactly the map we are looking for. 
\end{proof}

\section{Splitting $BPGL \langle 0 \rangle \wedge BPGL \langle 0 \rangle$}
\label{sec:BPGL0}
In this section, we construct spectrum-level splittings of the cooperations algebra $BPGL \langle 0\rangle \wedge BPGL \langle 0 \rangle$. We also compute the homotopy ring of cooperations $\pi_{*,*}(BPGL \langle 0 \rangle \wedge BPGL \langle 0 \rangle)$. Note that at any prime $p$, a model for $BPGL \langle 0 \rangle$ is $H \mathbb{Z}_{(p)}$. 

\subsection{Strategy}
In this section, all results are stated independent of base field unless otherwise stated. We will begin by computing the homotopy groups $\pi_{*,*}(BPGL \langle 0 \rangle)$ by a motivic Adams spectral sequence. The $\textbf{mASS}_p(BPGL \langle 0 \rangle)$ takes the form
\[\text{Ext}^{s,f,w}_{\mathcal{A}_p^\vee}(H_{*,*}(BPGL \langle 0 \rangle)) \implies \pi_{s,w}(BPGL \langle 0 \rangle).\]
Recall from \cref{prop:BPGLnHomology} that there is an isomorphism of $\mathcal{A}^\vee_p$-comodules:
\[H_{*,*}(BPGL \langle 0 \rangle) \cong \mathcal{A}^\vee_p\underset{\mathcal{E}(0)^\vee_p}{\Box}\mathbb{M}_p.\]
Using a change-of-rings isomorphism, the $E_2$-page can then be rewritten as
\[E_2 = \text{Ext}^{s,f,w}_{\mathcal{A}_p^\vee}(\mathcal{A}^\vee_p\underset{\mathcal{E}(0)^\vee_p}{\Box}\mathbb{M}_p) \cong \textup{Ext}^{s,f,w}_{\mathcal{E}(0)_p^\vee}(\mathbb{M}_p).\]

After computing these homotopy groups, we will compute the cooperations algebra by another motivic Adams spectral sequence. The $\textbf{mASS}_p(BPGL \langle 0 \rangle \wedge BPGL \langle 0 \rangle)$ has signature
\[\text{Ext}^{s,f,w}_{\mathcal{A}_p^\vee}(H_{*,*}(BPGL \langle 0 \rangle \wedge BPGL \langle 0 \rangle)) \implies \pi_{s,w}(BPGL \langle 0 \rangle \wedge BPGL \langle 0 \rangle).\]
Using the K\"unneth isomorphism of \cref{prop:BPGLnKunneth}, the change-of-rings isomorphism, and the $\mathcal{E}(0)_p^\vee$-comodule isomorphism of \cref{prop:BGdecomposition}, we can rewrite the $E_2$-page as
\[E_2= \text{Ext}^{s,f,w}_{\mathcal{A}_p^\vee}(H_{*,*}(BPGL \langle 0 \rangle \wedge BPGL \langle 0 \rangle)) \cong \bigoplus_{k \geq 0}\Sigma^{2k(p-1), k(p-1)}\text{Ext}^{s,f,w}_{\mathcal{E}(0)_p^\vee}(B_{-1}(k)).\]
For $k > 0$, we have that $B_{-1}(k)$ is free and injective over $\mathcal{E}(0)_p^\vee$ by \cref{prop:BGfreeinjective}, giving an isomorphism 
\[\text{Ext}^{s,f,w}_{\mathcal{E}(0)_p^\vee}(B_{-1}(k))=W_k,\] 
where $W_k$ is a free $\mathbb{M}_p$-module  concentrated in Adams filtration 0. Since $B_{-1}(0)=\mathbb{M}_p$, we can decompose the $E_2$-page as
\begin{equation}
\label{e2-mass BPGL0 coop}
E_2=\text{Ext}^{s,f,w}_{\mathcal{E}(0)_p^\vee}(\mathbb{M}_p) \oplus \bigoplus_{k \geq 1}W_k,
\end{equation}
with $\bigoplus_{k \geq 1}W_k$ a free $\mathbb{M}_p$-module in Adams filtration 0. The following allows us to determine differentials.
\begin{lemma}
\label{BPGL0 coop mass differentials}
    The differentials in the $\textup{\textbf{mASS}}_p(BPGL \langle 0 \rangle \wedge BPGL \langle 0 \rangle)$ are determined by the differentials in the $\textup{\textbf{mASS}}_p(BPGL \langle 0 \rangle).$
\end{lemma}
\begin{proof}
    The map $BPGL \langle 0 \rangle \to BPGL \langle0 \rangle \wedge BPGL \langle 0 \rangle$ induces an inclusion on $E_2$-pages
    \[\text{Ext}^{s,f,w}_{\mathcal{E}(0)_p^\vee}(\mathbb{M}_p) \to \text{Ext}^{s,f,w}_{\mathcal{E}(0)_p^\vee}(\mathbb{M}_p) \oplus \bigoplus_{k \geq 1} W_k.\]
    This determines all of the differentials on the $v_0$-periodic summand $\text{Ext}^{s,f,w}_{\mathcal{E}(0)^\vee_p}(\mathbb{M}_p)$. We will show that there are no differentials involving the summand $\bigoplus_{k \geq 1}W_k$, proving the claim. Since $v_0$ is a permanent cycle in the $\textbf{mASS}_p(BPGL \langle 0 \rangle)$, all differentials in the $\textbf{mASS}_p(BPGL \langle 0 \rangle \wedge BPGL \langle 0 \rangle)$ are $v_0$-linear. Indeed, if $x$ is any class on the $E_r$-page, then
    \[d_r(v_0x) = d_r(v_0)x + v_0d_r(x) = v_0d_r(x).\]
    This implies that there can be no differential from any $v_0$-torsion class to any $v_0$-torsion free class. In particular, there are no differentials from $\bigoplus_{k \geq 1}W_k$ to $\text{Ext}^{s,f,w}_{\mathcal{E}(0)_p^\vee}(\mathbb{M}_p)$. Note that Adams differentials increase filtration. Since $\bigoplus_{k \geq 1}W_k$ is contained in filtration 0 , this rules out any differentials from $\text{Ext}^{s,f,w}_{\mathcal{E}(0)_p^\vee}(\mathbb{M}_p)$ to $\bigoplus_{k \geq 1}W_k$ or from $\bigoplus_{k \geq 1}W_k$ to itself, concluding the proof.
\end{proof}

To produce a spectrum-level splitting, we will use the $BPGL \langle 0 \rangle$-relative Adams spectral sequence. Towards this end, we make the following relative homology computations.
\begin{proposition}\label{prop:homology splitting height zero}
    There exists an isomorphism of $\cE(0)_p^\vee$-comodules \[ H_{*,*}^{BPGL\langle 0 \rangle}\left( BPGL\langle 0 \rangle \wedge BPGL\langle 0 \rangle \right) \cong H_{*,*}^{BPGL\langle 0 \rangle}\left(BPGL\langle 0 \rangle \vee V \right) ,\]
where $V$ is a finite-type sum of suspensions of $H$.
\end{proposition}
\begin{proof}
Recall from \cref{subsec:relative} that $$H_{**}^{BPGL\langle 0 \rangle}BPGL\langle 0 \rangle \wedge X \cong H_{**}X$$ for any spectrum $X$, so we only need to show that $H_{**}BPGL\langle 0 \rangle \cong H_{**}(S^0 \vee V)$. Recall from \cref{prop:BGdecomposition} that 
\[H_{*,*}BPGL\langle 0 \rangle \cong \bigoplus\limits_{k=0}^{\infty}\Sigma^{k(2p-2, p-1)}B_{-1}(k).\]
Note that $H_{*,*}^{BPGL \langle 0 \rangle}BPGL\langle 0 \rangle \cong \mathbb{M}_p^F \cong B_{-1}(0)$. For $k > 0$, observe from  \cref{prop:BGfreeinjective} that $B_{-1}(k)$ is a finite sum of suspensions of $\cE(0)_p^{\vee}$, and then recall from \cref{prop:relHomology} that  $H_{*,*}^{BPGL\langle 0 \rangle}H \cong \cE(0)_p^\vee$.
 \end{proof}

\subsection{$BPGL\langle 0 \rangle$ coefficients}
We begin by computing the homotopy of $BPGL\langle 0 \rangle$. Note that in this section and all subsequent sections, we depict all charts in $(s,f)$-grading, with motivic weight $w$ suppressed.
\begin{proposition}
\label{mass BPGL0 C and R}
    Let $F = \mathbb{C}$ or $\mathbb{R}$. The $\textup{\textbf{mASS}}_p^F(BPGL \langle 0 \rangle)$ collapses on the $E_2$-page, giving isomorphisms
    \[\pi_{*,*}^\mathbb{C}(BPGL \langle 0 \rangle) \cong \mathbb{Z}_p[\tau]\]
    for any prime $p$ and
    \[ \pi_{*,*}^\mathbb{R}(BPGL \langle 0 \rangle) \cong\left\{\begin{array}{ll}
       \mathbb{Z}_2[\tau^2, \rho]/(2\rho)  & p=2  \\
       \mathbb{Z}_p[\theta]  & p>2.
    \end{array} \right.\]
\end{proposition}

\begin{proof}
    Over $\mathbb{C}$, since the motivic $\mathcal{E}(0)^\vee_p$ is the base change of the classical $E(0)^\vee_p$ from $\mathbb{F}_p$ to $\mathbb{F}_p[\tau] = \mathbb{M}_p^\mathbb{C}$, the $E_2$-page is the base change of the classical $E_2$-page of the Adams spectral sequence for $BP\langle 0 \rangle$:
    \[E_2 = \text{Ext}^{s,f,w}_{\mathcal{E}(0)^\vee_p}(\mathbb{M}_p^{\mathbb{C}}) \cong \mathbb{F}_p[\tau, v_0],\]
    where $|v_0| = (0,1,0).$ Over $\mathbb{R}$, the $E_2$-page was computed at the prime $p=2$ by Hill in \cite[Theorem 3.1]{Hill11} and shown to be
    \[\text{Ext}^{s,f,w}_{\mathcal{E}(0)^\vee_2}(\mathbb{M}_2^\mathbb{R}) \cong \mathbb{F}_2[\rho, \tau^2, v_0]/(\rho v_0).\]
    For odd primes, the $E_2$-page is determined in the same way as over $\mathbb{C}$, except that we are base changing from $\mathbb{F}_p$ to $\mathbb{F}_p[\theta] \cong \mathbb{M}_p^\mathbb{R}$:
    \[\text{Ext}^{s,f,w}_{\mathcal{E}(0)^\vee_p}(\mathbb{M}_p^\mathbb{R}) \cong \mathbb{F}_p[\theta, v_0].\]
    In all cases, there can be no differentials for degree reasons, so the spectral sequences must collapse and the result follows.
\end{proof}

The $E_2$-pages described in \Cref{mass BPGL0 C and R} are depicted in \Cref{Ext_e(0)_C_R}.
\begin{figure}[ht]
    \includegraphics[width=0.25\linewidth]{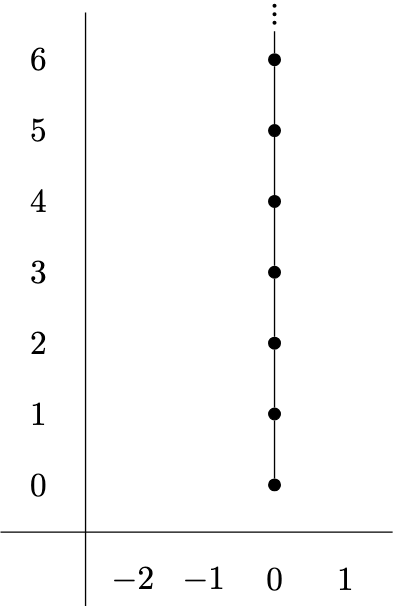}
    \hspace{1.5cm}
    \includegraphics[width=.25\linewidth]{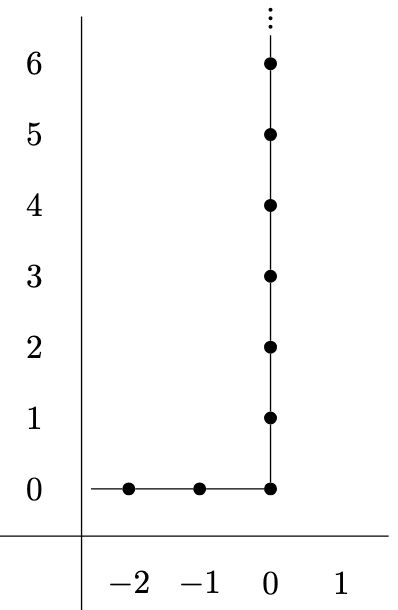}
    \caption{Charts for $\text{Ext}^{s,f,w}_{\mathcal{E}(0)_p^\vee}(\mathbb{M}_p^F)$ for $F=\mathbb{C}$ or $\mathbb{R}$. The left chart represents the case $F=\mathbb{C}$ and $p$ any prime, where a $\bullet$ denotes $\mathbb{F}_p[\tau]$, and also represents the case $F = \mathbb{R}$ and $p>2$, where a $\bullet$ denotes $\mathbb{F}_p[\theta].$ The right chart represents the case $F=\mathbb{R}$ and $p=2$, where a $\bullet$ denotes $\mathbb{F}_2[\tau^2].$ A vertical line represents multiplication by $v_0$, and a horizontal line represents multiplication by $\rho$.}
    \label{Ext_e(0)_C_R}
\end{figure}

\begin{proposition}
\label{BPGL mass finite field p=2}
    Let $F = \mathbb{F}_q$ be a finite field with $\textup{char}(\mathbb{F}_q) \neq 2$. Then the differentials in the $\textup{\textbf{mASS}}_2^{\mathbb{F}_q}(BPGL \langle 0 \rangle)$ at the prime $p=2$ have the following behavior. For $q \equiv 1 \, (4)$, all nontrivial differentials are generated under the Liebniz rule by
    \[d_{\nu_2(q-1)+s}(\tau^{2^s}) = u \tau^{2^{s}-1}v_0^{\nu_2(q-1)+s}, \quad s \geq 0 .\]
    For $q \equiv 3 \, (4)$, all differentials are generated under the Liebniz rule by
    \[d_{\nu_2(q^2-1)+s}(\tau^{2^s}) = \rho\tau^{2^s-1}v_0^{\nu_2(q^2-1)+s}, \quad s \geq 1.\]
\end{proposition}

\begin{proof}
    The $E_2$-page of the $\textbf{mASS}^{\mathbb{F}_q}(BPGL \langle 0 \rangle)$ was determined by Kylling in \cite[Theorem 4.1.2, Theorem 4.1.3]{Kylingkqfinite} and takes the form
    \[\text{Ext}^{s,f,w}_{\mathcal{E}(0)^\vee_2}(\mathbb{M}_2^{\mathbb{F}_q}) \cong \left\{ \begin{array}{lc}
    \mathbb{F}_2[u, \tau, v_0]/(u^2) & q \equiv 1 \, (4) \\
    \mathbb{F}_2[\rho, \tau^2, v_0, \rho\tau]/(\rho^2, \rho v_0, \rho(\rho\tau), (\rho\tau)^2) & q \equiv 3 \, (4)
    \end{array}\right..\]
    For degree reasons, there can be no differential on a class which is not divisible by $\tau$. For a class $\tau^nx$ where $x$ is not divisible by $\tau$, the Liebniz rule gives a formula for any differential:
    \[d_r(\tau^nx) = d_r(\tau^n)x + \tau^nd_r(x) = d_r(\tau^n)x.\]
    Thus all differentials are determined by their values on powers of $\tau$. 
    Recall that the spectrum $H\mathbb{Z}$ represents motivic cohomology with integer coefficients \cite{Spitzweck18}. In particular, there is an isomorphism
    \[\pi_{s,w}^{\mathbb{F}_q}(H\mathbb{Z}) \cong H^{-s,-w}(\mathbb{F}_q; \mathbb{Z})_2^\wedge, \]
    where the right hand side is 2-completed.
    These motivic cohomology groups were calculated by Soul\'e to be \cite{Soule79}:
    \begin{equation}
    \label{eq:FinitefieldMotivicCohomology2Complete}
    H^{-s,-w}(\mathbb{F}_q;\mathbb{Z})_2^\wedge = \left\{\begin{array}{rl}
    \mathbb{Z}_2 & s=w=0; \\
    \mathbb{Z}/(q^{-w}-1)_2 & s=-1, w \leq -1;\\
    0 & \text{else}.
    \end{array}\right.
    \end{equation}    
    This implies that $\pi_{0, -w}^{\mathbb{F}_q}(H\mathbb{Z}) = 0$ for $w \geq 1$, thus all powers of $\tau$ must be killed in the spectral sequence. The particular formulas given are now forced by the groups given in \Cref{eq:FinitefieldMotivicCohomology2Complete}. For example, take $q=3$. We have that $\pi_{-1, -2}^{\mathbb{F}_3}(H\mathbb{Z}) \cong \mathbb{Z}/8$, which implies that the class $\rho \tau v_0^4$ is the class in stem $-1$ and weight $-2$ of lowest Adams filtration which must be the target of a differential. Since $\tau^2$ must die, this forces the differential $d_{4}(\tau^2) = \rho \tau v_0^4$. The Liebniz rule then determines all $d_4$-differentials on $\tau^{2n}$ for $n$ odd. Similarly, we have $\pi_{-1, -3}^{\mathbb{F}_3}(H\mathbb{Z}) \cong \mathbb{Z}/2$, so we must have that $d_{5}(\tau^4) = \rho\tau^3v_0^5$, and the Liebniz rule determines all $d_5$-differentials on $\tau^{4n}$ for $n$ odd.
\end{proof}

The $E_2$-pages described in \Cref{BPGL mass finite field p=2} are depicted in \Cref{Ext_E0_Fq_p2}.
\begin{figure}[ht]
    \centering
    \includegraphics[width=.5\linewidth]{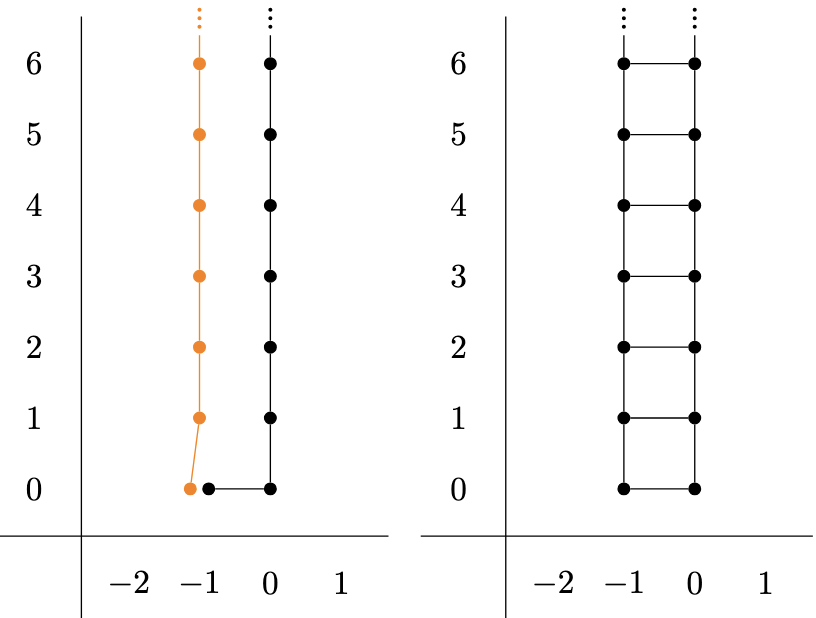} 
    \caption{Charts for $\text{Ext}^{s,f,w}_{\mathcal{E}(0)^\vee_2}(\mathbb{M}_2^{\mathbb{F}_q})$. The left chart represents the case where $q \equiv 3 \, (4)$, where a $\bullet$ denotes $\mathbb{F}_2[\tau^2]$ and an orange class denotes $\rho\tau$-divisibility. The right chart represents the case where $q \equiv 1 \, (4)$, where a $\bullet$ denotes $\mathbb{F}_2[\tau]$. A vertical line represents $v_0$ multiplication, and a horizontal line represents $\rho$-multiplication in the left chart and $u$-multiplication in the right chart.}
    \label{Ext_E0_Fq_p2}
\end{figure}

\begin{proposition}
\label{mass BPGL0 finite field p>2}
    Let $F = \mathbb{F}_q$ be a finite field with $\textup{char}(\mathbb{F}_q) \neq p$. Then the differentials in the $\textup{\textbf{mASS}}_p^{\mathbb{F}_q}(BPGL \langle 0 \rangle)$ at a prime $p >2$ have the following behavior. Let $i$ be the smallest positive integer such that $p \,|\, q^i-1$. For $q^i \equiv 1 \, (p^2)$, all nontrivial differentials are generated under the Liebniz rule by
    \[d_{\nu_p(q^i-1)+s}(\zeta^{p^s}) \doteq u \zeta^{p^s-1}v_0^{\nu_p(q^i-1)+s},\quad  s \geq 0.\]
    For $q^i \not\equiv 1 \, (p^2)$, all nontrivial differentials are generated under the Liebniz rule by
    \[d_{\nu_p(q^{pi}-1)+s}(\zeta^{p^s})\doteq \gamma\zeta^{p^s-1}v_0^{\nu_p(q^{pi}-1)+s},\quad  s \geq 1.\]
\end{proposition}

\begin{proof}
    This follows in the same way as in \cref{BPGL mass finite field p=2}. The $E_2$-page of the $\textbf{mASS}^{\mathbb{F}_q}_p(BPGL \langle 0 \rangle)$ was determined by Wilson in \cite[Proposition 6.4, Proposition 7.7]{Wilson16}. For $q^i \equiv 1 \, (p^2)$, we have
    \[E_2^{s,f,w}=\text{Ext}^{s,f,w}_{\mathcal{E}(0)^\vee_p}(\mathbb{M}_p^{\mathbb{F}_q}) \cong        \mathbb{F}_p[u, \zeta, v_0]/(u^2), \]
    and for $q^i \not\equiv 1 \, (p^2)$, we have
    \[E_2^{s,f,w}=\text{Ext}_{\mathcal{E}(0)^\vee_p}^{s,f,w}(\mathbb{M}_p^{\mathbb{F}_q}) \cong \frac{\mathbb{F}_p[\gamma, \zeta^p, v_0, \gamma\zeta, \gamma\zeta^2, \dots, \gamma\zeta^{p-1}]}{(\gamma^2, \gamma v_0, \gamma(\gamma\zeta^j), (\gamma\zeta^j)(\gamma\zeta^k) : 1 \leq j, k \leq p-1)}.\]
    The multiplicative structure in the $q^i \not\equiv 1 \, (p^2)$ is less opaque than it appears at first glance. In terms of Adams charts, there is a $v_0$-tower in stem 0 which supports a single $\gamma$-multiplication on the generator and there are $p-1$-many $v_0$-towers in stem $-1$ which admit no further multiplication. Recall that the homotopy of $H\mathbb{Z}$ is determined by the motivic cohomology groups, which were determined to be \cite{Soule79}:
    \begin{equation}
    \label{eq:FinitefieldMotivicCohomologypComplete}
    H^{-s,-w}(\mathbb{F}_q;\mathbb{Z})_p^\wedge = \left\{\begin{array}{rl}
    \mathbb{Z}_p & s=w=0; \\
    \mathbb{Z}/(q^{-w}-1)_p & s=-1, w \leq -1;\\
    0 & \text{else}.
    \end{array}\right.
    \end{equation}       
    As in the proof of \Cref{BPGL mass finite field p=2}, the differentials are forced for degree reasons: no power of $\zeta$ may survive to the $E_\infty$-page, and so the particular formulas for the differentials are the only ones possible to get the groups described in \Cref{eq:FinitefieldMotivicCohomologypComplete}.
\end{proof}

The $E_2$-page for the case of $q^i \equiv 1 \, (p^2)$ is depicted in the right side of \Cref{Ext_E0_Fq_p2}, where a $\bullet$ denotes $\mathbb{F}_p[\zeta]$ and a horizontal line denotes $u$-multiplication. The $E_2$-page for the case of $q^i \not \equiv 1 \, (p^2)$ when $p=5$ is depicted in \Cref{Ext_E0_p5_hard}.

\begin{figure}[ht]
    \centering
    \includegraphics[width=0.25\linewidth]{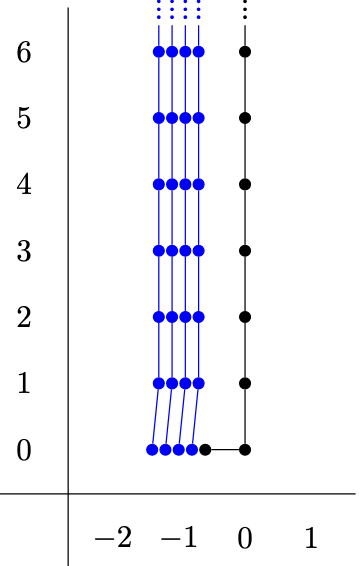}
    \caption{Charts for $\text{Ext}^{s,f,w}_{\mathcal{E}(0)^\vee_p}(\mathbb{M}_p^{\mathbb{F}_q})$ with a nontrivial Bockstein action on $\mathbb{F}_q$ and for $p=5$. A $\bullet$ denotes $\mathbb{F}_5[\zeta^5]$, and a blue class denotes $\mathbb{F}_5\{\gamma\zeta^j\}[\zeta^5]$ for some $1 \leq j \leq 4$. A vertical line represents $v_0$-multiplication, and a horizontal line represents $\gamma$-multiplication. Each blue $v_0$-tower is in stem -1. The case of a trivial Bockstein action at any odd prime is depicted in the left chart of \cref{Ext_E0_Fq_p2}, where a $\bullet$ denotes $\mathbb{F}_p[\zeta].$}
    \label{Ext_E0_p5_hard}
\end{figure}

\subsection{The homotopy ring of $BPGL \langle 0 \rangle$-cooperations}
Next, we compute the homotopy of $BPGL\langle 0 \rangle \wedge BPGL \langle 0 \rangle.$

\begin{theorem}
\label{thm:BPGL0homotopyringcoops}
    Let $F \in \{\mathbb{C}, \mathbb{R}, \mathbb{F}_q\}$, where $\textup{char}(\mathbb{F}_q) \neq p$. Then there is an isomorphism for all primes $p$:
    \[\pi^F_{*,*}(BPGL \langle 0 \rangle \wedge BPGL \langle 0 \rangle) \cong \pi^F_{*,*}(BPGL \langle 0 \rangle) \oplus W,\]
    where $W$ is a sum of shifts of $\mathbb{M}_p$.
\end{theorem}

\begin{proof}
    We will show this by analyzing the $\textbf{mASS}_p(BPGL \langle 0 \rangle \wedge BPGL \langle 0 \rangle)$. We saw in \cref{e2-mass BPGL0 coop} that this spectral sequence has signature
    \[E_2=\text{Ext}^{s,f,w}_{\mathcal{E}(0)_p^\vee}(\mathbb{M}_p) \oplus \bigoplus_{k \geq 1}W_k \implies \pi_{s,w}(BPGL \langle 0 \rangle \wedge BPGL \langle0 \rangle).\]
    By \cref{BPGL0 coop mass differentials}, the only differentials in this spectral sequence are those contained in the summand $\text{Ext}^{s,f,w}_{\mathcal{E}(0)^\vee_p}(\mathbb{M}_p)$. These were analyzed in the previous section in \cref{mass BPGL0 C and R}, \cref{BPGL mass finite field p=2}, and \cref{mass BPGL0 finite field p>2}. Since each component of $\bigoplus_{k \geq 1}W_k$ contributes a free summand of $\mathbb{M}_p$, the result follows.
\end{proof}

\subsection{Constructing the spectrum-level splitting} \label{subsec:BPGL0relASS}

Now, we use the relative Adams spectral sequence to lift this algebraic splitting to a spectrum-level splitting.

\begin{theorem}
\label{thm:bpgl0coopalgebra}
    For any prime $p$ and for any $F \in \{\mathbb{C}, \mathbb{R}, \mathbb{F}_q\}$, there is an equivalence of spectra (up to p-completion):
    \[BPGL \langle 0 \rangle \wedge BPGL \langle 0 \rangle \simeq BPGL \langle 0 \rangle \vee V,\]
    where $V$ is a wedge of suspensions of $H\mathbb{F}_p$.
\end{theorem}

\begin{proof}
       Consider the relative Adams spectral sequence
\begin{align*}
     E_2^{s,f,w} \cong \Ext_{\cE(0)^\vee_p}^{s,f,w}& \left( H_{*,*}^{BPGL\langle 0 \rangle} \left(BPGL\langle 0 \rangle \vee V \right) , \, H_{*,*}^{BPGL\langle 0 \rangle}\left(BPGL\langle 0 \rangle \wedge BPGL\langle 0 \rangle \right) \right) \\\implies & [BPGL\langle 0 \rangle \vee V , BPGL\langle 0 \rangle \wedge BPGL\langle 0 \rangle]^{BPGL\langle 0 \rangle}_{(s,w)}. 
\end{align*}

The isomorphism of \cref{prop:homology splitting height zero} is represented by a class in degree $\theta \in E_{2}^{0,0,0}$. If $\theta$ survives the spectral sequence, then it lifts to a homotopy equivalence (up to $p$-completion) $\widetilde{\theta}: BPGL\langle 0 \rangle \vee V \to BPGL\langle 0 \rangle \wedge BPGL\langle 0 \rangle$. Recall from the proof of \cref{prop:homology splitting height zero} that 
\[
H_{*,*}^{BPGL\langle 0 \rangle}\left(BPGL\langle 0 \rangle \vee V \right) \cong \mM_{p}^{F} \oplus\bigoplus\limits_{i \in I}\Sigma^{\alpha(i)}\mathcal{E}(0)^\vee_p,
\]
where $\bigoplus\limits_{i \in I}\Sigma^{\alpha(i)}\mathcal{E}(0)^\vee_p$ is a finite-type sum of suspensions of $\cE(0)_p^{\vee}$. So we can rewrite the $E_2$-page as  \[E_2^{s,f,w} \cong \Ext^{s,f,w}_{\cE(0)_p^{\vee}}\left(\mM_{p}^{F} \oplus \bigoplus\limits_{i \in I}\Sigma^{\alpha(i)}\mathcal{E}(0)_p^\vee, \mM_{p}^{F} \oplus \bigoplus\limits_{i \in I}\Sigma^{\alpha(i)}\mathcal{E}(0)_p^\vee\right).\]

Since $\mathcal{E}(0)^\vee_p$ is free and injective over itself,
\[E_2^{s,f,w} \cong \Ext^{s,f,w}_{\cE(0)_p^{\vee}}\left(\mM_{p}^{F}, \mM_{p}^{F}\right) \bigoplus\limits_{j \in J} \Sigma^{\beta(j)}\mM_{p}^{F}.\]

Note that the second summand, 
\[
\bigoplus\limits_{j \in J} \Sigma^{\beta(j)}\mM_{p}^{F},
\] 
is concentrated in filtration $f=0$, so no differentials are possible within that summand. Furthermore, this summand is $v_0$-torsion, while the  first summand 
\[
\Ext^{s,f,w}_{\cE(0)_p^{\vee}}\left(\mM_{p}^{F}, \mM_{p}^{F}\right)
\]
is $v_0$-torsion free above the line $f=0$, so no differentials can go from the second summand to the first. Therefore the only potential differentials must occur within the first summand. By the same arguments as in the analysis of the Adams spectral sequence for computing $\pi_{*,*}BPGL\langle 0 \rangle$, no differentials can originate from degree $(s,f,w) = (0,0,0)$ in the first summand, so indeed $\theta$ survives the spectral sequence.
\end{proof}

\begin{remark}
    \rm One may use the computations in this section to compute the $E_1$-page of the $BPGL \langle 0 \rangle$-motivic Adams spectral sequence. In recent work, Burklund--Pstragowski \cite{BurPst25} study the relationship between the classical $BP \langle 0\rangle$ and $H\mathbb{F}_p$-Adams spectral sequences. For example, they show that the quotient map $BP \langle 0 \rangle \to H\mathbb{F}_p$ induces a map of spectral sequences that allows one to identify the $E_2$-page of the $BP \langle 0 \rangle$-Adams spectral sequence for the sphere:
    \[
    ^{BP \langle 0 \rangle}E_2^{s,f} = \left\{\begin{array}{rl}
        \mathbb{Z}_{(p)} &  s=0,f=0\\
        0 & s=0, f \neq 0\\
        \text{Ext}_{\mathcal{A}^\vee_p}^{s,f}(\mathbb{F}_p, \mathbb{F}_p) & s \neq 0.
    \end{array} \right.
    \]
    Our splitting indicates that an analogous statement should hold for the $E_2$-page of the $BPGL \langle 0 \rangle$-motivic Adams spectral sequence. 
    
\end{remark}

\section{Splitting $BPGL\langle 1 \rangle\wedge BPGL \langle 1 \rangle$}
\label{sec:BPGL1}
In this section, we produce spectrum-level splittings of the cooperations algebra $BPGL \langle 1 \rangle \wedge BPGL \langle 1 \rangle$. We also compute the homotopy ring of cooperations $\pi_{*,*}(BPGL \langle 1 \rangle \wedge BPGL \langle 1 \rangle)$. Interestingly, we show that the differentials in the motivic Adams spectral sequence computing the homotopy ring of cooperations are completely determined by $v_1$-linearity and the integral motivic cohomology of the base field, which was discussed in the previous section. As an application, we deduce the $E_1$-page of the $BPGL\langle 1 \rangle$-motivic Adams spectral sequence. Specifically, we describe the $n$-line in terms of a splitting of spectra, and in terms of homotopy groups, as a module over the 0-line. Recall that at the prime $2,$ a model for $BPGL \langle 1 \rangle$ is the (2-local) effective algebraic $K$-theory spectrum $kgl$, while at odd primes, a model for $BPGL \langle 1 \rangle$ is the ($p$-local) connective motivic Adams summand $m\ell$ \cite{NauSpiOst15}.

\subsection{Strategy}
In this section, we outline our strategy for computing the homotopy ring of cooperations. The spectrum-level splitting of the cooperations algebra is more involved, and we defer our overview for the relative Adams spectral sequence arguments until \Cref{subsec:BPGL1relASS}. All results here are stated independent of base field unless otherwise indicated. 

We will begin by computing the homotopy groups $\pi_{*,*}(BPGL \langle 1 \rangle)$ by a motivic Adams spectral sequence. The $\textbf{mASS}_p(BPGL \langle 1 \rangle)$ takes the form
\[\text{Ext}^{s,f,w}_{\mathcal{A}_p^\vee}(H_{*,*}(BPGL \langle 1 \rangle)) \implies \pi_{s,w}(BPGL \langle 1 \rangle).\]
Recall from \cref{prop:BPGLnHomology} that there is an isomorphism of $\mathcal{A}^\vee_p$-comodules:
\[H_{*,*}(BPGL \langle 1 \rangle) \cong \mathcal{A}^\vee_p\underset{\mathcal{E}(1)^\vee_p}{\Box}\mathbb{M}_p.\]
Using a change-of-rings isomorphism, the $E_2$-page can then be rewritten as 
\[E_2 = \text{Ext}^{s,f,w}_{\mathcal{A}^\vee_p}(\mathcal{A}^\vee_p\underset{\mathcal{E}(1)^\vee_p}{\Box}\mathbb{M}_p) \cong \text{Ext}^{s,f,w}_{\mathcal{E}(1)^\vee_p}(\mathbb{M}_p).\]

After computing these homotopy groups, we will compute the cooperations algebra by another motivic Adams spectral sequence. The $\textbf{mASS}_p(BPGL \langle 1 \rangle \wedge BPGL \langle 1 \rangle)$ has signature
\[E_2 = \text{Ext}^{s,f,w}_{\mathcal{A}_p^\vee}(H_{*,*}(BPGL \langle 1 \rangle \wedge BPGL \langle 1 \rangle)) \implies \pi_{s,w}(BPGL \langle 1 \rangle \wedge BPGL \langle 1 \rangle).\]
Using the Kunneth isomorphism of \cref{prop:BPGLnKunneth}, the change-of-rings isomorphism, the Brown-Gitler decomposition of \cref{prop:BGdecomposition}, and the $\mathcal{E}(1)_p^\vee$-comodule isomorphism of \cref{prop:LightningFlashBGiso}, we can rewrite the $E_2$-page as
\begin{equation}
\label{bpgl1 coop e2}
\text{Ext}^{s,f,w}_{\mathcal{A}_p^\vee}(H_{*,*}(BPGL \langle 1 \rangle \wedge BPGL \langle 1 \rangle)) \cong \bigoplus_{k \geq 0}\Sigma^{2k(p-1), k(p-1)}\text{Ext}^{s,f,w}_{\mathcal{E}(1)^\vee_p}(L_p(\nu_p(k!)) \oplus W_k),
\end{equation}
where $W_k$ is a wedge of suspensions of $\mathbb{M}_p$.

Determining the differentials in the $\textbf{mASS}_p(BPGL \langle 1 \rangle \wedge BPGL \langle 1 \rangle)$ is slightly more involved. We will first show that the differentials in the $\textbf{mASS}_p(BPGL \langle 1 \rangle)$ are determined by $v_1$-linearity combined with the differentials in the $\textbf{mASS}_p(BPGL \langle 0 \rangle)$. Then, we will use the description of the $E_2$-page in \Cref{bpgl1 coop e2} to lift differentials and compute the homotopy ring of cooperations.

\subsection{$BPGL \langle 1 \rangle$ coefficients}
We begin by computing the homotopy of $BPGL \langle 1 \rangle.$

\begin{proposition} 
\label{C and R BPGL1}
    Let $F = \mathbb{C}$ or $\mathbb{R}$. Then the motivic Adams spectral sequence for $BPGL \langle 1 \rangle$ collapses on the $E_2$-page, giving isomorphisms
    \[\pi_{*,*}^\mathbb{C}(BPGL \langle 1 \rangle) \cong \mathbb{Z}_p[\tau, v_1]\]
    for any prime $p$ and
    \[ \pi_{*,*}^\mathbb{R}(BPGL \langle 1 \rangle) \cong\left\{\begin{array}{ll}
       \mathbb{Z}_2[\rho, \tau^4, 2\tau^2, v_1]/(2\rho, v_1\rho^3)  & p=2  \\
       \mathbb{Z}_p[\theta, v_1]  & p>2.
    \end{array} \right.\]
\end{proposition}

\begin{proof}
    Over $\mathbb{C}$ for all primes and over $\mathbb{R}$ at odd primes, similar to the case of $BPGL \langle 0 \rangle$, the $E_2$-page is the base change of the classical Adams spectral sequence for $BP\langle 1 \rangle$ from $\mathbb{F}_p$ to $\mathbb{M}^{F}_p$. Thus, the $E_2$-page over $\mathbb{C}$ at all primes is given by
    \[\text{Ext}^{s,f,w}_{\cE(1)_p^\vee}(\mathbb{M}_p^\mathbb{C}) \cong \mathbb{F}_p[\tau, v_0, v_1],\]
    and the $E_2$-page over $\mathbb{R}$ for odd primes is given by
    \[\text{Ext}^{s,f,w}_{\cE(1)_p^\vee}(\mathbb{M}_p^\mathbb{R}) \cong \mathbb{F}_p[\theta, v_0, v_1],\]  
    where $|v_1| = (2(p-1), 1, p-1)$. Over $\mathbb{R}$ and at the prime $p=2$, the $E_2$-page was computed in \cite[Theorem 3.1]{Hill11} and shown to be
    \[\text{Ext}^{s,f,w}_{\cE(1)_2^\vee}(\mathbb{M}_2^\mathbb{R}) \cong \mathbb{F}_2[\rho, \tau^4, v_0, \tau^2v_0, v_1]/\left(\rho v_0, \rho^3 v_1, (\tau^2v_0)^2\right) = \tau^4v_0^2.\]
    There can be no differentials in any of these spectral sequences for degree reasons, so the spectral sequences must all collapse. The result follows.
\end{proof}

The $E_2$-pages described in \Cref{C and R BPGL1} are depicted in \Cref{Ext_E1_C} and \Cref{Ext_e1_R}.

\begin{figure}[ht]
    \centering
    \includegraphics[width=0.75\linewidth]{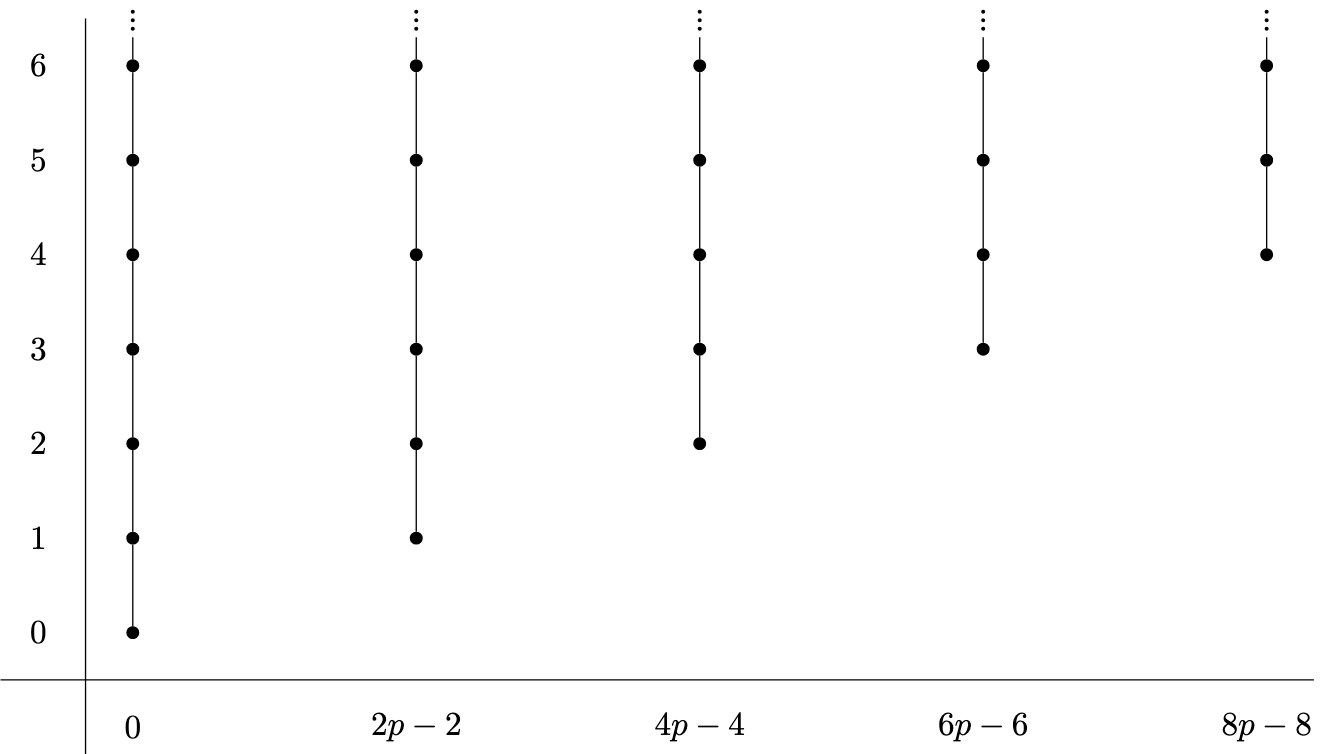}
    \caption{Charts for $\text{Ext}^{s,f,w}_{\mathcal{E}(1)^\vee_p}(\mathbb{M}_p^F,\mathbb{M}_p^F)$ for $F = \mathbb{C}$ at all primes or $\mathbb{R}$ at odd primes. For $F=\mathbb{C}$, a $\bullet$ denotes $\mathbb{F}_p[\tau]$, and for $F=\mathbb{R}$, a $\bullet$ denotes $\mathbb{F}_p[\theta]$. A vertical line indicates $v_0$-multiplication, while $v_1$-multiplication is suppressed.}
    \label{Ext_E1_C}
\end{figure}

\begin{figure}[ht]
    \centering
    \includegraphics[width=0.75\linewidth]{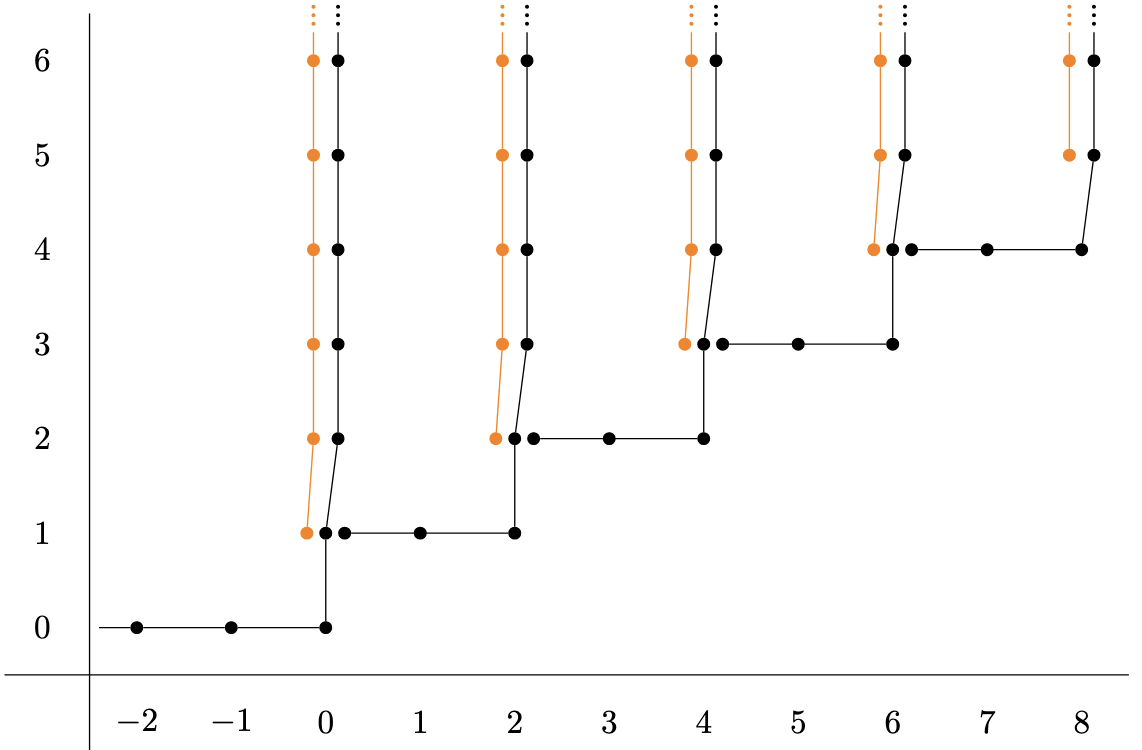}
    \caption{Charts for $\text{Ext}^{s,f,w}_{\mathcal{E}(1)^\vee_2}(\mathbb{M}_2^\mathbb{R},\mathbb{M}_2^{\mathbb{R}}).$ A $\bullet$ denotes $\mathbb{F}_2[\tau^4]$ and an orange class denotes $\tau^2v_0$-divisibility. A vertical line represents $v_0$-multiplication and a horizontal line represents $\rho$-multiplication, while $v_1$-multiplication is suppressed.}
    \label{Ext_e1_R}
\end{figure}

Before determining the coefficients of $BPGL\langle 1 \rangle$ over finite fields, we make a few preliminary calculations.

\begin{lemma}
\label{no bockstein BPGL1 Mass E2}
    Let $p$ be an odd prime and let $\mathbb{F}_q$ be a finite field such that $\textup{char}(\mathbb{F}_q) \neq p$ and there is a trivial Bockstein action. Then there is an isomorphism
    \[\textup{Ext}^{s,f,w}_{\mathcal{E}(1)^\vee_p}(\mathbb{M}_p^{\mathbb{F}_q}) \cong \mathbb{F}_p[u, \zeta, v_0, v_1]/(u^2).\]
\end{lemma}

\begin{proof}
    The dual $\mathbb{F}_q$-motivic Steenrod algebra is the base change of the classical dual Steenrod algebra $A^\vee_p$ to $\mathbb{M}_p^{\mathbb{F}_q}$:
    \[\mathcal{A}^\vee_p = A^\vee_{p} \otimes \mathbb{F}_p[u, \zeta]/(u^2).\]
    Note that this isomorphism also holds true for $\mathcal{E}(1)^\vee_p$. This observation gives rise to a natural map of Hopf algebroids
    \[(\mathbb{F}_p, E(1)^\vee_{p}) \to (\mathbb{M}_p^{\mathbb{F}_q}, \mathcal{E}(1)_p^\vee)\]
    which is just base change to $\mathbb{M}_p^{\mathbb{F}_q}$. By the change-of-rings theorem for Hopf algebroid extensions \cite[A1.3.12]{Rav86}, we have an isomorphism
    \[\text{Ext}^{s,f,w}_{\mathcal{E}(1)^\vee_p}(\mathbb{M}_p^{\mathbb{F}_q}) \cong \text{Ext}^{s,f,w}_{E(1)^\vee_{p}}(\mathbb{F}_p) \otimes \mathbb{M}_p^{\mathbb{F}_q} \cong \mathbb{F}_p[u, \zeta, v_0, v_1]/(u^2)\]
    The particular degrees of the generators can be obtained by an analysis of the cobar complex $\mathcal{C}_{\mathcal{E}(1)^\vee_p}(\mathbb{M}_p^{\mathbb{F}_q})$. 
\end{proof}

\begin{lemma}
\label{bockstein BPGL1 Mass E2}
    Let $p$ be an odd prime and let $\mathbb{F}_q$ be a finite field such that $\textup{char}(\mathbb{F}_q) \neq p$ and there is a nontrivial Bockstein action. Then there is an isomorphism
    \[
    \textup{Ext}^{s,f,w}_{\mathcal{E}(1)^\vee_p}(\mathbb{M}_p^{\mathbb{F}_q}) \cong \frac{\mathbb{F}_p[\gamma, \zeta^p, v_0, v_1, \gamma\zeta, \gamma \zeta^2, \dots, \gamma \zeta^{p-1}]}{(\gamma^2, \gamma v_0, \gamma(\gamma \zeta^j), (\gamma\zeta^j)(\gamma\zeta^k) : 1 \leq j, k \leq p-1)}.
    \]
\end{lemma}

\begin{proof}
    This follows in a similar fashion to \cite[Proposition 7.7]{Wilson16}. We may calculate this $\Ext$ group by a $\gamma$-Bockstein spectral sequence. To be precise, consider the cobar complex $\mathcal{C}_{\mathcal{E}(1)^\vee_p}(\mathbb{M}_p^{\mathbb{F}_q})$ whose cohomology computes $\text{Ext}^{s,f,w}_{\mathcal{E}(1)^\vee_p}(\mathbb{M}_p^{\mathbb{F}_q})$. We may filter this by powers of the class $\gamma \in \mathbb{M}_p^{\mathbb{F}_q}$. Since $\gamma^2=0$, this is a short filtration, and has an associated short exact sequence of the form
    \[0 \to \gamma \cdot\mathcal{C}_{\mathcal{E}(1)^\vee_p}(\mathbb{M}_p^{\mathbb{F}_q}) \to \mathcal{C}_{\mathcal{E}(1)^\vee_p}(\mathbb{M}_p^{\mathbb{F}_q}) \to \mathcal{C}_{\mathcal{E}(1)^\vee_p}(\mathbb{M}_p^{\mathbb{F}_q})/\gamma \to 0.\]
    Notice that this filtration supplies us with isomorphisms of Hopf algebroids:
    \[\gamma \cdot\mathcal{C}_{\mathcal{E}(1)^\vee_p}(\mathbb{M}_p^{\mathbb{F}_q}) \cong \mathcal{C}_{\mathcal{E}(1)^\vee_p} ( \mathbb{M}_p^{\mathbb{F}_q} ) /\gamma \cong \mathbb{F}_p[\zeta] \otimes \mathcal{C}_{E(1)^\vee_{p,}}(\mathbb{F}_p).\] 
    This implies that $E_1$-page of the $\gamma$-\textbf{BSS} takes the form
    \[E_1 = \Big( \mathbb{F}_p[\zeta] \otimes\text{Ext}^{s,f}_{E(1)^\vee_{p}} ( \mathbb{F}_p) \Big) [\gamma]/(\gamma^2) \cong \mathbb{F}_p[\gamma, \zeta, v_0,v_1]/(\zeta^2).\]
    Since $\gamma^2=0$, the $d_1$-differential is the only possible nontrivial differential. For degree reasons, the differential must be trivial on classes which are not divisible by $\zeta$. The value on $\zeta$ is given by $d_1(\zeta) \doteq \gamma v_0$. The Liebniz rule forces all other differentials. Notice that as we are working over a field of characteristic $p$, we have that $d_1(\zeta^k)=0$ for any $k \geq 0$ such that $p \,|\, k$. The classes $\gamma \zeta, \gamma\zeta^2, \dots, \gamma \zeta^{p-1}$ all survive to the $E_\infty$-page with multiplicative structure leftover from the $E_1$-page, giving the result.
\end{proof}

\begin{proposition}
\label{bpgl1 difs finite fields}
    Let $\textup{{char}}(\mathbb{F}_q) \neq p$. Then the differentials in the $\textup{\textbf{mASS}}_p^{\mathbb{F}_q}(BPGL \langle 1 \rangle)$ are determined by the differentials in the $\textup{\textbf{mASS}}^{\mathbb{F}_q}_p(BPGL \langle 0 \rangle)$. 
\end{proposition}

\begin{proof}
    We start with the case where $p=2$. Then the $E_2$-page of the $\textbf{mASS}^{\mathbb{F}_q}_2(BPGL \langle 1 \rangle)$ was determined by Kylling in \cite[Theorem 4.1.2, Theorem 4.1.4]{Kylingkqfinite} and takes the form 
    \[\text{Ext}^{s,f,w}_{\mathcal{E}(1)^\vee_2}(\mathbb{M}_2^{\mathbb{F}_q}) \cong \left\{ \begin{array}{ll}
    \mathbb{F}_2[u, \tau, v_0, v_1]/(u^2) & q \equiv 1 \, (4) \\
    \mathbb{F}_2[\rho, \tau^2, v_0, v_1, \rho\tau]/(\rho^2, \rho v_0, \rho(\rho\tau), (\rho\tau)^2) & q \equiv 3 \, (4).
    \end{array}\right.\]
    We now follow the argument of \cite[Section 4.2]{Kylingkqfinite}. As in the case of $BPGL \langle 0 \rangle$, for degree reasons there can be no differential on a class which is not divisible by $\tau$, and so the Liebniz rule implies that all differentials are determined by their value on powers of $\tau$. Recall from \Cref{sec:motivic prelim} (\cref{eq:BPGLcofib}) that there is a cofiber sequence
    \[\Sigma^{2,1}BPGL \langle 1 \rangle \to BPGL \langle 1 \rangle \to BPGL \langle 0 \rangle.\]
    The reduction map $BPGL \langle 1 \rangle \to BPGL \langle 0 \rangle$ induces a map of spectral sequences, which is realized on $E_2$-pages as
    \[\text{Ext}^{s,f,w}_{\mathcal{E}(1)_2^\vee}(\mathbb{M}_2^{\mathbb{F}_q}) \to \text{Ext}^{s,f,w}_{\mathcal{E}(0)_2^\vee}(\mathbb{M}^{\mathbb{F}_q}_2). \]
    Classes in the domain get sent to their namesakes in the target for degree reasons. Since the map of spectral sequences commutes with differentials, the values of the differentials on powers of $\tau$ in the $\textbf{mASS}_2^{\mathbb{F}_q}(BPGL \langle 1 \rangle)$ are determined by the values of the differentials in the $\textbf{mASS}_2^{\mathbb{F}_q}(BPGL \langle 0 \rangle)$. Recall that a model for the spectrum $BPGL \langle 1 \rangle$ is the effective algebraic $K$-theory spectrum $kgl$. This spectrum represents algebraic $K$-theory in that for $s \geq 1$ there is an isomorphism
    \[\pi_{s,w}^{\mathbb{F}_q}(kgl) \cong K_{s-2w}(\mathbb{F}_q)_2.\]
    The algebraic $K$-theory of finite fields was determined by Quillen to be \cite{Quillen73}
    \[K_i(\mathbb{F}_q)_2 \cong \left\{\begin{array}{rl}
        \mathbb{Z}_2 & i=0; \\
        \mathbb{Z}/(q^n-1)_2 & i=2n-1, n \geq 1;\\
        0 & \text{else.}
    \end{array} \right.\]
    In particular, since $\pi_{2n, n}^{\mathbb{F}_q}(kgl) \cong \mathbb{Z}_2,$ we see that all powers of $v_1$ are permanent cycles. This implies that the spectral sequence is both $v_0$ and $v_1$-linear, and thus all differentials are completely determined by their values on powers $\tau$, giving the result.
    
    Now, let $p >2$. the $E_2$-page of the $\textbf{mASS}_p^{\mathbb{F}_q}(BPGL \langle 1 \rangle)$ was calculated above in \cref{no bockstein BPGL1 Mass E2} and \cref{bockstein BPGL1 Mass E2}. Recall that a model for the spectrum $BPGL \langle 1 \rangle$ is the connective motivic Adams summand $m\ell$. Since $\pi_{2(p-1), p-1}^{\mathbb{F}_q}(m\ell) \cong \mathbb{Z}_p$, all powers of $v_1$ must be permanent cycles \cite{NauSpiOst15}. Thus, the differentials in the spectral sequence are completely determined by their values on $\zeta$. The argument now follows exactly as above: the differentials on $\zeta$ determined in \cref{BPGL mass finite field p=2} and \cref{mass BPGL0 finite field p>2} lift to our case, finishing the proof.
\end{proof}

The $E_2$-pages described in \Cref{bpgl1 difs finite fields} are depicted in \Cref{Ext_E1_f3_p2}, \Cref{Ext_E1_fq_easy_podd}, and \Cref{Ext_E1_fq_hard_p5}.

\begin{figure}[ht]
    \centering
    \includegraphics[width=0.75\linewidth]{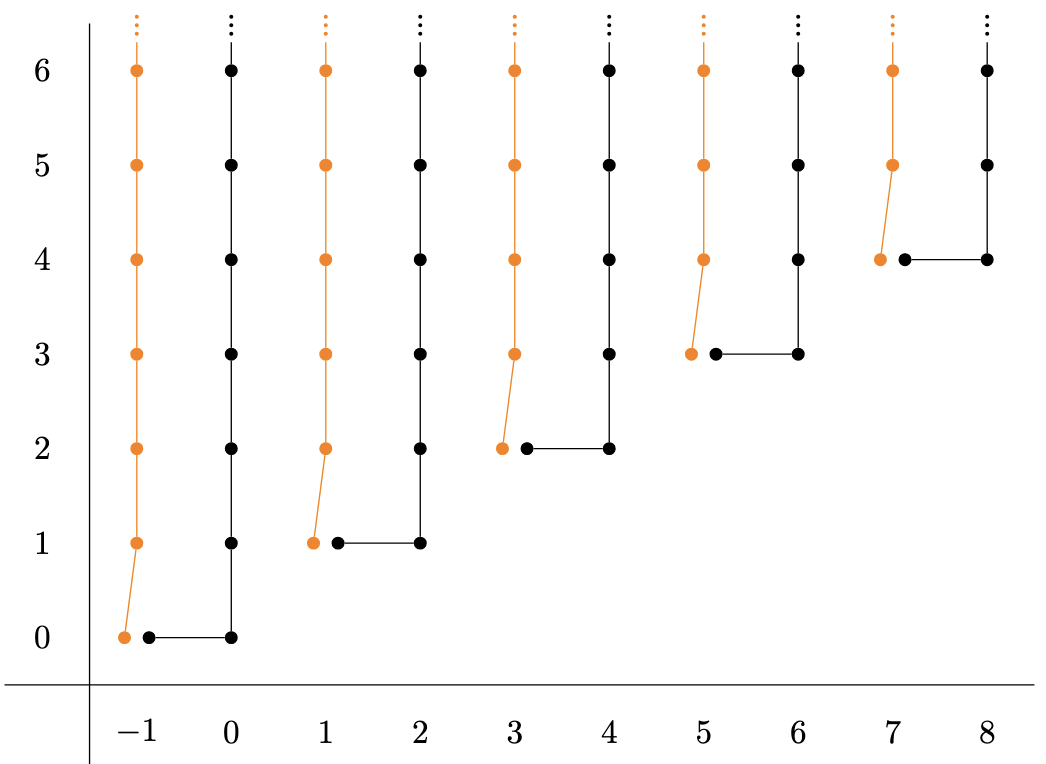}
    \caption{Charts for $\text{Ext}^{s,f,w}_{\mathcal{E}(1)^\vee_2}(\mathbb{M}_2^{\mathbb{F}_q},\mathbb{M}_2^{\mathbb{F}_q})$ for $q\equiv 3 \, (4)$. A $\bullet$ denotes $\mathbb{F}_2[\tau^2]$, and an orange class denotes $\mF_2 \{\rho\tau \} [\tau^2]$. A vertical line represents $v_0$-multiplication and a horizontal line represents $\rho$-multiplication, while $v_1$-multiplication is suppressed.}
    \label{Ext_E1_f3_p2}
\end{figure}

\begin{figure}[ht]
    \centering
    \includegraphics[width=.8\linewidth]{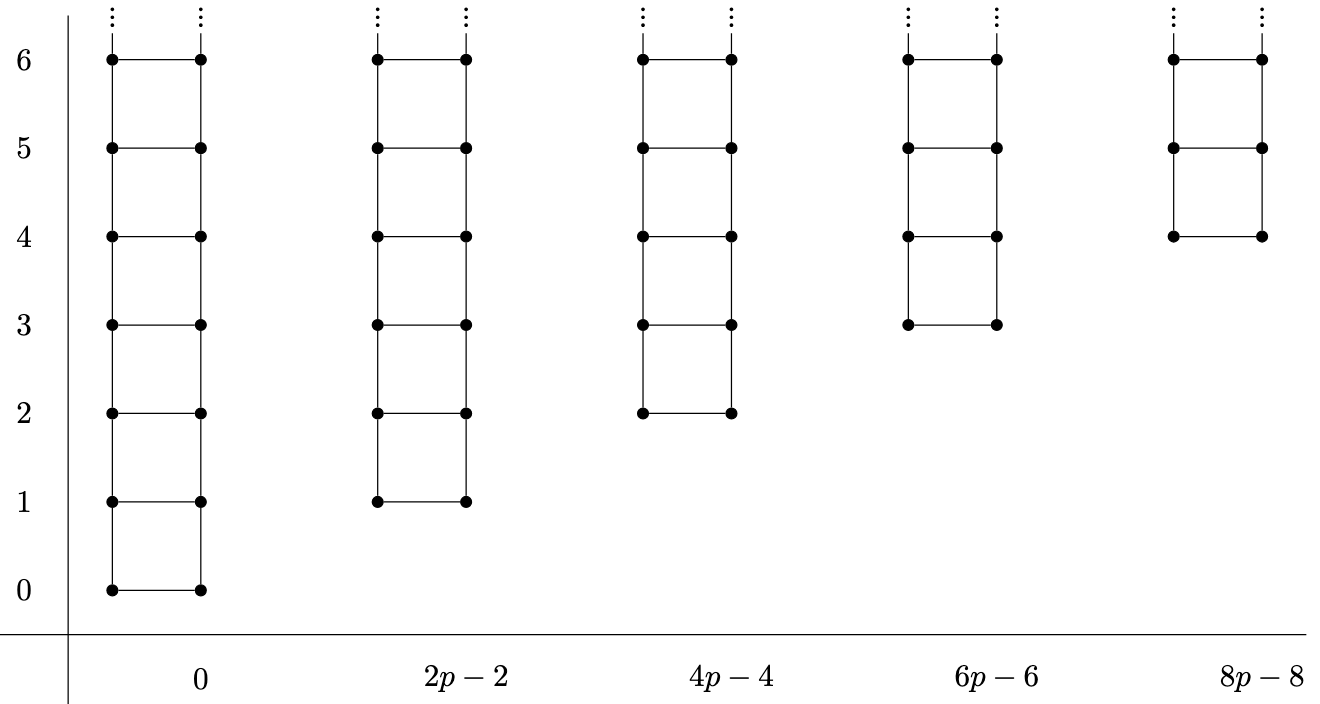}
    \caption{Charts for $\text{Ext}_{\mathcal{E}(1)^\vee_p}^{s,f,w}(\mathbb{M}_p^{\mathbb{F}_q},\mathbb{M}_p^{\mathbb{F}_q})$ where $p$ is any prime with a trivial Bockstein action on $\mathbb{F}_q$. A $\bullet$ denotes $\mathbb{F}_2[\tau]$ in the case of $p=2$ and $\mathbb{F}_p[\zeta]$ for $p$ odd. A vertical line represents $v_0$-multiplication and a horizontal line represents $u$-multiplication, and $v_1$-multiplication is suppressed.}
    \label{Ext_E1_fq_easy_podd}
\end{figure}

\begin{figure}[ht]
    \centering
    \includegraphics[width=.8\linewidth]{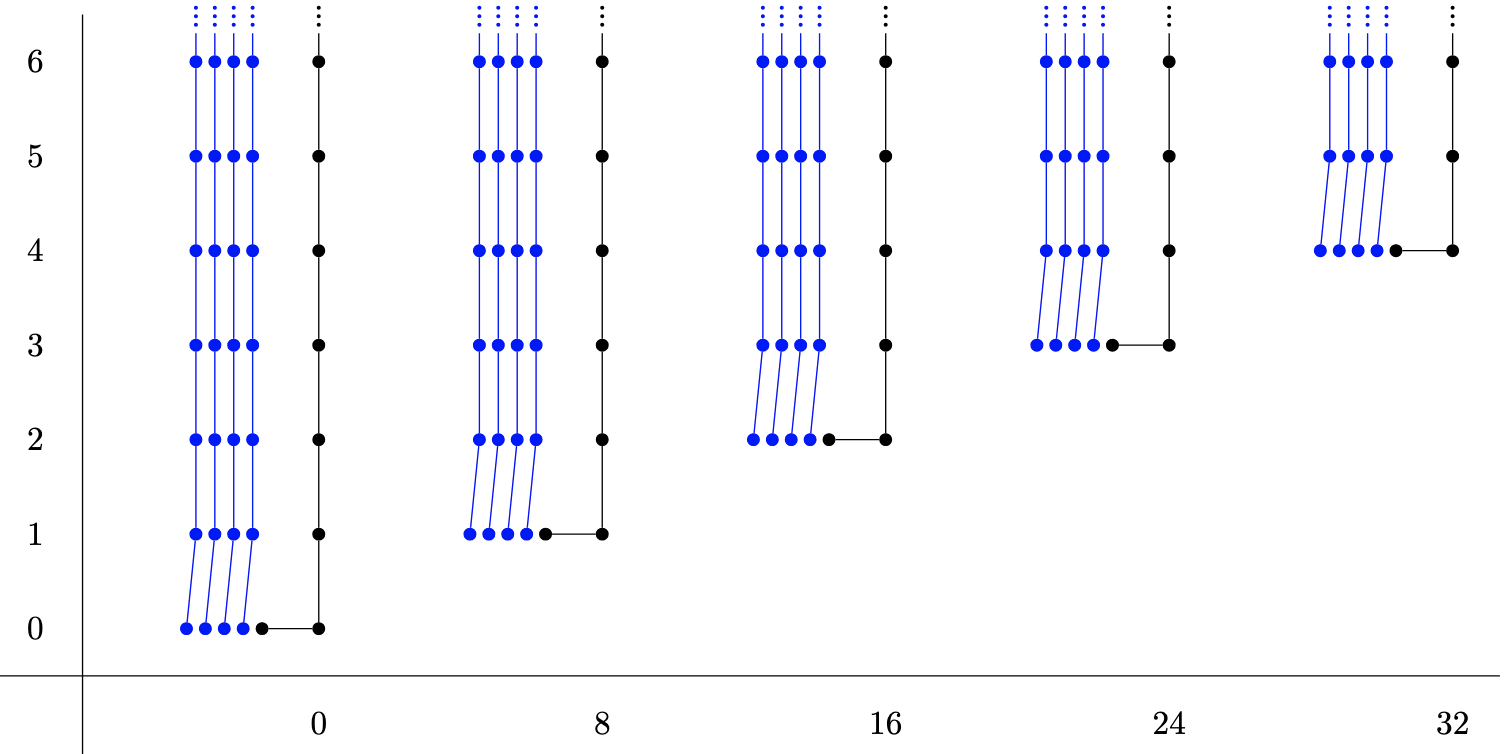}
    \caption{Charts for $\text{Ext}^{s,f,w}_{\mathcal{E}(1)^\vee_p}(\mathbb{M}_p^{\mathbb{F}_q},\mathbb{M}_p^{\mathbb{F}_q})$ for $p=5$ with a nontrivial Bockstein action on $\mathbb{F}_q$. A $\bullet$ denotes $\mathbb{F}_5[\zeta^5]$, and a blue class denotes $\mathbb{F}_5\{\gamma\zeta^j\}[\zeta^5]$ for some $1 \leq j \leq 4$. A vertical line represents $v_0$-multiplication, and a horizontal line represents $\gamma$-multiplication, and $v_1$-multiplication is suppressed.}
    \label{Ext_E1_fq_hard_p5}
\end{figure}



\begin{remark}
    \rm While we are primarily concerned with the spectra $BPGL \langle 0 \rangle$ and $BPGL \langle 1 \rangle$, one may theoretically also deduce results in greater generality, which we outline now. For $F=\mathbb{C}$, the $E_2$-page of the $\textup{\textbf{mASS}}_p^{\mathbb{C}}(BPGL \langle n \rangle)$ for all $n \geq 0$ follows from base change of the classical case to $\mathbb{M}_p^\mathbb{C}$:
    \[E^{s,f,w}_2 = \mathbb{M}_p^\mathbb{C}[v_0, \dots, v_n] \implies \pi_{*,*}^\mathbb{C}(BPGL\langle n \rangle),\]
    where $|v_i| = (2(p^i-1), 1, p^i-1)$.
    For $F=\mathbb{R}$, the same argument describes the $E_2$-page at odd primes, by base-changing to $\mathbb{M}_p^\mathbb{R}.$ At the prime $p=2$, the $E_2$-page of the $\textup{\textbf{mASS}}_2^{\mathbb{R}}(BPGL\langle n \rangle)$ was computed for all $n \geq 0$ by the $\rho$-Bockstein spectral sequence of Hill \cite[Theorem 3.1]{Hill11}. This is the predecessor to the $\gamma\textup{-\textbf{BSS}}$ of \cref{bockstein BPGL1 Mass E2} at the prime 2: by filtering the cobar complex $C_{\mathcal{E}(n)_2^\vee}(\mathbb{M}_2^\mathbb{R})$ by powers of $\rho$, one obtains a spectral sequence
    \[E^{s,f,w}_1 = \textup{Ext}^{s,f,w}_{\mathcal{E}(n)_2^\vee}(\mathbb{M}_2^\mathbb{C})[\rho] \implies\textup{Ext}^{s,f,w}_{\mathcal{E}(n)_2^\vee}(\mathbb{M}_2^\mathbb{R}).\]
    Alternatively, one may also compute this $E_2$-page by a $v_n$-Bockstein spectral sequence
    \[E^{s,f,w}_1 = \textup{Ext}^{s,f,w}_{\mathcal{E}(n-1)_2^\vee}(\mathbb{M}_2^\mathbb{R})[v_n] \implies\textup{Ext}^{s,f,w}_{\mathcal{E}(n)_2^\vee}(\mathbb{M}_2^\mathbb{R}).\]
    These two spectral sequences fit together into the following square, reminiscent of the ones given in \cite{CulKonQui21}.
    \begin{equation}
    \begin{tikzcd}
    & \textup{Ext}^{s,f,w}_{\mathcal{E}(n-1)_2^\vee}(\mathbb{M}_2^\mathbb{C}) \arrow[dl,Rightarrow,"v_n\text{-}\textup{\textbf{BSS}}"'] \arrow[dr,Rightarrow,"\rho\text{-}\textup{\textbf{BSS}}"] & \\
    \textup{Ext}^{s,f,w}_{\mathcal{E}(n)_2^\vee}(\mathbb{M}_2^\mathbb{C}) \arrow[dr,Rightarrow,"\rho\text{-}\textup{\textbf{BSS}}"'] & & \textup{Ext}^{s,f,w}_{\mathcal{E}(n-1)_2^\vee}(\mathbb{M}_2^\mathbb{R}) \arrow[dl,Rightarrow,"v_n\text{-}\textup{\textbf{BSS}}"] \\
    & \textup{Ext}^{s,f,w}_{\mathcal{E}(n)_2^\vee}(\mathbb{M}_2^\mathbb{R}). 
    \end{tikzcd}
    \label{ss bockstein square C and R}
    \end{equation}
    Regardless of the method, upon computing the $E_2$-page at any prime one easily deduces that the motivic Adams spectral sequence collapses over $\mathbb{C}$ and $\mathbb{R}$ for degree reasons. A similar story holds over finite fields, replacing the $\rho$-\textbf{BSS} with the $\gamma$-\textbf{BSS} in the odd primary case.
    
\end{remark}

\begin{remark}
\label{remark:ExtE(1)fromClassical}
    \rm Our computations show that in the cases where $F=\mathbb{C}$ and $p$ is any prime, $F=\mathbb{R}$ and $p>2$, and $F=\mathbb{F}_q$ with a trivial Bockstein action and $p$ is any prime, the $E_2$-page of the $\textbf{mASS}^F_p(BPGL \langle 1 \rangle)$ can be expressed as
    \[\text{Ext}_{\mathcal{E}(1)^\vee_p}^{s,f,w}(\mathbb{M}^F_p, \mathbb{M}_p^F) \cong \text{Ext}_{E(1)^\vee_p}^{s,f}(\mathbb{F}_p, \mathbb{F}_p) \otimes \mathbb{M}_p^F \cong \mathbb{M}_p^F[v_0, v_1].\]
    To be more direct, this shows that the motivic $E_2$-page is determined by the classical $E_2$-page of the $\textbf{ASS}_p(BP \langle 1 \rangle)$ and the integral motivic homology of a point. A similar observation was made in \cite{finitekqcoop}, where in the case of a trivial Bockstein action, the $E_2$-page of the $\textbf{mASS}^{\mathbb{F}_q}_2(kq)$ was shown to be determined by the classical $E_2$-page of the $\textbf{ASS}_2(bo)$ and the motivic homology of a point.  
    
\end{remark}

\subsection{The homotopy ring of $BPGL \langle 1 \rangle$-cooperations}
Next, we compute the homotopy of $BPGL \langle 1 \rangle \wedge BPGL \langle 1 \rangle.$ We begin by computing the $\Ext$ groups $\text{Ext}^{s,f,w}_{\mathcal{E}(1)^\vee_p}(L_p(k))$.
\begin{lemma}
\label{lem:lightning flash ext 1}
    Let $L_p(k)$ denote the $k^{th}$ lightning flash module. Then there is an isomorphism of $\mathbb{M}_p$-modules and $\mathcal{E}(1)_p^\vee$-comodules:
    \[\textup{Ext}^{s,f,w}_{\mathcal{E}(1)_p^\vee}(L_p(k)) \cong \bigoplus_{i=0}^{k-1}\textup{Ext}^{s,f,w}_{\mathcal{E}(0)_p^\vee}(\mathbb{M}_p)\{x_i\} \oplus \textup{Ext}^{s,f,w}_{\mathcal{E}(1)_p^\vee}(\mathbb{M}_p) \{x_k\},\]
    where $|x_i|=(2i(p-1), 0, i(p-1))$ and where \[v_1x_i = v_0x_{i+1}\]
    for all $0 \leq i \leq k$.
\end{lemma}

\begin{proof}
    For $k=0$ we have that $L_p(0) = \mathbb{M}_p$, and the $\Ext$ group in question was calculated in \cref{C and R BPGL1}, \cref{no bockstein BPGL1 Mass E2}, and \cref{bockstein BPGL1 Mass E2}.

    For $k=1$, we use the short exact sequence of $\mathcal{E}(1)^\vee_p$-comodules from \cref{lightning flash ses}, which takes the form
    \[0 \to \Sigma^{2(p-1), p-1}\mathbb{M}_p \to L_p(1) \to  (\mathcal{E}(1)_p // \mathcal{E}(0)_p)^\vee \to 0.\]
    Applying the functor $\text{Ext}^{s,f,w}_{\mathcal{E}(1)^\vee_p}(-)$ gives the following long exact sequence in $\Ext$:
    \[\cdots \to \text{Ext}^{s,f,w}_{\mathcal{E}(1)_p^\vee}(\Sigma^{2(p-1), p-1}\mathbb{M}_p)  \to{\text{Ext}^{s,f,w}_{\mathcal{E}(1)_p^\vee}(L_p(1))} \to {\text{Ext}^{s,f,w}_{\mathcal{E}(1)_p^\vee}((\mathcal{E}(1)_p//\mathcal{E}(0)_p)^\vee)} \to \cdots	\]
    Observe that $(\mathcal{E}(1)_p//\mathcal{E}(0)_p)^\vee \cong \mathcal{E}(1)^\vee_p\underset{\mathcal{E}(0)^\vee_p}{\Box}\mathbb{M}_p$, and so we can rewrite the right hand term using a change-of-rings isomorphism:
    \[\text{Ext}^{s,f,w}_{\mathcal{E}(1)_p^\vee}(\mathcal{E}(1)^\vee_p\underset{\mathcal{E}(0)^\vee_p}{\Box}\mathbb{M}_p) \cong \text{Ext}^{s,f,w}_{\mathcal{E}(0)_p^\vee}(\mathbb{M}_p).\]
    Our prior calculations of $\text{Ext}^{s,f,w}_{\mathcal{E}(1)_p^\vee}(\mathbb{M}_p)$ and $\text{Ext}^{s,f,w}_{\mathcal{E}(0)^\vee_p}(\mathbb{M}_p)$ show that the connecting homomorphism
    \[\text{Ext}^{s,f,w}_{\mathcal{E}(0)_p^\vee}(\mathbb{M}_p) \xrightarrow{\delta} \text{Ext}^{s+1,f-1,w}_{\mathcal{E}(1)_p^\vee}(\mathbb{M}_p)\]
    must be trivial for degree reasons. This implies that the long exact sequence decomposes into short exact sequences which allow us to construct $\text{Ext}^{s,f,w}_{\mathcal{E}(1)_p^\vee}(L_p(1))$ from the surrounding $\Ext$ groups. In particular, we have an isomorphism
    \[\text{Ext}^{s,f,w}_{\mathcal{E}(1)^\vee_p}(L_p(1)) \cong \text{Ext}^{s,f,w}_{\mathcal{E}(0)_p^\vee}(\mathbb{M}_p)\{x_0\} \oplus \text{Ext}^{s,f,w}_{\mathcal{E}(1)_p^\vee}(\mathbb{M}_p)\{x_{1}\},\]
    where $|x_0| = (0,0,0)$ and $|x_i| = \left(2i(p-1), 0, i(p-1)\right)$
    The extension $v_1x_0 = v_0x_1$ follows from the same methods as in the classical case. 

    Now, suppose the statement is true for all $k \leq n$. By the inductive hypothesis on the structure of $\text{Ext}^{s,f,w}_{\mathcal{E}(1)_p^\vee}(L_p(n-1))$, the long exact sequence in $\Ext$ induced by 
    \[0 \to \Sigma^{2(p-1), p-1}L_p(n-1) \to L_p(n) \to (\mathcal{E}(1)_p//\mathcal{E}(0)_p)^\vee \to 0\]
    has a trivial connecting homomorphism. Using the change-of-rings isomorphism again supplies us with an isomorphism
    \[\text{Ext}^{s,f,w}_{\mathcal{E}(1)_p^\vee}(L_p(n)) \cong \text{Ext}^{s,f,w}_{\mathcal{E}(1)_p^\vee}(\Sigma^{2(p-1), p-1}L_p(n-1)) \oplus \text{Ext}^{s,f,w}_{\mathcal{E}(0)_p^\vee}(\mathbb{M}_p).\]
    Relabel the generators $x_i \in \text{Ext}_{\mathcal{E}(1)^\vee_p}^{s,f,w}(L_p(n-1))$ by $x_{i+1}$, and let $x_0$ denote the generator of $\text{Ext}^{s,f,w}_{\mathcal{E}(0)_p^\vee}(\mathbb{M}_p)$. The hidden extension $v_1x_n = v_0x_{n+1}$ follows from the same methods as in the classical case. By our inductive hypothesis, we have extensions $v_1x_i = v_0x_{i+1}$ for all $0 \leq i \leq n$, concluding the proof.
\end{proof}

For the benefit of the reader, we depict the group $\text{Ext}_{\mathcal{E}(1)^\vee_p}^{s,f,w}(L_p(2))$ for all fields in question in \Cref{lightning3_C}, \Cref{lightning3_R}, \Cref{lightning3_fq_easy}, \Cref{lightning3_f3_hard_p2}, and \Cref{lightning3_fq_hard_p5}.

\begin{figure}[ht]
    \centering
    \includegraphics[width=0.8\linewidth]{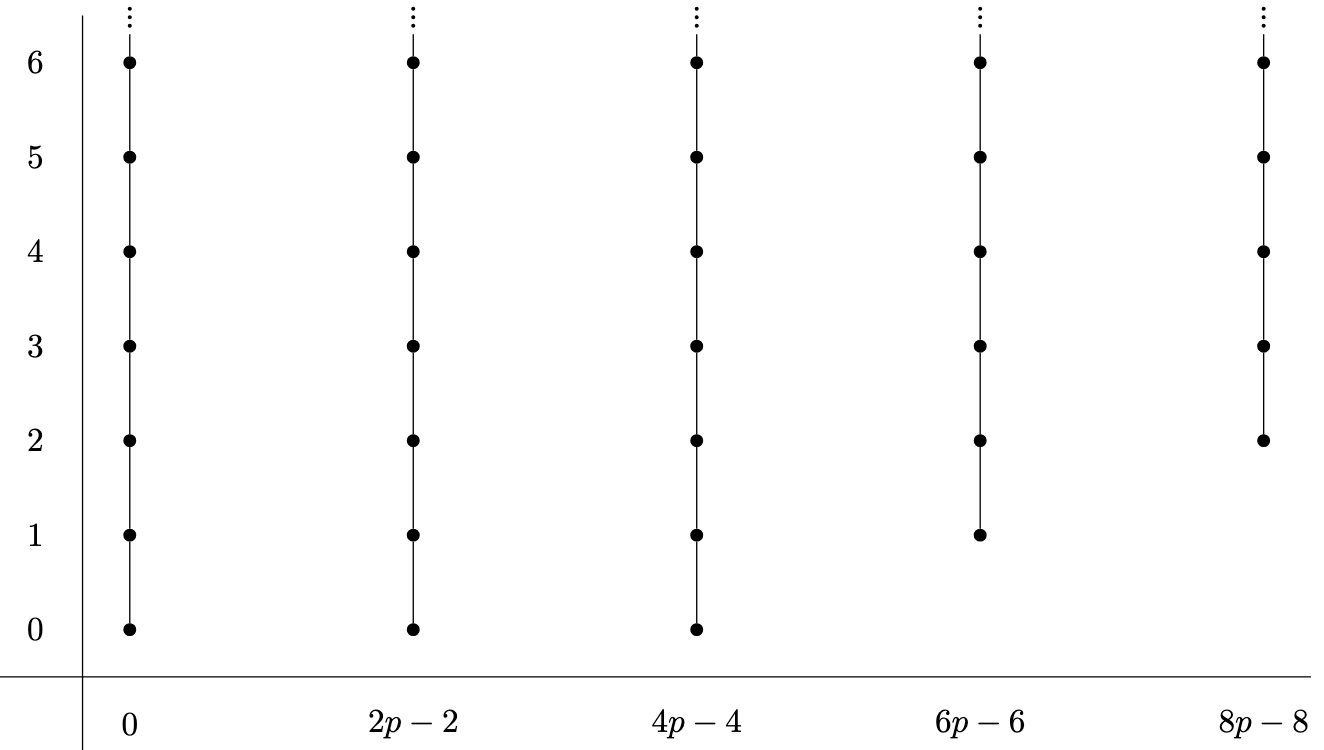}
    \caption{Charts for $\text{Ext}_{\mathcal{E}(1)^\vee_p}^{s,f,w}(\mathbb{M}_p^F,L_p^F(2))$ for $F = \mathbb{C}$ at all primes or $\mathbb{R}$ at odd primes. For $F=\mathbb{C}$, a $\bullet$ denotes $\mathbb{F}_p[\tau]$, and for $F=\mathbb{R}$, a $\bullet$ denotes $\mathbb{F}_p[\theta]$. A vertical line indicates $v_0$-multiplication, while $v_1$-multiplication is suppressed.}
    \label{lightning3_C}
\end{figure}

\begin{figure}[ht]
    \centering
    \includegraphics[width=0.8\linewidth]{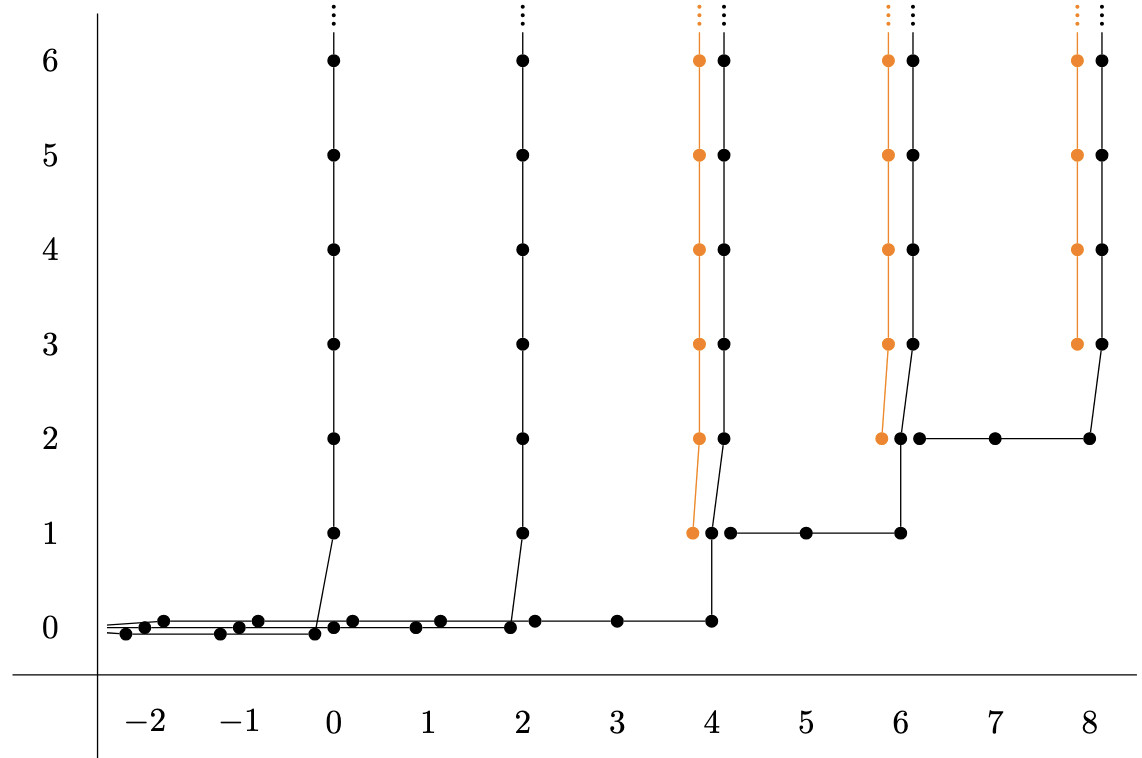}
    \caption{Charts for $\text{Ext}_{\mathcal{E}(1)^\vee_2}^{s,f,w}(\mathbb{M}_2^{\mathbb{R}},L_2^\mathbb{R}(2))$. A $\bullet$ denotes $\mathbb{F}_2[\tau^4]$, and an orange class denotes $\tau^2v_0$-divisibility. A vertical line represents $v_0$-multiplication and a horizontal line represents $\rho$-multiplication, while $v_1$-multiplication is suppressed.}
    \label{lightning3_R}
\end{figure}

\begin{figure}[ht]
    \centering
    \includegraphics[width=0.8\linewidth]{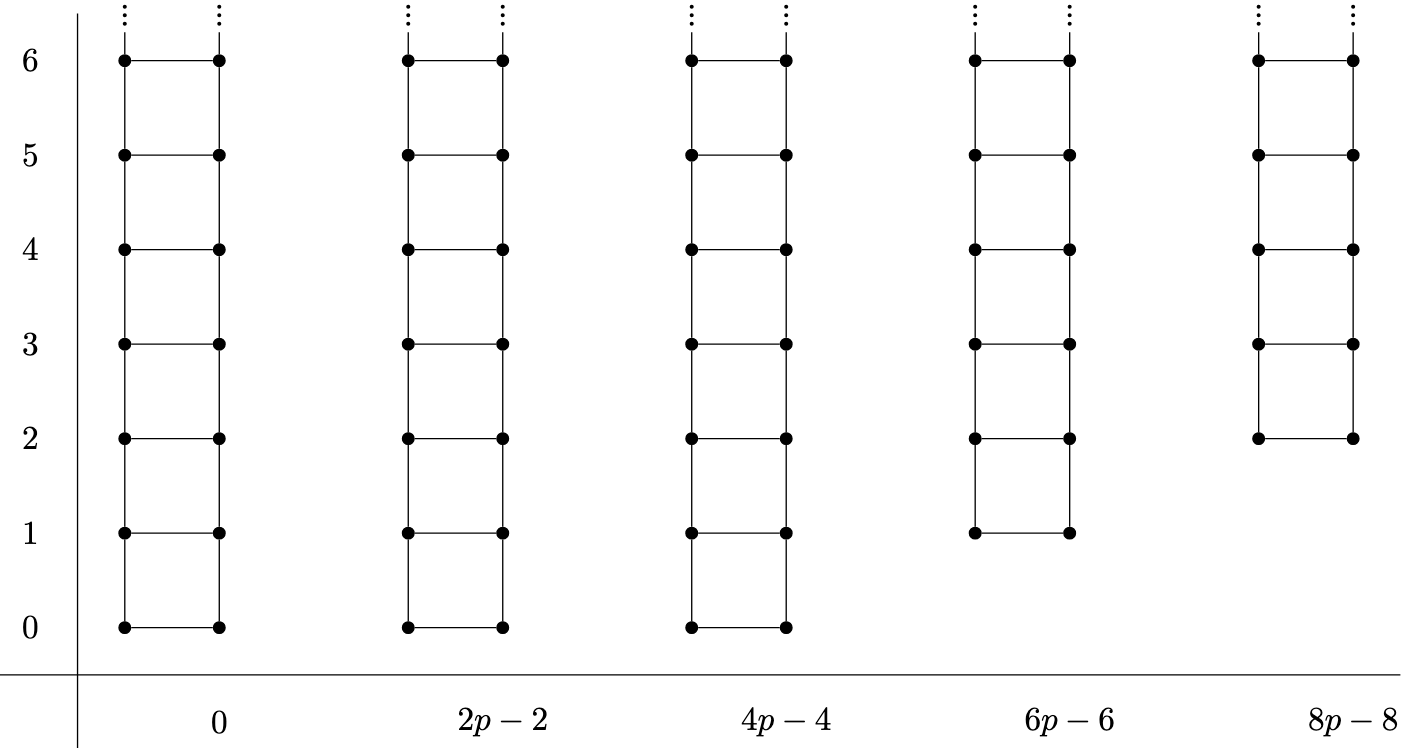}
    \caption{Charts for $\text{Ext}_{\mathcal{E}(1)^\vee_p}^{s,f,w}(\mathbb{M}_p^{\mathbb{F}_q},L^{\mathbb{F}_q}_p(2))$ for $p$ any prime with a trivial Bockstein action on $\mathbb{F}_q$. A $\bullet$ denotes $\mathbb{F}_2[\tau]$ in the case of $p=2$ and $\mathbb{F}_p[\zeta]$ for $p$ odd. A vertical line represents $v_0$-multiplication and a horizontal line represents $u$-multiplication, while $v_1$-multiplication is suppressed.}
    \label{lightning3_fq_easy}
\end{figure}

\begin{figure}[ht]
    \centering
    \includegraphics[width=0.8\linewidth]{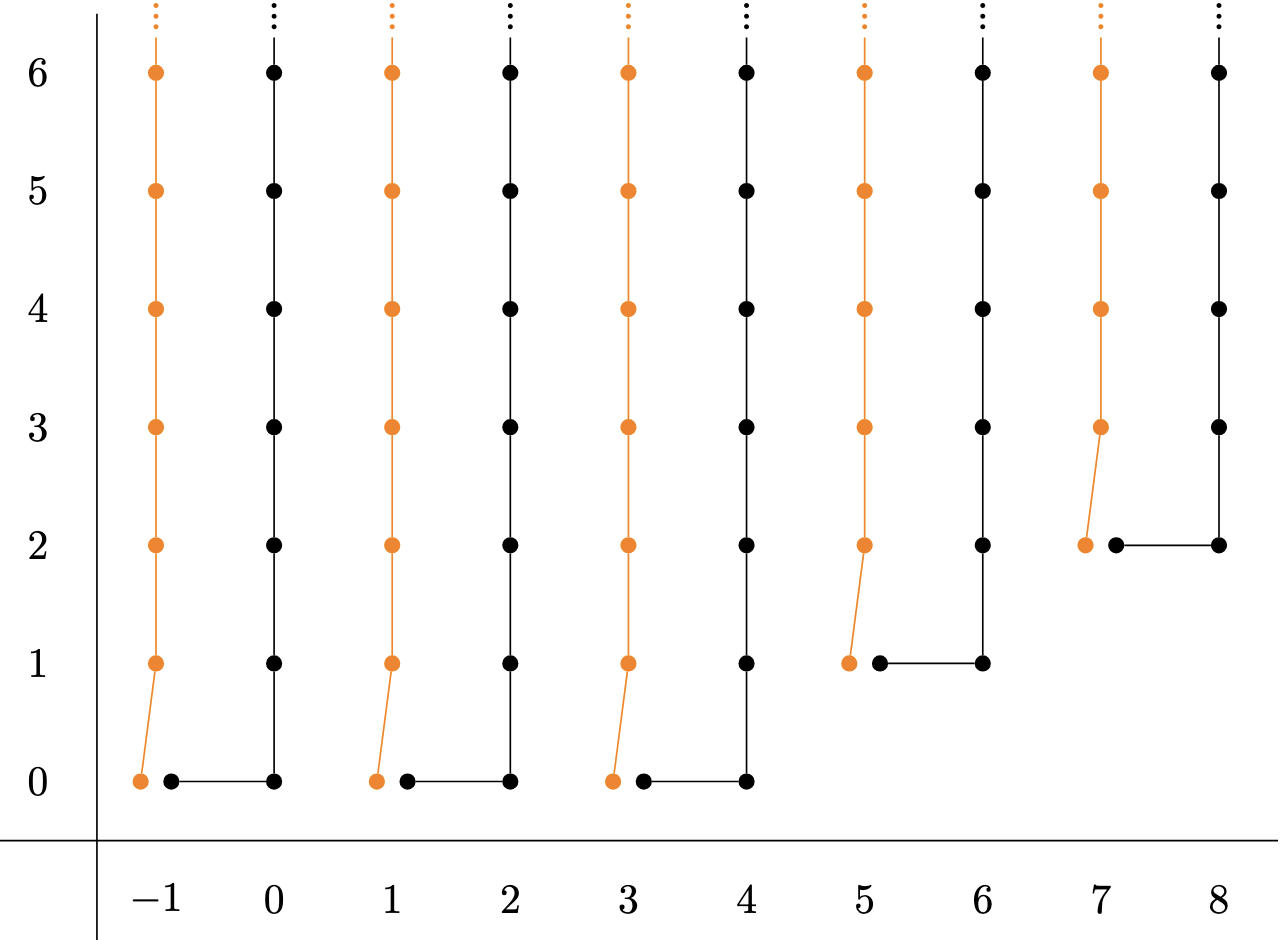}
    \caption{Charts for $\text{Ext}_{\mathcal{E}(1)^\vee_2}^{s,f,w}(\mathbb{M}_2^{\mathbb{F}_q},L^{\mathbb{F}_q}_2(2))$  with a nontrivial Bockstein action on $\mathbb{F}_q$. A $\bullet$ denotes $\mathbb{F}_2[\tau^2]$, and an orange class denotes $\rho\tau$-divisibility. A vertical line represents $v_0$-multiplication and a horizontal line represents $\rho$-multiplication, while $v_1$-multiplication is suppressed.}
    \label{lightning3_f3_hard_p2}
\end{figure}

\begin{figure}[ht]
    \centering
    \includegraphics[width=0.8\linewidth]{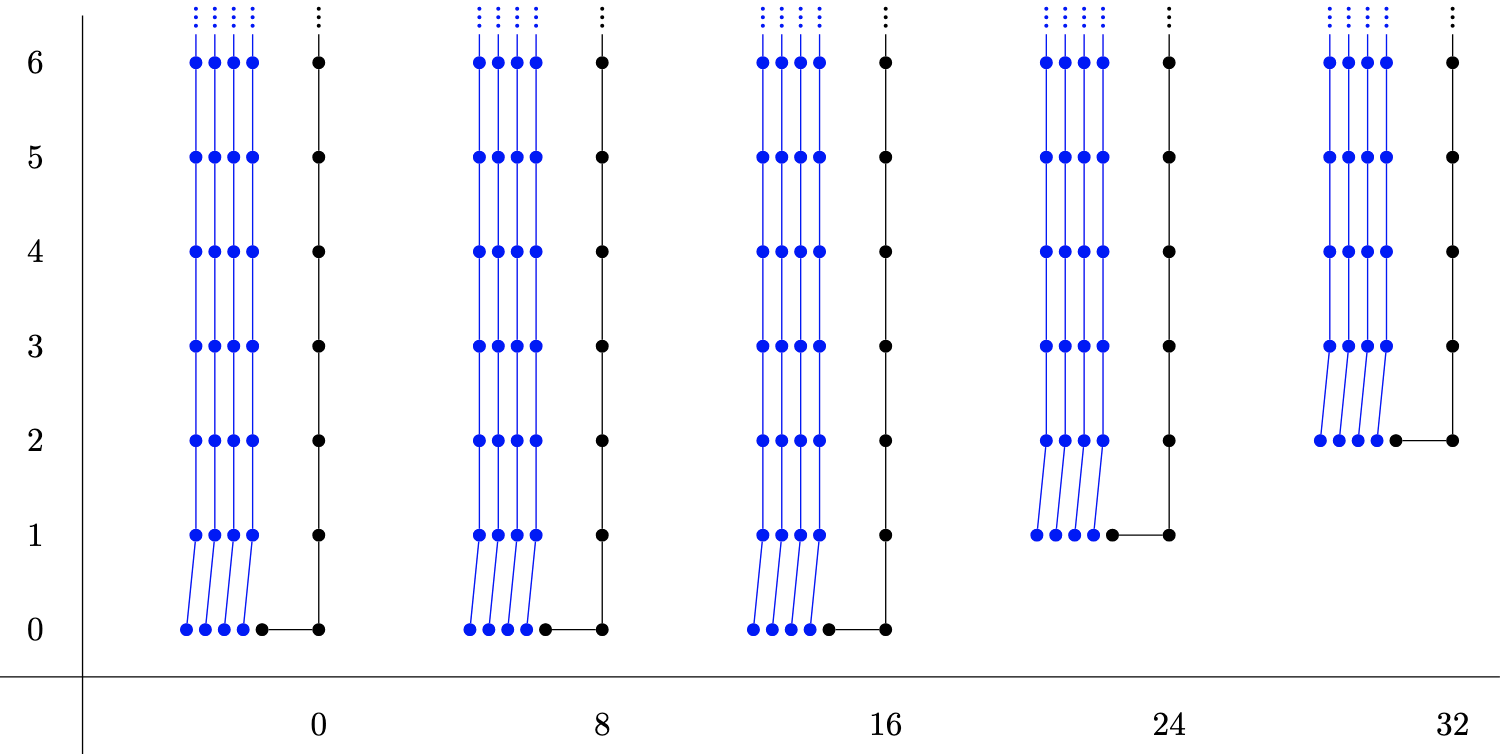}
    \caption{Charts for $\text{Ext}_{\mathcal{E}(1)^\vee_p}^{s,f,w}(\mathbb{M}_p^{\mathbb{F}_q},L^{\mathbb{F}_q}_p(2))$ for $p=5$ with a nontrivial Bockstein action on $\mathbb{F}_q$. A $\bullet$ denotes $\mathbb{F}_5[\zeta^5]$, and a blue class denotes $\mathbb{F}_5\{\gamma\zeta^j\}[\zeta^5]$ for some $1 \leq j \leq 4$. A vertical line represents $v_0$-multiplication, and a horizontal line represents $\gamma$-multiplication, while $v_1$-multiplication is suppressed.}
    \label{lightning3_fq_hard_p5}
\end{figure}

Building off of the observation we made in \Cref{remark:ExtE(1)fromClassical}, we have the following.
\begin{cor}\label{cor: Ext L_0 L_m}
    If $F = \mathbb{C}$ at any prime, $F = \mathbb{R}$ and $p$ is odd, or $F = \mF_q$ and there is a trivial Bockstein action, then 
    \[ \Ext_{\mathcal{E}(1)^\vee_p}^{s,f,w}(\mathbb{M}_p^F,L_p(m)) \cong \Ext_{{E}(1)^\vee_p}^{s,f}(\mathbb{F}_p,L_p^{cl}(m)) \otimes \mM_{p}^{F} ,\]
    where $L_p^{cl}(m)$ denotes the classical lightning flash module.
\end{cor}

\begin{proof}
The descriptions of $\text{Ext}_{\mathcal{E}(1)^\vee_p}^{s,f,w}(\mM_p^F, \mM_p^F)  $ given in the proof of \cref{C and R BPGL1} and \cref{no bockstein BPGL1 Mass E2} tell us that $\text{Ext}_{\mathcal{E}(1)^\vee_p}^{s,f,w}(\mM_p^F, \mM_p^F)  \cong \text{Ext}^{s,f}_{E(1)^\vee_p}(\mF_p, \mF_p) \otimes \mM_p^F$. Likewise, the descriptions of $\text{Ext}_{\mathcal{E}(0)^\vee_p}^{s,f,w}(\mM_p^F, \mM_p^F)  $ given in \cref{mass BPGL0 C and R} and the proofs of \cref{BPGL mass finite field p=2} and \cref{mass BPGL0 finite field p>2} tell us that $\text{Ext}_{\mathcal{E}(0)^\vee_p}^{s,f,w}(\mM_p^F, \mM_p^F)  \cong \text{Ext}^{s,f}_{E(0)^\vee_p}(\mF_p, \mF_p) \otimes \mM_p^F$. The claim follows immediately from combining these observations with \cref{lem:lightning flash ext 1}. 
\end{proof}

The following analogue of \cref{BPGL0 coop mass differentials} will be useful.
\begin{lemma}
\label{BPGL1 coop mass differentials}
    The differentials in the $\textup{\textbf{mASS}}_p(BPGL \langle 1 \rangle \wedge BPGL \langle 1 \rangle)$ are determined by the differentials in the $\textup{\textbf{mASS}}_p(BPGL \langle 1 \rangle).$
\end{lemma}
\begin{proof}
    The map $BPGL \langle 1 \rangle \to BPGL \langle1 \rangle \wedge BPGL \langle 1 \rangle$ induces an inclusion on $E_2$-pages
    \[\text{Ext}^{s,f,w}_{\mathcal{E}(1)_p^\vee}(\mathbb{M}_p) \to \bigoplus_{k \geq 0}\Sigma^{2k(p-1), k(p-1)}\text{Ext}^{s,f,w}_{\mathcal{E}(1)^\vee_p}(L(\nu_p(k!)) \oplus W_k).\]
    This is the inclusion into the $k=0$ summand of the right side, recalling that $L_p(0) \cong \mathbb{M}_p.$ This determines all of the differentials on this summand. By \cref{lem:lightning flash ext 1}, we understand the $\text{Ext}_{\mathcal{E}(1)^\vee_p}^{s,f,w}(\mathbb{M}_p)$-module structure of the remaining summands $\text{Ext}^{s,f,w}_{\mathcal{E}(1)^\vee_p}(L_p(\nu_p(k!)) \oplus W_k)$, which we may rewrite as
    \[ \bigoplus_{i =0}^{\nu_p(k!)-1}(\text{Ext}^{s,f,w}_{\mathcal{E}(0)_p^\vee}(\mathbb{M}_p)\{x_{i}\}) \oplus \text{Ext}^{s,f,w}_{\mathcal{E}(1)_p^\vee}(\mathbb{M}_p)\{x_{\nu_p(k!)}\} \oplus V_k,\]
    where $V_k$ is a sum of suspensions of $\mathbb{M}_p$. The module structure lifts the differentials on $\text{Ext}_{\mathcal{E}(1)^\vee_p}^{s,f,w}(\mathbb{M}_p)$ to these summands in the expected way.
\end{proof}

We can now compute the homotopy ring of cooperations.
\begin{theorem}
\label{thm:BPGL1homotopyCooperations}
    Let $F \in \{\mathbb{C}, \mathbb{R}, \mathbb{F}_q\}$, where $\textup{char}(\mathbb{F}_q) \neq p$. Then there is an isomorphism for all primes $p$:
    \begin{align*}
        \pi_{*,*}(BPGL & \langle 1 \rangle \wedge  BPGL \langle 1 \rangle) \\
        & \cong \underset{k \geq 0}{\bigoplus} \left((\pi_{*,*}BPGL \langle 0 \rangle)\{x_0, \dots, x_{\nu_p(k!)-1}\}
    \oplus \pi_{*,*}(BPGL \langle 1 \rangle) \{x_{\nu_p(k!)}\} \oplus W_k \right),
    \end{align*}
    where $W_k$ is a sum of suspensions of $\mathbb{M}_p$.
\end{theorem}

\begin{proof}
    By \cref{BPGL1 coop mass differentials}, the differentials in the $\textbf{mASS}_p(BPGL \langle 1 \rangle \wedge BPGL \langle 1 \rangle)$ are determined by the differentials in the $\textbf{mASS}_p(BPGL \langle 1 \rangle)$ by using the $\text{Ext}_{\mathcal{E}(1)^\vee_p}^{s,f,w}(\mathbb{M}_p)$-module structure on $\text{Ext}_{\mathcal{E}(1)^\vee_p}^{s,f,w}(L_p(\nu_p(k!)) \oplus W_k)$ obtained from \cref{lem:lightning flash ext 1}. Note that these differentials are in turn determined by the differentials in the $\textbf{mASS}_p(BPGL \langle 0 \rangle)$ by $v_1$-linearity, following the argument given in the proof of \cref{bpgl1 difs finite fields}.
    
    To be precise, let $d_r^{coop}$ denote the differential in the $\textbf{mASS}_p(BPGL \langle 1 \rangle \wedge BPGL \langle 1 \rangle )$, and let $d_r$ denote the differential in the $\textbf{mASS}_p(BPGL \langle 0 \rangle)$. Each summand of the $E_2$-page of the $\textbf{mASS}_p(BPGL \langle 1 \rangle \wedge BPGL \langle 1 \rangle)$ takes the form
    \[ \bigoplus_{i =0}^{\nu_p(k!)-1}(\text{Ext}^{s,f,w}_{\mathcal{E}(0)_p^\vee}(\mathbb{M}_p)\{x_{i}\}) \oplus \text{Ext}^{s,f,w}_{\mathcal{E}(1)_p^\vee}(\mathbb{M}_p)\{x_{\nu_p(k!)}\} \oplus V_k.\]   
    For any class $\alpha \in \text{Ext}_{\mathcal{E}(0)^\vee_p}^{s,f,w}(\mathbb{M}_p)$, we have that
    \[d_r^{coop}(\alpha x_{i}) = d_r(\alpha)x_{i}\]
    for all $0 \leq i \leq \nu_p(k!)-1$. For any class $\beta \in \text{Ext}_{\mathcal{E}(1)^\vee_p}^{s,f,w}(\mathbb{M}_p)$, we can rewrite $\beta$ by factoring out the maximal power of $v_1$ which divides it, so that $\beta = v_1^n\beta'$, where $\beta' \in \text{Ext}^{s,f,w}_{\mathcal{E}(0)^\vee_p}(\mathbb{M}_p)$. This implies that
    \[d_r^{coop}(\beta x_{\nu_p(k!)}) = d_r^{coop}(v_1^n\beta' x_{ \nu_p(k!)}) = v_1^nd_r(\beta')x_{ \nu_p(k!)}.\]
    For the same reason as given in \cref{BPGL0 coop mass differentials}, there can be no differentials involving the $V_k$-term.

    Thus, in each summand of the $\textbf{mASS}_p(BPGL \langle 1 \rangle \wedge BPGL \langle 1 \rangle)$, the differentials are contained within their individual sub-summands. Using our knowledge of the values of these differentials, we have shown that each summand abuts to 
    \[(\pi_{*,*}BPGL \langle 0 \rangle)\{x_1, \dots x_{\nu_p(k!)-1}\} \oplus (\pi_{*,*}BPGL \langle 1 \rangle)\{x_{{\nu_p(k!)}}\} \oplus W_k,\]
    where $W_k$ is a sum of suspensions of $\mathbb{M}_p$. Ranging over $k \geq 0$ proves the claim.
\end{proof}

\begin{remark}
    \rm For $p=2$, this gives a computation of the homotopy ring of cooperations for effective algebraic $K$-theory, recalling that a model for $BPGL\langle 0 \rangle$ is $H\mathbb{Z}$:
    \[\pi^F_{*,*}(kgl \wedge kgl) \cong \bigoplus_{k \geq 0}\left((\pi_{*,*}^FH\mathbb{Z})\{x_0, \dots, x_{\nu_p(k!)-1}\} \oplus (\pi^F_{*,*}kgl)\{x_{\nu_p(k!)}\} \oplus W_k\right).\]
    This extends the computation for $F=\mathbb{R}$ shown in \cite{LiPetTat25}.
\end{remark}

We depict the $E_2$-page of the $\textbf{mASS}^{\mathbb{C}}_2(BPGL \langle 1 \rangle\wedge BPGL \langle 1 \rangle)$ in \Cref{fig:BPGL1CoopmASSE2}. Note that since the spectral sequence collapses over $\mathbb{C}$, one may also interpret this figure as the actual homotopy ring of cooperations $\pi_{*,*}^{\mathbb{C}}(BPGL\langle 1 \rangle \wedge BPGL \langle 1 \rangle)$ as a module over $\pi_{*,*}^{\mathbb{C}}(BPGL \langle 1 \rangle)$. A $\blacksquare$ denotes $\mathbb{F}_2[\tau, v_0]$ and a $\bullet$ denotes $\mathbb{F}_2[\tau]$, with an upwards arrow indicating an infinite tower of $v_0$-multiplication. As usual, we supress $v_1$-multiplication. Different colors correspond to different summands. For example, the color black refers to the $k=0$-summand, which is isomorphic to $\text{Ext}_{\mathcal{E}(1)^\vee_2}^{*,*,*}(\mathbb{M}_2^{\mathbb{C}})$, and the color blue refers to the $k=1$-summand, which is isomorphic to 
\[\Sigma^{2,1}\text{Ext}_{\mathcal{E}(1)_2^\vee}^{*,*,*}(L_2^\mathbb{C}(\nu_2(1!)) \cong \Sigma^{2,1}\text{Ext}_{\mathcal{E}(1)^\vee_2}^{*,*,*}(\mathbb{M}_2^{\mathbb{C}}),\]
and the color red refers to the $k=2$-summand, which is isomorphic to
\[\Sigma^{4,2}\text{Ext}_{\mathcal{E}(1)^\vee_2}^{*,*,*}(L_2^\mathbb{C}(\nu_2(2!))) \cong \Sigma^{4,2}\text{Ext}_{\mathcal{E}(0)^\vee_2}^{*,*,*}(\mathbb{M}_2^{\mathbb{C}}) \oplus \Sigma^{6,3}\text{Ext}_{\mathcal{E}(1)_2^\vee}^{*,*,*}(\mathbb{M}_2^{\mathbb{C}}).\]
Note that one may alter this image to give the $E_2$-page over all base fields $F$ of interest in the following way. First, replace each $\blacksquare$ with a copy of $\text{Ext}_{\mathcal{E}(0)^\vee_2}^{*,*,*}(\mathbb{M}_2^F)$. Then, replace each diagonal $v_1$-tower with a copy of $\text{Ext}_{\mathcal{E}(1)^\vee_2}^{*,*,*}(\mathbb{M}_2^F)$. Similar changes produce charts for the case of an odd prime $p$, where the stem axis is relabeled with $2(p-2)$ instead of 2, and the proliferation of $\blacksquare$ changes depending on the values of $\nu_p(k!)$.

\begin{figure}[ht]
    \centering
    \includegraphics[width=1\linewidth]{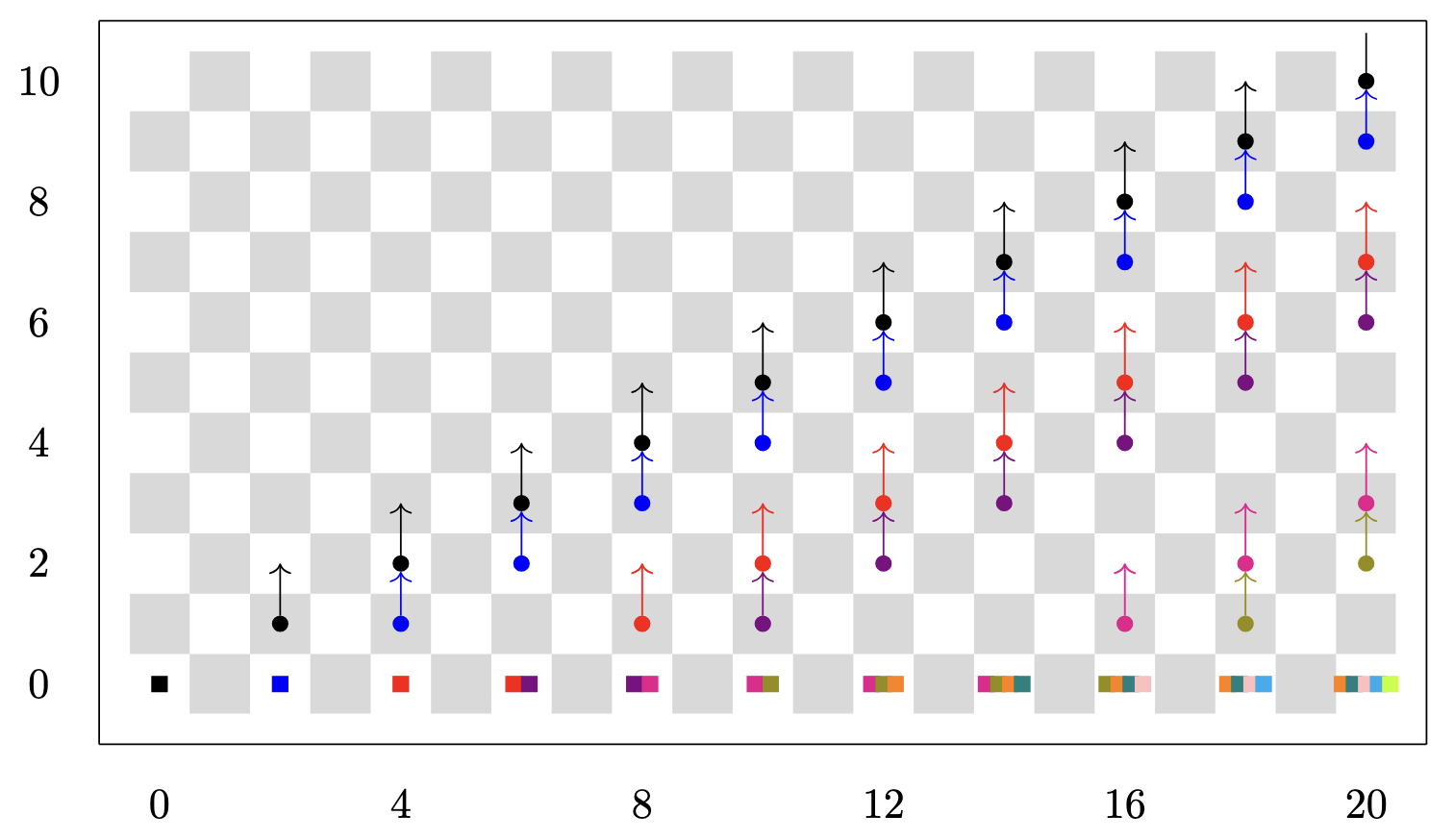}
    \caption{The homotopy ring of cooperations $\pi_{*,*}^{\mathbb{C}}(BPGL \langle 1 \rangle \wedge BPGL \langle 1 \rangle)$ for $p=2$.}
    \label{fig:BPGL1CoopmASSE2}
\end{figure}

\begin{remark}
    \rm While it is useful to be able to analyze the $\textup{\textbf{mASS}}_p(BPGL \langle 1 \rangle \wedge BPGL \langle 1 \rangle)$ in a way that is independent of base field, it also makes the cases of $F=\mathbb{C}$ or $\mathbb{R}$ sound more complex than they actually are. Indeed, although we are able to lift the differentials all the way from the $\textup{\textbf{mASS}}_p(BPGL \langle 0 \rangle)$ to this spectral sequence, there is a bit of a triviality in these cases: there are no differentials to lift, and everything in sight collapses at $E_2$.
\end{remark}

\begin{remark}
    \rm In the cases of $BPGL\langle 0 \rangle$ and $BPGL \langle 1 \rangle$, we see a pattern in the differentials in the motivic Adams spectral sequences for the coefficients and the cooperations: the only classes which could possible support differentials are those which are divisible $\tau$ or $\zeta$. Over $\mathbb{C}$ and $\mathbb{R}$ there are no differentials, with every spectral sequence collapsing at the $E_2$-page. This behavior is a genuine lift of the classical story, as the corresponding spectral sequences for $BP\langle 0 \rangle$ and $BP\langle 1 \rangle$ also collapse at the $E_2$-page. In fact, as the Adams spectral sequences for $BP\langle 2\rangle$ and $BP\langle 2 \rangle \wedge BP \langle 2 \rangle$ also collapse at $E_2$ \cite{culver-bp2coop}, one is inclined to believe that our motivic findings at lower heights also carry over to the $BPGL\langle 2 \rangle$ cooperations algebra.
\end{remark}

\subsection{Constructing the spectrum-level splitting of $BPGL \langle 1 \rangle \wedge BPGL \langle 1 \rangle$} \label{subsec:BPGL1relASS}

Now, we construct splittings at the level of spectra. Note that many of the proofs given here bear similarity to, and are indeed inspired by, those given in \cite[Section 6]{LiPetTat25}. We include our arguments in full for the reader's convenience, starting with an overview of the strategy of the argument.

\subsubsection{Overview}
Similar to our splitting of $BPGL \langle 0 \rangle \wedge BPGL \langle 0 \rangle$, we will produce a splitting of $BPGL \langle 1 \rangle \wedge BPGL \langle 1 \rangle$ by means of a relative Adams spectral sequence. First, we will give a much simpler $BPGL\langle 1 \rangle$-splitting,
\begin{equation}\label{eqn: simple splitting}
    BPGL \langle 1 \rangle \wedge BPGL \langle 1 \rangle \simeq C \vee V,
\end{equation}
    where $V$ is a wedge of suspensions of $H$ and $C$ contains no $H$ summands. Then we will show that there is an isomorphism of $\mathcal{E}(1)^\vee_p$-comodules
\[\bigoplus_{k \geq 0} \theta_k:\Sigma^{2k(p-1), k(p-1)}H_{*,*}^{BPGL\langle 1\rangle}BPGL \langle 1 \rangle^{\langle \nu _p(k!) \rangle}  \xrightarrow{\cong}H_{*,*}^{BPGL\langle 1\rangle}C.\]
We can realize the maps $\theta_k$ on the individual summands as classes in degree $(0,0,0)$ of the following relative Adams spectral sequence.
\begin{align*}
    E_2^{s,f,w} = \text{Ext}^{s,f,w}_{\mathcal{E}(1)^\vee_p}\left(H_{*,*}^{BPGL \langle 1 \rangle}\Sigma^{2k(p-1), k(p-1)}BPGL \langle 1 \rangle^{\langle \nu_p(k!) \rangle}, H_{*,*}C\right) \\
    \implies [\Sigma^{2k(p-1), k(p-1)}BPGL \langle 1 \rangle^{\langle \nu_p(k!)\rangle}, C]^{BPGL \langle 1 \rangle}_{(s,w)}.
\end{align*}  
Showing that each $\theta_k$ survives the spectral sequence produces maps
\[\tilde{\theta}_k:\Sigma^{2k(p-1), k(p-1)}BPGL \langle 1 \rangle^{\nu_p(k!)} \to C \]
which, since they are induced from the prescribed homology isomorphisms, must assemble to give an equivalence of spectra (up to $p$-completion). Combined with \cref{eqn: simple splitting}, this indeed yields a $BPGL\langle 1 \rangle$-module equivalence
\[BPGL\langle 1 \rangle \wedge BPGL\langle 1 \rangle \simeq \bigvee\limits_{k=0}^{\infty}\Sigma^{2k(p-1), k(p-1)}BPGL \langle 1 \rangle^{\nu_p(k!)} \vee V  \]

\subsubsection{Adams covers and relative homology}
We begin by establishing a simpler topological splitting.
\begin{proposition}
\label{prop:BPGL1SplittingBeta}
    There is a $BPGL\langle 1 \rangle$-module splitting
    \[BPGL \langle 1 \rangle \wedge BPGL \langle 1 \rangle \simeq C \vee V,\]
    where $V$ is a wedge of suspensions of $H$ and $C$ contains no $H$ summands.
\end{proposition}

\begin{proof}
    Recall from \Cref{prop:BPGLnHomology} that there is an isomorphism of $\mathcal{E}(1)_p^\vee$-comodules:
    \[H_{*,*}BPGL \langle 1 \rangle \cong \bigoplus_{k \geq 0}\Sigma^{2k(p-1), k(p-1)}L_p(\nu_p(k!)) \oplus W_k,\]
    where each $W_k$ is a finite sum of suspensions of $\mathcal{E}(1)^\vee_p$. In \Cref{prop:relHomology}, we showed that $H_{*,*}^{BPGL \langle 1 \rangle}H \cong \mathcal{E}(1)^\vee_p$. Now, take $V_k$ to be a wedge of suspensions of $H$ such that $H^{BPGL \langle 1 \rangle}_{*,*}V_k \cong W_k$. Letting $V = \bigvee_{k \geq 0}V_k$, we have that 
    \[H^{BPGL \langle 1 \rangle}_{*,*}V \cong \bigoplus_{k \geq 0}W_k.\]
    Consider the $BPGL \langle 1 \rangle$-relative Adams spectral sequences
    \[E_2^{s,f,w} = \text{Ext}^{s,f,w}_{\mathcal{E}(1)^\vee_p}(H_{*,*}^{BPGL \langle 1 \rangle}V, H_{*,*}BPGL \langle 1 \rangle) \implies [V, BPGL \langle1 \rangle \wedge BPGL \langle 1 \rangle]^{BPGL \langle 1 \rangle}_{(s,w)}\]
    \[E_2^{s,f,w} = \text{Ext}^{s,f,w}_{\mathcal{E}(1)^\vee_p}(H_{*,*}BPGL \langle 1 \rangle, H^{BPGL \langle 1 \rangle}_{*,*}V) \implies [BPGL \langle 1 \rangle \wedge BPGL \langle 1 \rangle, V]^{BPGL \langle 1 \rangle}_{(s,w)}\]
    We can view the inclusion
    \[i:H_{*,*}^{BPGL \langle 1 \rangle}V \cong W \hookrightarrow H_{*,*}BPGL \langle 1 \rangle\]
    as a class in filtration 0 of the first spectral sequence, and we may view the projection
    \[j:H_{*,*}BPGL \langle 1 \rangle \to W \cong H^{BPGL \langle 1 \rangle}_{*,*}V\]
    as a class in filtration 0 of the second spectral sequence. Since  $H_{*,*}^{BPGL \langle 1 \rangle}V$ is free over $\mathcal{E}(1)^\vee_p$, both spectral sequences are entirely contained in filtration $f=0$. Adams differentials increase filtration, hence both spectral sequences collapse. This implies that the inclusion and projection maps lift to give maps of $BPGL \langle 1 \rangle$-modules:
    \[V \to BPGL \langle 1 \rangle \wedge BPGL \langle 1 \rangle \to V.\]
    This gives the desired splitting $BPGL \langle 1 \rangle \wedge BPGL \langle 1 \rangle \simeq C \vee V$, where $C$ contains no $H$-summands.
\end{proof}

\begin{remark}
    \rm Classically, there is a simplified version of this proof. It is a theorem of Margolis \cite{Margolis74} that \textit{any} bounded below, locally finite spectrum $X$ admits a decomposition into $C \vee V$, where $C$ is a wedge of suspensions of $H$ and $C$ contains no $H$-summands. Margolis's proof is exactly the same as the one given here, except that he uses the classical Adams spectral sequence, rather than a $BP\langle 1 \rangle$-relative Adams spectral sequence. If one were to know that the motivic Steenrod algebra was injective as a module over itself, then one could give such a splitting for any bounded-below, locally finite motivic spectrum. Such a splitting would be useful in studying the cooperations algebra for Hermitian $K$-theory $kq \wedge kq$. For that case, $\cA(1)$ plays the role that $\cE(n)$ plays for $BPGL\langle n \rangle$, so knowing that $\cA(1)$ is self-injective would suffice. 
\end{remark}

\begin{proposition}
\label{prop:RelativeHomologyBPGL1AdamsCoverLightning}
Let $BPGL\langle 1 \rangle ^{\langle n \rangle}$ denote the $n^{th}$ Adams cover of $BPGL \langle 1 \rangle.$ Then there is an isomorphism
\[H^{BPGL \langle 1 \rangle}_{*,*}(BPGL\langle 1 \rangle^{\langle k\rangle}) \cong L_p(k).\]
\end{proposition}

\begin{proof}
    Recall that the Adams covers $BPGL \langle 1 \rangle^{\langle n \rangle}$ are defined by a minimal Adams resolution of $BPGL \langle 1 \rangle$:
    \[\begin{tikzcd}
	{BPGL \langle 1 \rangle^{\langle 0 \rangle}} & {BPGL \langle 1 \rangle^{\langle 1 \rangle}} & {BPGL \langle 1 \rangle^{\langle 2 \rangle}} & \cdots \\
	{K_0} & {K_1} & {K_2}
	\arrow["{i_0}", from=1-1, to=2-1]
	\arrow[from=1-2, to=1-1]
	\arrow["{i_1}", from=1-2, to=2-2]
	\arrow[from=1-3, to=1-2]
	\arrow["{i_2}", from=1-3, to=2-3]
	\arrow[from=1-4, to=1-3]
	\arrow[dashed, from=2-1, to=1-2]
	\arrow[dashed, from=2-2, to=1-3]
	\arrow[dashed, from=2-3, to=1-4]
    \end{tikzcd}\]   
    We let $BPGL \langle 1 \rangle ^{\langle 0 \rangle} := BPGL \langle 1 \rangle,$ and we inductively define the Adams covers so that $BPGL \langle 1 \rangle^{\langle n \rangle}$ is the fiber of $BPGL \langle 1 \rangle^{\langle n-1 \rangle} \xrightarrow{i_{n-1}}K_{n-1}$ for all $n \geq 1.$ Now, since $BPGL \langle 1 \rangle$ and $H$ are both $BPGL \langle 1 \rangle$-modules, we can apply relative homology $H_{*,*}^{BPGL \langle 1 \rangle}(-)$ to get a minimal resolution of $\mathbb{M}_p^F$ by $\mathcal{E}(1)^\vee_p$-comodules.

    We proceed by induction. Choose $K_0=H$ and the map $i_0:BPGL \langle 1 \rangle \to H$ to be the unit map in $H^{*,*}BPGL \langle 1 \rangle = [BPGL\langle 1\rangle, H]_{(*,*)}.$ Applying relative homology to the cofiber sequence
    \[BPGL \langle 1 \rangle^{\langle 1 \rangle} \to BPGL \langle 1 \rangle \xrightarrow{i_0}H\]
    gives a long exact sequence
    \[\cdots \to H_{*,*}^{BPGL \langle 1 \rangle}BPGL \langle 1 \rangle^{\langle 1 \rangle} \to \mathbb{M}_p^F \to \mathcal{E}(1)^\vee_p \to \cdots.\]
    By inspection, we have that $H_{*,*}^{BPGL \langle 1 \rangle}BPGL \langle 1 \rangle^{\langle 1 \rangle} \cong L_p(1).$

    Now, suppose that $H_{*,*}^{BPGL \langle 1 \rangle}BPGL \langle 1\rangle^{\langle k \rangle} \cong L_p(k)$ for all $k \leq n$. Define    
    \[
    K_n := \bigoplus_{i=0}^n\Sigma^{2i(p - 1)+1,i(p-1)}H.
    \]
     Recall that
    \[L_p(n) = \mathcal{E}_p(1) \{x_1, x_2, \cdots, x_n \, |\,  x_{i + 1} Q_1 =  x_i Q_0,\, 1 \leq i \leq n-1 \}.\]
    Define a map
    \[{i_{n}}:L_p(n) \to H_{*,*}^{BPGL \langle 1 \rangle}K_n\]
    by sending $x_i \mapsto \Sigma^{2i(p - 1)+1,i(p-1)}1 \in \Sigma^{2i(p - 1)+1,i(p-1)}\mathcal{E}(1)^\vee_p$. This map can be realized by a map of spectra $i_n:BPGL \langle 1 \rangle^{\langle n \rangle} \to K_n$. To see this, consider the relative Adams spectral sequence
    \[ \Ext^{s,f,w}_{\cE(1)_p^\vee}\left(L_p(n), H_{*,*}^{BPGL\langle 1 \rangle}K_n \right)\Longrightarrow [BPGL\langle 1 \rangle^{\langle n \rangle}, K_n]_{(s,w)}^{BPGL\langle 1\rangle}.\]
    The homological map $i_n:L_p(n) \to H_{*,*}^{BPGL \langle 1 \rangle}K_n$ is represented by a a class in filtration $f=0$. Since $H_{*,*}^{BPGL\langle 1 \rangle}K$ is a finite type sum of suspensions of $\cE(1)_p^\vee$ and we showed in \cref{prop:freeImpliesInjective} that $\cE(1)_p^\vee$ is self-injective, the spectral sequence collapses at the $E_2$-page, and so we get the indicated map of spectra. By induction and the long exact sequence in relative homology associated to
    \[BPGL \langle 1 \rangle ^{ \langle n+1 \rangle} \to BPGL \langle 1 \rangle ^{\langle n\rangle} \xrightarrow{i_{n}} K_n,\]
    we obtain the desired isomorphism.
\end{proof}

Combining \cref{prop:LightningFlashBGiso} and \cref{prop:BGdecomposition} with \cref{prop:RelativeHomologyBPGL1AdamsCoverLightning} tells us that we have $\cE(1)_p^\vee$-comodule inclusions
\[\theta_k: \Sigma^{2(p-1)k, (p-1)k}H_{*,*}^{BPGL\langle 1\rangle} BPGL\langle 1\rangle^{\langle \nu_p(k!) \rangle} \hookrightarrow H_{*,*}^{BPGL\langle 1 \rangle}C , \]
such that $\bigoplus\limits_{k=0}^{\infty}\theta_k$ is an isomorphism. We can view each $\theta_k$ as a class in degree $(0,0,0)$ in the relative Adams spectral sequence 
\begin{align}
\label{ss:RelativeAssBPGL1Splittings}
    E_2^{s,f,w} = \text{Ext}^{s,f,w}_{\mathcal{E}(1)^\vee_p}\left(H_{*,*}^{BPGL \langle 1 \rangle}\left(\Sigma^{2k(p-1), k(p-1)}BPGL \langle 1 \rangle^{\langle \nu_p(k!) \rangle}\right), H_{*,*}^{BPGL \langle 1 \rangle}C\right) \\
    \implies [\Sigma^{2k(p-1), k(p-1)}BPGL \langle 1 \rangle^{\langle \nu_p(k!) \rangle}, C]^{BPGL \langle 1 \rangle}_{(s,w)}.
\end{align}
Combining \cref{prop:LightningFlashBGiso} and \cref{prop:BGdecomposition} allows us to rewrite the $E_2$-page as  \[\bigoplus_{m \geq 0}\text{Ext}^{s,f,w}_{\mathcal{E}(1)^\vee_p}\left(\Sigma^{2k(p-1), k(p-1)}L_p(\nu_p(k!)), L_p(\nu_p(m!))\right),\]
and \[\theta_k:\Sigma^{2k(p-1), k(p-1)}L_p(\nu_p(k!)) \hookrightarrow \bigoplus_{m \geq 0}\Sigma^{2m(p-1), m(p-1)}L_p(\nu_p(m!)) .\]
In particular, note that $\theta_k$ is the identity map into the summand $m=k$, and maps trivially into all other summands. 

If we can show that each map $\theta_k$ survives the spectral sequence, then we will have maps \[\widetilde{\theta_k}:\Sigma^{2(p-1)k, (p-1)k}BPGL\langle 1 \rangle^{\langle \nu_p(k!) \rangle} \to C\] such that their sum is a $BPGL\langle 1 \rangle$-module equivalence. First, we must compute the groups $\text{Ext}^{s,f,w}_{\mathcal{E}(1)^\vee_p}\left(L_p(k), L_p(m)\right)$ so that we can analyze the $E_2$-page of the spectral sequence in \Cref{ss:RelativeAssBPGL1Splittings}.

\subsubsection{Computations with lightning flash modules} Our final argument relies on the computation of the groups $\text{Ext}_{\mathcal{E}(1)_p^\vee}^{s,f,w}(L_p(k), L_p(m))$ for all values of $k, m \geq 0$. We begin with a useful lemma.

\begin{lemma}
\label{lemma:WrongSide}
    There is a `wrong-side' change-of-rings isomorphism
    \[\textup{Ext}_{\mathcal{E}(1)_p^\vee}^{s,f,w}((\mathcal{E}(1)_p//\mathcal{E}(0)_p)^\vee, -) \cong \Sigma^{-(2p-2)-1, -(p-1)}\textup{Ext}^{s,f,w}_{\mathcal{E}(0)_p^\vee}(\mathbb{M}_p, -).\]
\end{lemma}

\begin{proof}
    By the equivalence of categories between left $\mathcal{E}(n)^\vee_p$-comodules and right $\mathcal{E}(n)_p$-modules of \Cref{prop:EquivofCats}, we have an isomorphism
    \[\text{Ext}^{s,f,w}_{\mathcal{E}(1)^\vee_p}((\mathcal{E}(1)_p//\mathcal{E}(0)_p)^\vee, -) \cong \text{Ext}^{s,f,w}_{\mathcal{E}(1)_p}((\mathcal{E}(1)_p//\mathcal{E}(0)_p)^\vee, -).\]
    Note that as an $\mathcal{E}(1)_p$-module, we have that
    \[(\mathcal{E}(1)_p//\mathcal{E}(0)_p)^\vee \cong \Sigma^{(2p-2)+1, p-1}\mathcal{E}(1)_p//\mathcal{E}(0)_p.\]
    The ordinary change-of-rings isomorphism yields
    \[\text{Ext}^{s,f,w}_{\mathcal{E}(1)_p}(\Sigma^{(2p-2)+1, p-1}\mathcal{E}(1)_p//\mathcal{E}(0)_p, -) \cong \Sigma^{-(2p-2)-1, -(p-1)}\text{Ext}^{s,f,w}_{\mathcal{E}(0)_p}(\mathbb{M}_p, -).\]
    Passing back through the equivalence of categories of \Cref{prop:EquivofCats} gives the result.
\end{proof}

The following is straightforward and is easily checked by drawing the relevant lightning flash modules. 
\begin{lemma}
\label{lemma:HomLightning}
    Suppose $m > k$. Then when $s<0$, we have 
    \[\textup{Hom}^{(2p-2)s, w}_{\mathcal{E}(1)^\vee_p}(L_p(k), L_p(m))=0\]
    for all $w$.
\end{lemma}

We are now ready to compute $\text{Ext}_{\mathcal{E}(1)^\vee_p}^{s,f,w}(L_p(k), L_p(m))$ for $k \leq m$. In the following proof, we say that an element $x$ generates a ``$v_0$-tower'' to mean that $x$ generates $\text{Ext}_{\mathcal{E}(0)^\vee_p}^{s,f,w}(\mathbb{M}_p, \mathbb{M}_p)$. We depict the spectral sequence described in the proof below in the case of $F=\mathbb{C}$ and $p=2$ in \Cref{fig:Ext_L1_L2_SS_C_p2} and \Cref{fig:Ext_L1_L2_C_p2}. Similar charts for $F=\mathbb{R}$ and $p=2$ are depicted in \cite[Section 6]{LiPetTat25}. We encourage the reader to refer to these charts while reading the proof.
\begin{proposition}
\label{prop:ExtLightningflashKleqM}
    Suppose that $k \leq m$. Then 
    \[\textup{Ext}^{s,f,w}_{\mathcal{E}(1)^\vee_p}(L_p(k), L_p(m)) \cong \textup{Ext}_{\mathcal{E}(1)^\vee_p}^{s,f,w}(\mathbb{M}_p, L_p(m-k)) \oplus B,\]
    where $B$ consists of $v_0$ and $v_1$-torsion concentrated in filtration $f=0$ and negative stems.
    To be precise, $B$ consists of:
    \begin{enumerate}
        \item A direct sum of $\mathbb{M}_p^{\mathbb{C}}$ in filtration 0 and negative stems congruent to 1 modulo $2(p-1)$, when $F=\mathbb{C}$ and $p$ is any prime;
        \item A direct sum of $\mathbb{M}_p^\mathbb{R}$ in filtration 0 and negative stems congruent to 1 modulo $2(p-1)$, whenever $F=\mathbb{R}$ and $p>2$;
        \item A direct sum of $\mathbb{F}_2[\rho, \tau^4]$ in filtration 0 with generator in negative stems congruent to 1 modulo $2$, whenever $F=\mathbb{R}$ and $p=2$;
        \item A direct sum of $\mathbb{M}_p^{\mathbb{F}_q}[u]/(u^2)$ in filtration 0 with generator in negative stems congruent to 1 modulo $2(p-1)$, whenever $F=\mathbb{F}_q$ has a trivial Bockstein action and $p$ is any prime;
        \item A direct sum of $\mathbb{F}_2[\tau^2]$,  $\mathbb{F}_2[\tau^2, \rho]/(\rho^2)$, and $\mathbb{F}_2[\tau^2, \rho\tau]/((\rho\tau)^2)$ in filtration 0 with generator in negative stems congruent to 1 modulo 2, whenever $F=\mathbb{F}_q$ has a nontrivial Bockstein action and $p=2$;
        \item A direct sum of $\mathbb{F}_p[\zeta^p]$, $\mathbb{F}_p[\zeta^p, \gamma]/(\gamma^2)$, and $\mathbb{F}_p[\zeta^p, \gamma\zeta^j]/((\gamma\zeta^j)^2)$ for all $1 \leq j \leq p-1$ in filtration 0 with generator in negative stems congruent to 1 modulo $2(p-1)$, whenever $F=\mathbb{F}_q$ has a nontrivial Bockstein action and $p>2$.
    \end{enumerate}
\end{proposition}

\begin{proof}
    We proceed by induction on $k$. The case of $k=0$ was computed in \Cref{lem:lightning flash ext 1}. Suppose the result holds for all $n <k$. The short exact sequence of lightning flash modules from \Cref{lightning flash ses}:
    \[0 \to \Sigma^{2(p-1), p-1}L_p(k-1) \to L_p(k) \to (\mathcal{E}(1)_p // \mathcal{E}(0)_p)^\vee \to 0.\]    
    induces a long exact sequence of the form
    \begin{align}
    \label{eq:LESLightningFlashExtKleqM}
    \begin{split}
        \cdots \to \text{Ext}^{s,f,w}_{\mathcal{E}(1)^\vee_p}(L_p(k), L_p(m)) \to \text{Ext}^{s,f,w}_{\mathcal{E}(1)^\vee_p}(\Sigma^{2(p-1), p-1}L_p(k-1), L_p(m)) \\
        \xrightarrow{d}\text{Ext}^{s-1,f+1,w}_{\mathcal{E}(1)^\vee_p}((\mathcal{E}(1)_p//\mathcal{E}(0)_p)^\vee, L_p(m)) \to \cdots
    \end{split}
    \end{align}
    We will show our result by studying the spectral sequence associated to this long exact sequence. By the inductive hypothesis, we have that:
    \[\text{Ext}_{\mathcal{E}(1)^\vee_p}^{s,f,w}(L_p(k-1), L_p(m)) \cong \text{Ext}_{\mathcal{E}(1)^\vee_p}^{s,f,w}(\mathbb{M}_p, L_p(m-k+1)) \oplus B.\]
    By the change-of-rings isomorphism of \Cref{lemma:WrongSide}, we also have that:
    \[\text{Ext}_{\mathcal{E}(1)^\vee_p}^{s,f,w}(\mathcal{E}(1)//\mathcal{E}(0))^\vee_p, L_p(m)) \cong \Sigma^{-2(p-1)-1, -(p-1)}\text{Ext}_{\mathcal{E}(0)^\vee_p}^{s,f,w}(\mathbb{M}_p, L_p(m)).\]
    Thus, we may write the spectral sequence associated to \Cref{eq:LESLightningFlashExtKleqM} as
    \begin{align}
    \label{eq:SSLightningFlashExtKleqM}
    \begin{split}
    \Sigma^{-2(p-1)-1, -(p-1)}\text{Ext}_{\mathcal{E}(0)^\vee_p}^{s,f,w}(\mathbb{M}_p, L_p(m)) \oplus \Sigma^{-2(p-1), -(p-1)}\text{Ext}^{s,f,w}_{\mathcal{E}(1)^\vee_p}(\mathbb{M}_p, L_p(m-k+1)) \oplus B \\
    \implies \text{Ext}_{\mathcal{E}(1)^\vee_p}^{s,f,w}(L_p(k), L_p(m)).
    \end{split}
    \end{align}
    Notice that the $B$ summand cannot support any differentials as it is concentrated in Adams filtration $f=0$.
    
    As an $\mathcal{E}(0)^\vee_p$-comodule, we have that
    \[L_p(m) \cong  \mathbb{M}_p \oplus \Sigma^{2(p-1)+1, (p-1)}\mathcal{E}(1)_p^\vee \oplus \dots \oplus \Sigma^{2m(p-1)+1, m(p-1)}\mathcal{E}(1)^\vee_p.\]
    This implies that the left-hand Ext group on the $E_1$-page can be expressed as
    \[E_1^{s,f,w} \cong\text{Ext}_{\mathcal{E}(0)^\vee_p}^{s,f,w}(\mathbb{M}_p, L_p(m)) \cong \text{Ext}_{\mathcal{E}(0)^\vee_p}^{s,f,w}(\mathbb{M}_p, \mathbb{M}_p) \oplus W,\]
    where $W$ is a sum of suspensions of $\mathbb{M}_p$ in stems congruent to 1 modulo $2(p-1)$. 
    
    Let the generator for the $v_0$-tower in $\Sigma^{-2(p-1), -(p-1)}\text{Ext}_{\mathcal{E}(1)^\vee_p}^{s,f,w}(\mathbb{M}_p, L_p(m-k+1))$ of lowest stem degree be denoted $x$,
    and let the generator for the \textit{only} $v_0$-tower in 
    
    $\Sigma^{-2(p-1)-1, -(p-1)}\text{Ext}_{\mathcal{E}(0)^\vee_p}^{s,f,w}(\mathbb{M}_p, L_p(m))$ be denoted $y$. Note that \
    \[|x| = (-2(p-1), 0, -(p-1)) \quad\text{ and } \quad|y| = (-2(p-1)-1, 0, -(p-1)).\]
    By inspection, we see that the only possible differential is from the $v_0$-tower on $x$ to the $v_0$-tower on $y$. By \Cref{lemma:HomLightning}, we know that $\text{Ext}^{2(p-1)s, 0, *}_{\mathcal{E}(1)^\vee_p}(L_p(k), L_p(m))=0$. Thus $x$ must support a differential, and for degree reasons it is of the form
    \[d(x) = v_0y.\]
    By the same argument, the expression $\tau^nx$, or $\theta^nx$, or $\zeta^n x$, depending on the base field and prime must support a differential for every $\tau^n$, $\theta^n$, or $\zeta^n$ appearing in $\text{Ext}_{\mathcal{E}(0)^\vee_p}^{s,f,w}(\mathbb{M}_p, \mathbb{M}_p)$. Thus, there are differentials
    \[d(\tau^nx) = \tau^nv_0y, \quad d(\theta^nx) = \theta^nv_0y, \quad d(\zeta^nx) = \zeta^nv_0y.\]
    Since the spectral sequence is linear over $\text{Ext}_{\mathcal{E}(1)^\vee_p}^{s,f,w}(\mathbb{M}_p, \mathbb{M}_p)$, we also have differentials
    \[d(v_0^ix) = v_0^{i+1}y,\]
    and, for base fields $F$ where the expressions make sense, we have differentials
    \[d(uv_0^ix)=uv_0^{i+1}y, \quad d(\rho\tau\cdot v_0^ix) = \rho\tau \cdot v_0^{i+1}y, \quad d(\gamma\zeta^j \cdot v_0^i x) = \gamma \zeta^j \cdot v_0^{i+1}y.\]
    \begin{figure}
    \centering
    \includegraphics[width=0.8\linewidth]{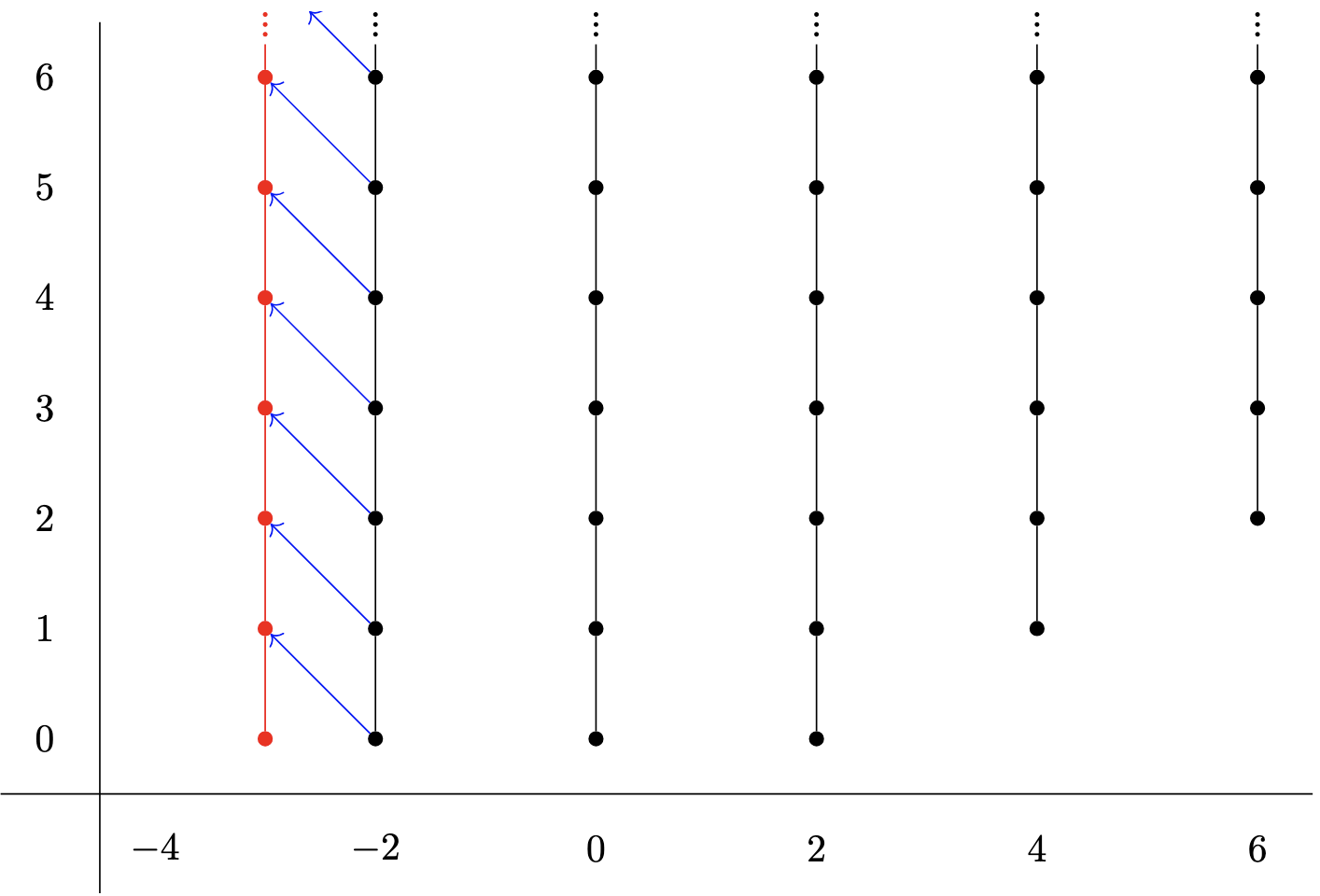}
    \caption{The $E_1$-page of the spectral sequence computing $\text{Ext}_{\mathcal{E}(1)^\vee_2}^{s,f,w}(L_2(1), L_2(2))$ in the case of $F=\mathbb{C}$. The black portion denotes $\text{Ext}_{\mathcal{E}(1)^\vee_2}^{s,f,w}(\mathbb{M}_2^{\mathbb{C}}, L_2(2))$ and the red portion denotes $\text{Ext}_{\mathcal{E}(0)_2^\vee}^{s,f,w}(\mathbb{M}_2^{\mathbb{C}}, \mathbb{M}_2^{\mathbb{C}})$, and differentials are blue}
    \label{fig:Ext_L1_L2_SS_C_p2}
    \end{figure}
    This determines all the differentials in the spectral sequence, and there are no extensions for degree reasons. The particular values of $B$ are determined by any class in $\text{Ext}_{\mathcal{E}(1)^\vee_p}^{s,f,w}(\mathbb{M}_p. \mathbb{M}_p)$ in negative stem which does not support $v_0$-multiplication, any class in $\text{Ext}_{\mathcal{E}(0)^\vee_p}^{s,f,w}(\mathbb{M}_p, \mathbb{M}_p)$ in filtration 0, as it cannot be in the target of a differential, and induction. This concludes the proof.
\end{proof}

\begin{figure}[ht]
    \centering
    \includegraphics[width=0.8\linewidth]{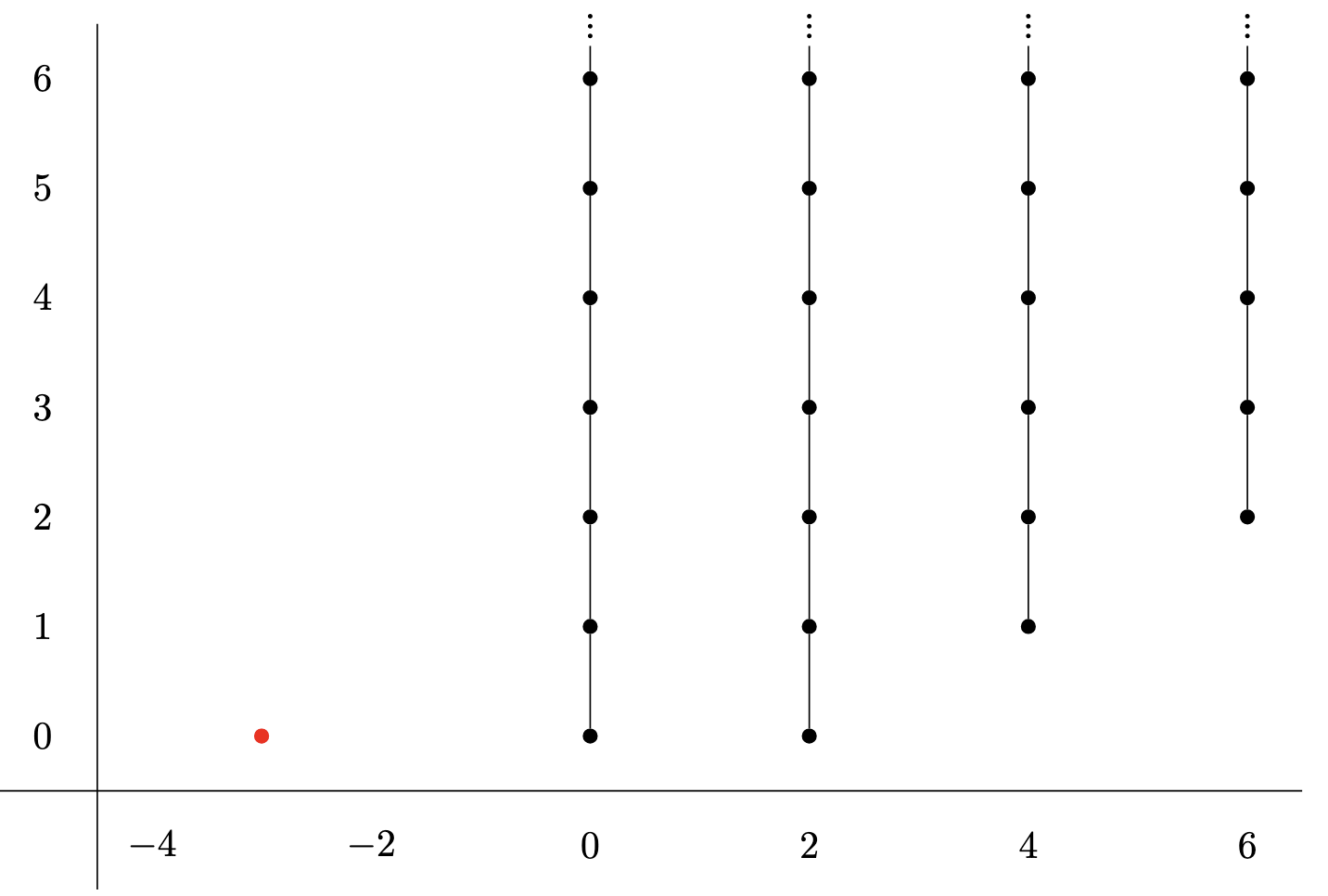}
    \caption{Charts for $\text{Ext}_{\mathcal{E}(1)^\vee_2}^{s,f,w}(L^{\mathbb{C}}_2(1), L^{\mathbb{C}}_2(2))$. The red class in bidegree $(-3,0)$ constitutes the torsion $B$ summand. }
    \label{fig:Ext_L1_L2_C_p2}
\end{figure}


    
To aid in our proofs, we recall the classical groups $\text{Ext}_{E(1)^\vee_p}^{s,f}(L_p^{cl}(k), L_p^{cl}(m)).$ When $k \le m$, 
\[
    \Ext_{E(1)^\vee_p}^{s,f}(L_p^{cl} (k), L_p^{cl} (m)) \cong \mF_{2}[v_{0}, v_{1}]\{x_{0}, x_{1}, \ldots x_{m-k}|\ v_{1}x_{i} = v_{0}x_{i+1} \},
\] 
where $|x_{i}| = (2i, 0)$. When $k \ge m$, 
\[
    \Ext_{E(1)^\vee_p}^{s,f}(L_p^{cl} (k), L_p^{cl} (m)) \cong \mF_{p}[v_{0}, v_{1}]\{x\} \oplus \mF_{p}[v_{0}, v_{1}]\left\{ y_{0}, y_{1}, \ldots, y_{k-m-1} \Bigg\vert \begin{array}{l}
         v_{1}y_{i} = v_{0}y_{i+1},\\  v_{0}y_{0} = 0,\\
         v_{1}y_{k-m-1} = 0 
    \end{array}\right\},
\]
where $|x| = (0, k-m)$ and $|y_{i}| = (-1 - 2(k-m-i), 0)$. These $\Ext$-terms can be visualized as in the charts given (which depict the $p=2$ case) in \cref{fig:clExtFpL3} (where $k \leq m$) and \cref{fig:clExtL3Fp} (where $k > m$). 

\begin{figure}
\centering
\begin{minipage}{.5\textwidth}
\centering
\includegraphics[width=2in, height=2in]{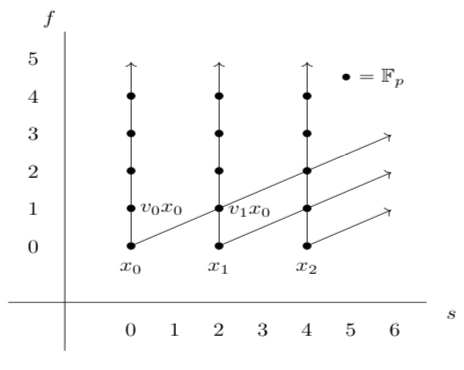}
 \captionof{figure}{\\$\Ext^{s,f}_{E(1)_*} (\mF_2, L_2^{cl}(3))$}
  \label{fig:clExtFpL3}
\end{minipage}%
\begin{minipage}{.5\textwidth}
\begin{center}
\includegraphics[width=2.75in,height=2in]{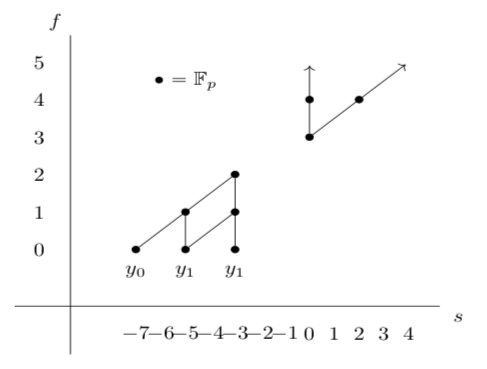}
 \end{center}
  \captionof{figure}{\\ $\Ext^{s,f}_{E(1)^\vee_2} (L_2^{cl}(3), \mF_2)$}
  \label{fig:clExtL3Fp}
\end{minipage}
\end{figure}

\begin{proposition}\label{prop:lightningExtCase1}
    Suppose $F = \mC$ or $F = \mF_q$ and $p$ is any prime where $\textup{char}(\mathbb{F}_q) \equiv 1 \,(p^2)$, or $F = \mR$ and $p$ is odd. Then \[\Ext_{\cE(1)_p^\vee}^{s,f,w}(L_p(k), L_p(m)) \cong \Ext_{E(1)_p^\vee}^{s,f}(L^{cl}_p(k), L^{cl}_p(m)) \otimes \mM_{p}^{F} .\]
\end{proposition}

\begin{proof}
    First, note that the case of $k \leq m$ is an immediate consequence of \cref{prop:ExtLightningflashKleqM} combined with the description of  $\text{Ext}_{\mathcal{E}(1)^\vee_p}^{s,f,w}(\mM_p^F, L_p(m))$ for $k \le m$ found in \cref{cor: Ext L_0 L_m}. Thus, we must only prove the case of $k>m$. 

Fix an $m \ge 0$, and start with the base case $k=m$. By \cref{prop:lightningExtCase1},
\[
\textup{Ext}_{\mathcal{E}(1)^\vee_p}^{s,f,w}(L_p(m), L_p(m)) \cong \textup{Ext}_{{E}(1)^\vee_p}^{s,f}(L^{cl}_p(m), L^{cl}_p(m)) \otimes \mM_{p}^{F}.
\]

Suppose that for all $k' < k$, $\Ext_{\cE(1)_p^\vee}^{s,f,w}(L_p^F (k'), L_p^F (m)) \cong \Ext_{E(1)_p^\vee}^{s,f}(L_p^{cl}(k'), L_p^{cl} (m)) \otimes \mM_{p}^{F}$. 
As in the previous proposition, the short exact sequence of lightning flash modules from \Cref{lightning flash ses} induces a long exact sequence of the form
\begin{align}
    \label{eq:LESLightningFlashExtKgeqM}
    \begin{split}
        \cdots \to \text{Ext}^{s,f,w}_{\mathcal{E}(1)^\vee_p}(L_p (k), L_p (m)) \to \text{Ext}^{s,f,w}_{\mathcal{E}(1)^\vee_p}(\Sigma^{2(p-1), p-1}L_p (k-1), L_p (m)) \\
        \xrightarrow{d}\text{Ext}^{s-1,f+1,w}_{\mathcal{E}(1)^\vee_p}((\mathcal{E}(1)_p//\mathcal{E}(0)_p)^\vee, L_p(m)) \to \cdots
    \end{split}
\end{align}

Recall that 
\begin{align*}
    {\Ext}^{s,f,w}_{\mathcal{E}(1)^\vee_p}((\mathcal{E}(1)//\mathcal{E}(0))_p^\vee, L_p(m)) & \cong {\Ext}^{s,f,w}_{\mathcal{E}(0)^\vee_p}(\Sigma^{-2(p-1), -(p-1)}\mM_p^F, L_p(m)) \\
    & \cong \mM_p^F[v_0]\{y\},
\end{align*} where $|y| = (-2(p-1),0, -(p-1))$. Consider the generator 
\[
\Sigma^{-2(p-1), -(p-1)}x \in \text{Ext}^{-2(p-1),0,0}_{\mathcal{E}(1)^\vee_p}(\Sigma^{2(p-1), p-1}L_p(k-1), L_p(m))
\]
in degree $(-2(p-1),0,0)$. Note that the differential $d$ preserves motivic weight, so either $d(x) = v_0^{m-k }y$ or $d(x) = 0$. Comparison with  $\Ext^{s,f,w}_{\cE(1)_p^\vee} \left(L_p (m+1), L_p (m+1)\right)$, which we computed in \cref{prop:ExtLightningflashKleqM}, implies that $d(x)$ must be nonzero. Specifically, suppose towards a contradiction that $d(\Sigma^{-2(p-1), -(p-1)}x) = 0$. Then we will have an infinite $v_{0}$-tower in stem $s=-(2p-1)$ of $\Ext^{s,f,w}_{\mathcal{E}(1)^\vee_p}\left(L_p^F(k), L_p^F(m)\right)$ for $k=m+1$ and $k=m$. Combining this with the long exact sequence used in the inductive computation of \cref{lem:lightning flash ext 1} would imply $\Ext^{s,f,w}_{\mathcal{E}(1)^\vee_p}\left(L_p(m+1), L_p(m+1)\right)$ must also have an infinite $v_{0}$-tower in odd stem. But we already have already proven in \cref{lem:lightning flash ext 1} that no such tower exists. Thus $d(\Sigma^{-2(p-1), -(p-1)}x) = v_0^{m-k}y$. Since  the map $d$ is $\Ext^{s,f,w}_{\mathcal{E}(1)^\vee_p}\left(\mM_p^F, \mM_p^F\right) \cong \mM_{p}^{F}[v_{0},v_{1}]$-linear, we get 
\[
d(cv_{0}^{n}\Sigma^{-2(p-1), -(p-1)}x_0 ) = cv_{0}^{m - k + n}y
\]
for all $c \in \mM_{p}^{F}$ and $n \ge 0$. It is clear from the stem degrees of the generators that no other differentials can occur. 

To finish the proof, relabel $\Sigma^{-2(p-1), -(p-1)}v_1x$ as $x$ and $\Sigma^{-2(p-1), -(p-1)}y$ as $y_{k-m+1}$. Finally, note that the extensions $v_{1}y_{i} = v_{0}y_{i+1}$ and $v_{1}\tau^{2}y_{i} = v_{0}\tau^{2}y_{i+1}$ must occur for exactly the same reasons as in the classical topological case (that is, by looking at the representatives of $v_1y_i$ and $v_0y_{i+1}$ in the chain complex computing $\Ext_{\cE(1)_p^\vee}^{*,1,*}\left(L_p(k),L_p(m)\right)$).
\end{proof}

The case $F = \mR$ and $p=2$ is addressed in \cite[Section 6]{LiPetTat25}, where charts are also given. We record the result here for the reader's convenience. 
\begin{proposition}[{\cite[Lemma 6.8]{LiPetTat25}}] \label{prop:RmotExtk>m}
    Suppose $F = \mR$ and $p = 2$. Then the group $\Ext_{\cE(1)_2^\vee}^{*,*,*}(L_2^\mathbb{R} (k), L_2^\mathbb{R} (m))$ when $k > m$ consists of: 
\begin{enumerate}
\item a triangle formation consisting of 
\[
\Ext^{*,*,*}_{\cE(1)_2^\vee}\left(\mM^{\mR}_{2}, \mM_2^{\mR} \right) \left\{y_{0}, \ldots, y_{m-k-1} \right\}
\]
with relations $v_{1}y_{i} = v_{0}y_{i+1},$ $v_{0}y_{0} = 0,$ and $v_{1}y_{m-k-1} = 0$ and 
\[
\Ext^{*,*,*}_{\cE(1)_2^\vee}\left(\mM_2^{\mR}, \mM_2^{\mR} \right) \left\{ \tau^{2}y_{0}, \ldots, \tau^{2}y_{m-k-1} \right\}
\]
with relations $v_{1} \tau^{2}y_{i} = v_{0}\tau^{2}y_{i+1},$ $v_{0}^{2}\tau^{2}y_{0} = 0,$ and $v_{1}\tau^2 y_{m-k-1} = 0$ where \newline $|y_{i}| = \big(-2(k-m - i)-1, 0, -(k-m - i)\big),$ 

     \item infinite $\rho$-towers generated in odd stem,

      \item a copy of $v_{1}\cdot\Ext^{*,*,*}_{\cE(1)_2^\vee}(\mM_2^\mR, \mM_2^{\mR})$, with generator denoted $x$ and $|x| = (0, k-m, 0),$

    \item $\rho$-pairs: \[\mF_2 [\rho, v_{1}] \left\{ v_1 b \, | \, v_{1}^{m-k-1}b = \rho x, \, \rho^{2}b = 0\right\}, \] where $|b| = (2(m-k)-1, \, 0, \, m-k-1)$. 
    \end{enumerate}
\end{proposition} 

The following proposition is proven in exactly the same way as \cref{prop:lightningExtCase1}.
\begin{proposition}\label{ExtKLeM:p=2Different}
    Suppose $F = \mF_q$ where $q \not\equiv 1 \, (p^2)$ and $char(\mathbb{F}_{q}) \neq p$. If $k > m$, then
     \begin{align*}
           \Ext_{\cE(1)_p^\vee}^{s,f,w}(L_p(k), L_p(m)) \cong \Ext_{\cE(1)_p^\vee}^{s,f,w} & (\mM^F_{p}, \mM^F_{p})\{x \} \\
           & \oplus \mM^F_{p}\left\{ y_{0}, y_{1}, \ldots, y_{k-m} \Bigg\vert \begin{array}{l}
         v_{1}y_{i} = v_{0}y_{i+1},\\  v_{0}y_{0} = 0,\\
         v_{1}y_{k-m} = 0,\\  z v_0 = 0
    \end{array}\right\} \oplus B,
   \end{align*}
    where $|x| = (0,k-m,0)$ $z=\rho$ if $p=2$, $z= \gamma$ if $p>2$, and where $B$ is as defined in \cref{prop:ExtLightningflashKleqM}
\end{proposition}

\begin{remark}
    \rm
    While this case does not allow a description as a tensor product of the classical with $\mM_{p}^{F}$, it is not so different: we just require either the relation $\rho v_0 = 0$ or $\gamma v_0 = 0$.
\end{remark}

We now have all the pieces we need to analyze the relative Adams spectral sequence. To expedite our computations, we note that at any prime and for $F= \mathbb{C}$, $\mathbb{R}$, or $\mathbb{F}_q$, the $\Ext$-group
\[
\text{Ext}^{s,f,w}_{\cE(1)^\vee_p}(L_p (k), L_p (m)),
\]
excluding the summand $B$, is generated by monomials of the form $cv_0^iv_1^jx_\ell$ when $k \le m$ with $c \in \mM_p^F$. Likewise, when $k \ge m$, the same $\Ext$-group excluding the summand $B$
is generated by monomials of the form $cv_0^iv_1^jx$ and $cv_0^iv_1^jy_\ell$. Thus, we can describe any monomial in 
\[
E_2^{s,f,w} = \textup{Ext}^{s,f,w}_{\mathcal{E}(1)^\vee_p}\left(H_{*,*}^{BPGL \langle 1 \rangle}\left(\Sigma^{2k(p-1), k(p-1)}BPGL \langle 1 \rangle^{\langle \nu_p(k!) \rangle}\right), H_{*,*}^{BPGL \langle 1 \rangle}C\right)
\]
as either $b \in B$, or $cv_0^iv_1^jz$, where $z$ is either $\Sigma^{2(m-k)(p-1), (m-k)(p-1)}x$, $\Sigma^{2(m-k)(p-1), (m-k)(p-1)}x_\ell$, or $\Sigma^{2(m-k)(p-1), (m-k)(p-1)}y_\ell$.


\begin{proposition}\label{prop:AdamsDiffernentialsAllCases}
    All differentials in the Adams spectral sequence \begin{align*}
    E_2^{s,f,w} = \textup{Ext}^{s,f,w}_{\mathcal{E}(1)^\vee_p}\left(H_{*,*}^{BPGL \langle 1 \rangle}\left(\Sigma^{2k(p-1), k(p-1)}BPGL \langle 1 \rangle^{\langle \nu_p(k!) \rangle}\right), H_{*,*}^{BPGL \langle 1 \rangle}C\right) \\
    \implies [\Sigma^{2k(p-1), k(p-1)}BPGL \langle 1 \rangle^{\langle \nu_p(k!)\rangle}, C]^{BPGL \langle 1 \rangle}_{s,w}.
    \end{align*}  are determined by those in the Adams spectral sequence for $BPGL\langle 1 \rangle$. Specifically, if $z$ is any of the generators described immediately above, then \begin{align*}
    d_r\left(cv_0^iv_1^jz\right) = d_r(c)v_0^iv_1^jz. 
    \end{align*} for all $c\in \mM_p^F$.

\end{proposition}
\begin{proof}

We will first show that the summand $B$ is neither the source nor the target of any differential. Note that the $BPGL\langle 1 \rangle$-module structures of $BPGL\langle 1 \rangle^{\langle \nu_p(k!) \rangle}$ and $C$ induce a pairing of relative Adams spectral sequences
\[ E_r\left(BPGL\langle 1 \rangle, BPGL\langle 1 \rangle \right) \otimes E_r\left(BPGL\langle 1 \rangle^{\langle \nu_p(k!) \rangle}, C\right) \to E_r\left(BPGL\langle 1 \rangle^{\langle \nu_p(k!) \rangle}, C\right) .\]

Recall from \cref{prop:ExtLightningflashKleqM} that the summand $B$ is concentrated in filtration $f=0$, and is generated over $\Ext^{s,f,w}_{\cE(1)^\vee_p}(\mM_p^F, \mM_p^F)$ by generators in odd stem $s$. Thus $B$ is generated over $\mathbb{M}_p$ by monomials of the form $cb$, where $c \in \mathbb{M}_p^F$. Note that 
\[
c \in E_2(BPGL\langle 1 \rangle, BPGL\langle 1 \rangle)\cong \Ext^{s,f,w}_{\cE(1)^\vee_p}(\mM_p^F, \mM_p^F),
\]
so we can use the pairing of spectral sequences to say \[d_r\left(cb\right) = d_r'\left(c\right)b \pm cd_r\left(b\right) ,\] 
where $d_r'$ is the differential in $E_r(BPGL\langle 1 \rangle, BPGL \langle 1 \rangle)$. Note that $d_r'$ is always either trivial or $v_0$ or $v_1$-divisible, so $d_r'(c)b = 0$. This leaves us with $d_r(bc) = cd_r(b)$. We will now show $d_r(b) = 0$ for all generators $b$ of $B$. 

Recall that the Adams differential $d_{r}$ has degree $(s,f,w) = (-1,r,0)$. Since the summand $B$ is concentrated in filtration $f=0$, no class in $B$ can be the target of any differential. Furthermore, the summand $B$ is entirely $v_1$-torsion, so there can be no differential from $B$ to a $v_1$-torsion class in $E_r^{s,f,w}$. Finally, note that all $v_1$-torsion classes in filtration $f>0$ are in odd stem, and the Adams differential decreases stem by $-1$. Thus any differential exiting a generator of $B$ must indeed be zero, and so $B$ must consist solely of permanent cycles.

Now we are ready to use the pairing of Adams spectral sequence to analyze the rest of the spectral sequence. Note that $cv_0^iv_1^j \in E_2(BPGL\langle 1 \rangle, BPGL\langle 1 \rangle)$. The pairing of the Adams spectral sequences implies 
\begin{align*}
d_r\left(cv_{0}^{i}v_{1}^{j}\Sigma^{2(m-k)(p-1), (m-k)(p-1)}x\right) \\
= d_r'\left(cv_0^{i}v_{1}^{j}\right)\Sigma^{2(m-k)(p-1), (m-k)(p-1)}x \pm cv_0^{i}v_{1}^{j}d_r\left(\Sigma^{2(m-k)(p-1), (m-k)(p-1)}x\right),
\end{align*} 
where $d_r'$ is the differential in $E_r(BPGL\langle 1 \rangle, BPGL \langle 1 \rangle)$. Furthermore, recall that 
\[
d_r'\left(cv_0^iv_1^j\right)  = d_r'\left(c\right)v_0^iv_1^j.
\]
So if $d_r\left(\Sigma^{2(m-k)(p-1), (m-k)(p-1)}x\right) = 0$, then
\[d_r\left(cv_{0}^{i}v_{1}^{j}\Sigma^{2(m-k)(p-1), (m-k)(p-1)}x\right) = d_r'\left(c\right)v_0^{i}v_{1}^{j}\Sigma^{2(m-k)(p-1), (m-k)(p-1)}x,\]
and the same claim will hold if we replace $x$ by $x_\ell$ or $y_\ell$.
We will now finish the proof by showing that for any generator $\Sigma^{2(m-k)(p-1), (m-k)(p-1)}x$, $\Sigma^{2(m-k)(p-1), (m-k)(p-1)}x_\ell$, or $\Sigma^{2(m-k)(p-1), (m-k)(p-1)}y_\ell$, the differential $d_r$ indeed acts trivially. This is almost exactly the same as the proof that the classical analogue of this spectral sequence collapses, for example \cite[Section~2]{LiPetTat25} presents this computation at $p=2$. 

Recall that the elements $\Sigma^{2(m-k)(p-1), (m-k)(p-1)}x$ and $\Sigma^{2(m-k)(p-1), (m-k)(p-1)}x_\ell$ generate the $v_1$-torsion free component over $\mM_p^F[v_0, v_1]$, while the elements $\Sigma^{2(m-k)(p-1), (m-k)(p-1)}y_\ell$ generate the $v_1$-torsion component over $\mM_p^F[v_0, v_1]$. We will start by showing that $d_r(\Sigma^{2(m-k)(p-1), (m-k)(p-1)}y_\ell) = 0$. First, recall that no nontrivial differentials can go from $v_{1}$-torsion classes to $v_{1}$-torsion free classes. So any potential target will be a sum of $v_1$-torsion classes, of the form $\Sigma^{2(m'-k)(p-1), (m'-k)(p-1)}cv_0^iv_1^jy_\ell'$, where $c \in \mM_p^F$ and $\ell' \in \mathbb{Z}$. Next, observe that all $v_{1}$-torsion generators $\Sigma^{2(m'-k)(p-1), (m'-k)(p-1)}y_\ell$ have degree 
\[|\Sigma^{2(m'-k)(p-1), (m'-k)(p-1)}y_\ell| = (2(p-1)n-1,0,n)\] 
for some $n \in \mathbb{Z}.$ Recall that the Adams differential $d_{r}$ has degree $(s,f,w) = (-1,r,0)$, so any potential target must have degree $(2(p-1)n-2, r, (p-1)n)$. Consider a potential target $\Sigma^{2(m'-k)(p-1), (m'-k)(p-1)}cv_0^iv_1^jy_l'$. Observe that this potential target will have degree 
\[|\Sigma^{2(m'-k)(p-1), (m'-k)(p-1)}cv_0^iv_1^jy_l| = (2n'-1, r, n') + |c|\] 
for some $n' \in \mathbb{Z}$. So we need $|c| = (2a-1,0,a)$ for some $a \in \mathbb{Z}$. No such monomial exists in any $\mathbb{M}_p^F$, so indeed $d_r(\Sigma^{2(m-k)(p-1), (m-k)(p-1)}y_\ell) = 0$.

Now we will show $d_r(\Sigma^{2(m-k)(p-1), (m-k)(p-1)}x) = 0$. First, we will show that no $v_1$-torsion-free class is a potential target, for degree reasons. This is exactly the same argument as we just gave above: $\Sigma^{2(m-k)(p-1), (m-k)(p-1)} x$ has degree $|x| = (2n,0,n)$ for some $n \in \mathbb{Z}$. So any potential target must have degree $(2n-1,r,n)$. But a potential target $\Sigma^{2(m-k)(p-1), (m-k)(p-1)} c v_0^m v_1^nx$ or $\Sigma^{2(m-k)(p-1), (m-k)(p-1)} c v_0^m v_1^nx_\ell$ has degree $(2n', r, n') + |c|$ for some $n' \in \mathbb{Z}$. So we need $|c| = (2a-1,0,a)$ for some $a \in \mathbb{Z}$. No such monomial exists in any $\mathbb{M}_p^F$. So there are no non-trivial differentials from $\Sigma^{2(m-k)(p-1), (m-k)(p-1)}x$ to any $v_1$-torsion free class. 

By the exact same argument, there are no non-trivial differentials from $\Sigma^{2(m-k)(p-1), (m-k)(p-1)}x_\ell$ to any $v_1$-torsion free class. 

However, we still need to show that the $v_1$-torsion classes cannot be potential targets of a differential exiting a generator $\Sigma^{2(m-k)(p-1), (m-k)(p-1)} x$ or $\Sigma^{2(m-k)(p-1), (m-k)(p-1)} x_\ell$. 

The $v_1$-torsion classes in filtration $f > 0$ are of the form $\Sigma^{2(m-k)(p-1), (m-k)(p-1)}cv_0^iv_1^jy_{\ell}$, which only occur in stem $s \le -1 - 2(p-1)$. Since 
\[
\bigoplus\limits_{0 \le k \le m} \Sigma^{2(m-k)(p-1), (m-k)(p-1)}\Ext^{s,f,w}_{\mathcal{E}(1)_p^\vee}( L_p (\nu_{p}(k!) ), L_p (\nu_{p}(m!) ) )
\]
is contained entirely in stem $s \ge 0$, it only remains to show that there are no differentials originating in $v_1$-torsion free generators of
\[
\bigoplus\limits_{k > m}\Sigma^{2(m-k)(p-1), (m-k)(p-1)}\Ext^{s,f,w}_{\mathcal{E}(1)_p^\vee}( L_p(\nu_{p}(k!) ), L_p(\nu_{p}(m!) ) )
,\]
which are of the form $\Sigma^{2(m-k)(p-1),(m-k)(p-1)}x$.
Note that these generators are all in stem $-2n$ for some $n \in \mathbb{N}$, and so any potential target would be in stem $-2n-1$. Fix $n \ge 0$. We will put bounds on the Adams filtration of the classes in stem $-2n$, and bounds on the Adams filtration of classes in stem $-2n-1$. We will use these bounds to conclude that no nontrivial differentials are possible. 

The only $v_{1}$-torsion free generator of $\Sigma^{2(m-k)(p-1), (m-k)(p-1)} \Ext^{s,f,w}_{\mathcal{E}(1)_p^\vee}\left(L_p(\nu_{p}(k!)), L_p(\nu_{2}(m!))\right)$ is $\Sigma^{2(m-k)(p-1), (m-k)(p-1)}x$, in Adams filtration $\nu_p(k!)-\nu_p(m!).$ It follows that in \[
\bigoplus\limits_{k > m}\Sigma^{2(m-k)(p-1), (m-k)(p-1)}\Ext^{s,f,w}_{\mathcal{E}(1)_p^\vee}( L_p^F (\nu_{p}(k!) ), L_p^F (\nu_{p}(m!) ) )
,\]
the only generator in stem $-2n$ is $\Sigma^{-2n,-n}x$, which has Adams filtration \[\nu_{p}(k!) - \nu_{p}((k-n)!).\] Now we will show that any $v_1$-torsion class in stem $-2n-1$ will have lower Adams filtration, meaning that no differential from $\Sigma^{-2n,-n}x$ to such a class are possible. 

Observe that in $ \Ext^{s,f,w}_{\mathcal{E}(1)_p^\vee}\left(L_p(\nu_{p}(k)), L_p(\nu_{p}(m'!))\right)$, the highest filtration of any $v_{1}$-torsion class in stem $-1-2(p-1)(1+i)$ is $\nu_{p}(k!)-\nu_{p}(m'!)-i-1$. So in \[\Sigma^{2(m'-k)(p-1), (m'-k)(p-1)}\Ext^{s,f,w}_{\cE(1)_p^\vee} \left(L_p (\nu_{p}(k!)), L_p (\nu_{p}(m'!))\right),\] 
the highest filtration of any $v_{1}$-torsion class in stem $-1-2(p-1)(1+i+k-m)$ is $\nu_{p}(k'!)-\nu_{p}(m'!)-i-1$. Thus, at any stem $-2n-1$, the highest filtration class in stem $-2n-1$ in 
\[\bigoplus\limits_{m=0}^{\infty}\Sigma^{2(m'-k)(p-1), (m'-k)(p-1)}\Ext^{s,f,w}_{\cE(1)_p^\vee} \left(L_p (\nu_{p}(k!)), L_p (\nu_{p}(m'!))\right)\]
 will be in filtration $\nu_{p}(k!) - \nu_{p}((k-n+1)!) - 1.$
So no $v_1$-torsion classes can be the target of a differential exiting $\Sigma^{-2n,-n}x$, and indeed $d_r(\Sigma^{2(m-k)(p-1), (m-k)(p-1)}x) = 0$ and $d_r(\Sigma^{2(m-k)(p-1), (m-k)(p-1)}y_i\ell) = 0$ for all $\ell$. 
\end{proof}

We may now prove our main theorem.

\begin{theorem}
\label{thm:BPGL1SplittingAlpha}
    For any prime $p$ and for any $F \in \{\mathbb{C},\mathbb{R}, \mathbb{F}_q\}$ where $\textup{char}(\mathbb{F}_q) \neq p$, there is a splitting of $p$-complete motivic spectra:
    \[BPGL \langle 1 \rangle \wedge BPGL \langle 1 \rangle \simeq \bigvee_{k \geq 0}\Sigma^{2k(p-1), k(p-1)}BPGL \langle 1 \rangle^{\langle\nu_p(k!)\rangle} \vee V,\]
    where $BPGL \langle 1 \rangle^{\langle \nu_p(k!) \rangle}$ is the $\nu_p(k!)^{th}$ Adams cover of $BPGL \langle 1 \rangle$ and $V$ is a wedge of suspensions of $H$.
\end{theorem}

\begin{proof}
Recall that by construction, $|\theta_k| = (0,0,0)$ is the class $x_0 \in \Ext^{0,0,0}_{\mathcal{E}(1)_p^{\vee}}(L_p(k), L_p(k))$. By \cref{prop:AdamsDiffernentialsAllCases}, $d_r(x_0) = 0$, and so $\theta_k$ survives the spectral sequence. Thus we have constructed maps 
\[ \widetilde{\theta_k}: BPGL\langle 1 \rangle ^{\langle\nu_p(k!)\rangle} \to C, \]
yielding a splitting
\[ \bigvee\limits_{k\geq 0}^{}\Sigma^{2k(p-1), k(p-1)}BPGL\langle 1 \rangle ^{\langle\nu_p(k!)\rangle} \vee V \xrightarrow{\simeq} C \vee V \xrightarrow{\simeq} BPGL\langle 1 \rangle \wedge BPGL\langle 1 \rangle.\]

\end{proof}

\begin{remark}
    \rm For $p=2$, this gives a splitting of the cooperations algebra for effective algebraic $K$-theory, recalling that a model for $BPGL\langle 1 \rangle$ is $kgl$:
    \[kgl \wedge kgl \simeq \bigvee_{k \geq 0}\Sigma^{2k, k}kgl^{\langle\nu_2(k!)\rangle} \vee V.\]
    This extends the splitting for $F=\mathbb{R}$ shown in \cite{LiPetTat25}.
\end{remark}

\begin{remark}
    \rm Classically, there is a spectrum-level decomposition of the $BP\langle 2 \rangle$-cooperations algebra \cite{culver-bp2coop,CulverOddp20}
    \[BP\langle 2 \rangle\wedge BP \langle 2 \rangle \simeq  C \vee V,\] 
    where $C$ is $v_2$-torsion-free and $V$ is a sum of suspensions of $H\mF_p$. One approach to understanding $BPGL \langle 2 \rangle \wedge BPGL \langle 2 \rangle$ is to lift Culver's result to the motivic context. There is certainly a splitting $BPGL\langle 2 \rangle \wedge BPGL\langle 2 \rangle \simeq C \vee V$ where $V$ is a sum of suspensions of the motivic mod-$p$ Eilenberg-Maclane spectrum. The summand $C$ should also consist of mostly $v_2$-torsion, but the presence of $\rho$, $u$, or $\gamma$-towers slightly complicates things. Culver also gives an inductive algorithm for computing the $E_{2}$-page of the Adams spectral sequence 
    \[\text{Ext}^{s,f}_{E(2)^\vee_p}(\mF_{p}, H_{*}(BP\langle 2 \rangle \wedge BP\langle 2 \rangle)) \Longrightarrow \pi_s(BP\langle 2 \rangle \wedge BP\langle 2 \rangle),\]
    and shows that this spectral sequence collapses at the $E_{2}$-page. We expect that this result readily adapts to the motivic setting as well, with the only modification being that the spectral sequence may not collapse, rather all differentials are determined by those in the Adams spectral sequence for $BPGL\langle 2 \rangle$, just as in the height one case.
\end{remark}

\subsection{Related Results}
One immediate consequence of our computations in the previous section is that these Adams covers are uniquely determined (up to homotopy) by their homology.
\begin{theorem}
\label{prop:UniquenessOfAdamsCoversBPGL1}
    If $X$ is a $BPGL \langle 1 \rangle$-module spectrum such that $H^{BPGL \langle 1 \rangle}_{*,*}(X) \cong L_p(k),$ then $X \simeq BPGL \langle 1 \rangle ^{\langle n \rangle}.$
\end{theorem}

\begin{proof}
    Consider the $BPGL \langle 1 \rangle$-relative Adams spectral sequence
    \[E_2^{s,f,w} = \text{Ext}^{s,f,w}_{\mathcal{E}(1)^\vee_p}(H_{*,*}^{BPGL \langle 1 \rangle}X, H_{*,*}^{BPGL \langle 1 \rangle}BPGL \langle 1 \rangle^{\langle n \rangle}) \implies [X, BPGL \langle 1 \rangle^{\langle n \rangle}]^{BPGL \langle 1 \rangle}_{(s,w)}.\] The $E_2$-page can be rewritten as $\Ext_{\cE(1)_p^\vee}^{s,f,w}(L_p(k), L_p(k))$, and we can view the identity map $L_p(k) \to L_p(k)$ as a class in $E_2^{0,0,0}$. This is a summand of the spectral sequence we analyzed in \cref{prop:AdamsDiffernentialsAllCases} and \cref{thm:BPGL1SplittingAlpha}, and by the analysis we gave there the class representing the identity map must survive the spectral sequence. So indeed this map lifts to an equivalence of $BP\langle 1 \rangle$-module spectra $X \to BPGL\langle 1 \rangle^{\langle\nu_p(k)\rangle}$.
\end{proof}

Our computations also yield a description of the $BPGL\langle 1  \rangle$ operations algebra. 

\begin{theorem}\label{thm: operations}
    The $BPGL\langle 1 \rangle$-operations algebra $[BPGL\langle 1 \rangle, BPGL\langle 1 \rangle]$ splits as
    \[ [BPGL\langle 1 \rangle, BPGL\langle 1 \rangle] \simeq \bigoplus\limits_{k=0}^{\infty} [\Sigma^{2k(p-1), k(p-1)}BPGL\langle 1 \rangle^{\langle \nu_p(k!) \rangle}, BPGL\langle 1 \rangle]^{BPGL\langle 1 \rangle}.\]

    Furthermore, the Adams spectral sequence computing $[BPGL\langle 1 \rangle, BPGL\langle 1 \rangle]$ has signature 
    \[E_2^{s,f,w} \cong \bigoplus\limits_{k=0}^{\infty}\Ext_{\cE(1)^\vee_p}(\Sigma^{2k(p-1), k(p-1)}L_p(k), \mM_p^F)\Longrightarrow [BPGL\langle 1 \rangle, BPGL\langle 1 \rangle] .\]
    All differentials in this Adams spectral sequence are determined by those in the Adams spectral sequence computing $\pi_{*,*}BPGL\langle 1 \rangle$, and \cref{prop:lightningExtCase1}, \cref{prop:RmotExtk>m}, and \cref{ExtKLeM:p=2Different} compute the relevant $\Ext$ groups.
\end{theorem}

\begin{proof}
    Observe that \[ [BPGL\langle 1 \rangle, BPGL\langle 1 \rangle] \simeq [BPGL\langle 1 \rangle \wedge BPGL\langle 1 \rangle, BPGL\langle 1 \rangle]^{BPGL\langle 1 \rangle},\]
    so applying \cref{thm:BPGL1SplittingAlpha} yields the splitting. For each $k$, the relative Adams spectral sequence \[E_2^{s,f,w} \cong \Ext_{\cE(1)_p^\vee}^{s,f,w}(\Sigma^{2k(p-1), k(p-1)}L_p(k), \mM_p^F)\Longrightarrow [\Sigma^{2k(p-1), k(p-1)}BPGL\langle 1 \rangle^{\langle \nu_p(k!) \rangle}, BPGL\langle 1 \rangle] \] is just the $m=0$ summand of the spectral sequence discussed in \cref{prop:AdamsDiffernentialsAllCases}, and so by the same reasoning as in that proof, all differentials are determined by those in the Adams spectral sequence computing $\pi_{**}BPGL\langle 1 \rangle$.
\end{proof}

\subsection{Application to the $BPGL \langle 1 \rangle$-motivic Adams spectral sequence} 
In this final section, we apply our results to describe the $E_1$-page of the $BPGL \langle 1 \rangle$-based motivic Adams spectral sequence. Recall from \Cref{subsec:adams spectral sequence} that the $BPGL \langle 1 \rangle$-based motivic Adams spectral sequence takes the form
\[E_1=\pi^F_{s+f, w}(BPGL \langle n \rangle \wedge \overline{BPGL \langle n \rangle}^{\wedge f}) \implies \pi^F_{s,w}\mathbb{S}_{(p)}.\]
We can compute each filtration $f=n$-piece of the $E_1$-page by a motivic Adams spectral sequence taking the form
\[E_2 = \text{Ext}_{\mathcal{A}^\vee_p}^{s,f,w}\left(\mathbb{M}_p, H_{*,*}(BPGL \langle 1 \rangle \wedge \overline{BPGL \langle 1 \rangle}^{\wedge n})\right) \implies \pi_{s,w}^{F}(BPGL \langle 1 \rangle \wedge \overline{BPGL \langle 1 \rangle} ^{\wedge n}).\]
We begin by describing the $E_2$-page of this spectral sequence. First, we make some more general observations.
\begin{lemma}
\label{lemma:BpglnBarHomology}
    There is an isomorphism of $\mathcal{E}(n)^\vee$-comodules
    \[H_{*,*}(\overline{BPGL \langle n \rangle}) \cong \bigoplus_{k\geq 1}\Sigma^{2k(p-1), k(p-1)}B_{n-1}(k).\]
\end{lemma}

\begin{proof}
    The cofiber sequence $\mathbb{S} \to BPGL \langle n \rangle \to \overline{BPGL \langle n \rangle}$ gives a long exact sequence in homology. Since $H_{*,*}(BPGL \langle n \rangle) \cong \bigoplus_{k \geq 0}\Sigma^{2k(p-1), k(p-1)}B_{n-1}(k)$ by \Cref{prop:BPGLnHomology} and $H_{*,*}(\mathbb{S}) = \mathbb{M}_p = B_{n-1}(0)$, the result follows.
\end{proof}

\begin{lemma}
\label{lemma:BigKunneth}
    There is an isomorphism of $\mathcal{A}^\vee_p$-comodules
    \[H_{*,*}(BPGL \langle n \rangle \wedge \overline{BPGL \langle n \rangle}^{\wedge f}) \cong H_{*,*}(BPGL \langle n \rangle) \otimes H_{*,*}(\overline{BPGL\langle n \rangle})^{\otimes f}.\]
\end{lemma}

\begin{proof}
    This follows by a similar analysis as was performed in \Cref{prop:BPGLnKunneth}. We induct on $f$. For $f=1$, consider the K\"unneth spectral sequence
    \[E_2 = \text{Tor}^{\mathbb{M}_p}(H_{*,*}(BPGL \langle n \rangle), H_{*,*}(\overline{BPGL \langle n \rangle})) \implies H_{*,*}(BPGL \langle n \rangle \wedge BPGL \langle n \rangle).\]
    Since $H_{*,*}(BPGL \langle n \rangle)$ is free over $\mathbb{M}_p$, the spectral sequence collapses and the result follow. Suppose now this is true for all $f$ up to $k-1$, and consider the K\"unneth spectral sequence
    \begin{align*}
    \begin{split}
    E_2 = \text{Tor}^{\mathbb{M}_p}(H_{*,*}(BPGL \langle n \rangle \wedge \overline{BPGL \langle n \rangle}^{\wedge {k-1}}), H_{*,*}(\overline{BPGL \langle n \rangle})) \\
    \implies H_{*,*}(BPGL \langle n \rangle \wedge \overline{BPGL \langle n \rangle}^{\wedge k}).
    \end{split}
    \end{align*}
    By induction, we have a K\"unneth isomorphism on the left hand factor in Tor. By \Cref{lemma:BpglnBarHomology}, $H_{*,*}(\overline{BPGL\langle n \rangle})$ is free over $\mathbb{M}_p$, so the spectral sequence collapses. The result follows.
\end{proof}

We now specialize to $BPGL\langle 1 \rangle$.
\begin{proposition}
\label{prop:E2BPGL1Resn-line}
    The $E_2$-page of the $\textup{\textbf{mASS}}_p^F(BPGL \langle 1 \rangle \wedge \overline{BPGL \langle 1 \rangle}^{\wedge f})$ may be rewritten as
    \[E_2^{s,f,w} = \bigoplus_{I \in \mathcal{I}_n}\Sigma^{2|I|(p-1), |I|(p-1)}\textup{Ext}^{s,f,w}_{\mathcal{E}(1)^\vee_p}(\mathbb{M}_p, L_p(\nu_p(I!))) \oplus W,\]
    where $\mathcal{I}_n = \{I=(k_1, \dots, k_n): k_j \geq 1 \text{ for all } 1 \leq j \leq n\}$, $|I| = k_1 + \cdots+ k_n$, and $L_p(\nu_p(I!)) = L_p(\nu_p(k_1!)) \otimes \cdots \otimes L_p(\nu_p(k_n!)),$ and $W$ is a sum of suspensions of $\mathbb{M}_p$ in filtration 0.    
\end{proposition}

\begin{proof}
    By the K\"unneth isomorphism of \Cref{lemma:BigKunneth} and the change-of-rings isomorphism, we may rewrite the $E_2$-page as
    \begin{align}
    \label{ss:mASS-E1-page-BPGL1-nline}
    \text{Ext}^{s,f,w}_{\mathcal{E}(1)^\vee_p}(\mathbb{M}_p, H_{*,*}(\overline{BPGL \langle 1 \rangle}) ^{\otimes n}).
    \end{align}
    \Cref{lemma:BpglnBarHomology} and \cref{prop:LightningFlashBGiso} supply us with isomorphisms of $\mathcal{E}(1)^\vee_p$-comodules 
    \[H_{*,*}(\overline{BPGL \langle 1 \rangle}) \cong\bigoplus_{k \geq 1}\Sigma^{2k(p-1), k(p-1)}B_0(k) \cong \bigoplus_{k \geq 1} \Sigma^{2k(p-1), k(p-1)}L_p(\nu_p(k!)) \oplus W_k,\]
    where $W_k$ is a sum of suspensions of $\mathcal{E}(1)^\vee_p$.
    Using this to replace the first factor of $H_{*,*}(\overline{BPGL \langle 1 \rangle})$ in \Cref{ss:mASS-E1-page-BPGL1-nline}, we can rewrite the $E_2$-page again as
    \[\bigoplus_{k_1 \geq 1}\Sigma^{2k_1(p-1), k_1(p-1)}\text{Ext}_{\mathcal{E}(1)^\vee_p}^{s,f,w}(\mathbb{M}_p, L_p(\nu_p(k_1!)) \otimes H_{*,*}(\overline{BPGL \langle 1 \rangle})^{\otimes n-1}) \oplus W_1,\]
    where $W_1$ is a sum of suspensions of $\mathbb{M}_p$ in filtration 0.
    By repeating this process, each time rewriting the tensor factors of $H_{*,*}(\overline{BPGL \langle 1 \rangle})$, we arrive at
    \[\bigoplus_{I \in \mathcal{I}_n}\Sigma^{2|I|(p-1), |I|(p-1)}\text{Ext}^{s,f,w}_{\mathcal{E}(1)^\vee_p}(\mathbb{M}_p, L_p(\nu_p(I!))) \oplus W,\]
    where $\mathcal{I}_n = \{I=(k_1, \dots, k_n): k_j \geq 1 \text{ for all } 1 \leq j \leq n\}$, $|I| = k_1 + \cdots +k_n$, and $L_p(\nu_p(I!)) = L_p(\nu_p(k_1!)) \otimes \cdots \otimes L_p(\nu_p(k_n!)),$ and $W$ is a sum of suspensions of $\mathbb{M}_p$ in filtration 0.
\end{proof}

The proof of \Cref{thm:BPGL1homotopyCooperations} then allows us to compute the homotopy of the $n$-line of the $BPGL \langle 1 \rangle$-motivic Adams spectral sequence.
\begin{cor}
\label{cor:n-line}
    Let $F \in \{\mathbb{C}, \mathbb{R}, \mathbb{F}_q\}$, where $\textup{char}(\mathbb{F}_q) \neq p$. Then there is an isomorphism for all primes $p$:
    \begin{align*}
    \begin{split}
    \pi_{*,*}^F(BPGL \langle 1 \rangle \wedge \overline{BPGL \langle 1 \rangle}^{\wedge n}) \cong \bigoplus_{I \in \mathcal{I}_n}((\pi_{*,*}^FBPGL \langle 0 )\rangle\{x_0, \dots, x_{m-1}\} \\
    \oplus (\pi_{*,*}^FBPGL \langle 1 \rangle)\{x_m\} \oplus W_I),
    \end{split}
    \end{align*}
    where $m = \nu_p(k_1!) + \cdots \nu_p(k_n!)$ for $I= (k_1, \dots, k_n)$, and $W_I$ is a sum of suspensions of $\mathbb{M}_p$.
\end{cor}

This determines the $E_1$-page of the $BPGL\langle 1 \rangle$-based motivic Adams spectral sequence as a module over $\pi_{**}BPGL\langle 1 \rangle$. To conclude, we observe that our spectrum-level splitting of the cooperations algebra lifts to describe the entire $n$-line.

\begin{proposition}
    There is an equivalence of $p$-complete spectra:
    \[BPGL \langle 1 \rangle \wedge \overline{BPGL \langle 1 \rangle}^{\wedge n} \simeq\bigvee_{I \in \mathcal{I}_n}\Sigma^{2|I|(p-1), |I|(p-1)}BPGL \langle 1 \rangle^{\langle m \rangle} \vee V,\]
    where $m = \nu_p(k_1!) + \dots +\nu_p(k_n!)$ for $I = (k_1, \dots, k_n)$, and $V$ is a wedge of suspensions of $H$.
\end{proposition}

\begin{proof}
    By the same proof as in \Cref{prop:BPGL1SplittingBeta}, we can produce a splitting
    \[BPGL \langle 1 \rangle \wedge \overline{BPGL \langle 1 \rangle}^{\wedge n} \simeq C \vee V,\]
    where $V$ is a wedge of suspensions of $H$ and 
    \[H_{*,*}^{BPGL \langle 1 \rangle}C \cong \bigoplus_{I \in \mathcal{I}_n}\Sigma^{2|I|(p-1), |I|(p-1)}L_p(m),\]
    where $m = \nu_p(k_1!) + \cdots +\nu_p(k_n!)$. By uniqueness of Adams covers shown in \Cref{prop:UniquenessOfAdamsCoversBPGL1}, it follows that $C \simeq BPGL \langle 1 \rangle^{\langle m \rangle}$.
\end{proof}

We expect this spectrum-level decomposition will be useful in determining differentials in the $BPGL \langle 1 \rangle$-motivic Adams spectral sequence. In particular, this decomposition is used in the analysis of the classical $ko$-and $BP\langle 1 \rangle$-Adams spectral sequences \cite{BBBCX, Gonzalez00}, and we plan to adapt those strategies to the motivic setting.

\printbibliography

\end{document}